\documentclass[11pt]{amsart}
\usepackage[inline]{enumitem}
\usepackage{hyperref,amssymb,verbatim,graphics,epsfig,psfrag,float}
\usepackage{caption}
\usepackage{subcaption}
\hypersetup{colorlinks}
\usepackage{color}
\definecolor{darkred}{rgb}{0.5,0,0}
\definecolor{darkgreen}{rgb}{0,0.5,0}
\definecolor{darkblue}{rgb}{0,0,0.5}
\hypersetup{ colorlinks,
linkcolor=darkblue,
filecolor=darkgreen,
urlcolor=darkred,
citecolor=darkblue }

\usepackage{amssymb,verbatim,graphicx}
\newtheorem{btheorem}{$\beta$-Theorem}[section]
\newtheorem{theorem}[btheorem]{Theorem}

\newtheorem{corollary}[btheorem]{Corollary}

\newtheorem{proposition}[btheorem]{Proposition}
\newtheorem{lemma}[btheorem]{Lemma}

\theoremstyle{definition}
\newtheorem{definition}[btheorem]{Definition}
\theoremstyle{remark}
\newtheorem{remark}[btheorem]{Remark}

\newcommand\M{\mathcal{M}}
\newcommand\B{\mathcal{B}}
\newcommand\E{\mathcal{E}}

\renewcommand{\L}{\mathcal{L}}

\newcommand{\J}{\mathcal{J}}

\newcommand{\F}{\mathcal{F}}

\newcommand{\R}{\mathbb{R}}

\newcommand{\C}{\mathbb{C}}

\newcommand{\Z}{\mathbb{Z}}

\newcommand{\ppt}{\frac{\partial}{\partial t}}

\renewcommand{\P}{\mathbb{P}}

\newcommand{\mP}{\mathcal{P}}

\newcommand{\on}{\operatorname}
\newcommand{\Aut}{ \on{Aut} }

\renewcommand{\ker}{ \on{ker}}

\newcommand{\im}{ \on{Im}}
\newcommand\Id{\on{Id}}

\newcommand{\hra}{\hookrightarrow}

\newcommand{\tensor}{\otimes}

\newcommand{\Bl}{ \on{Bl}}
\newcommand\bran[1]{ \lan {#1} \ran}

\newcommand{\loc}{{\on{loc}}}

\newcommand{\ol}{\overline}
\newcommand\ul{\underline}
\newcommand{\delbar}{\ol{\partial}}
\newcommand\bs{\backslash} 
\newcommand\ev{\on{ev}}

\newcommand\lam{\lambda}

\newcommand\Sig{\Sigma}
\newcommand\sig{\sigma}
\newcommand\eps{\varepsilon}
\newcommand\Om{\Omega}
\newcommand\om{\omega}

\newcommand{\lan}{\langle}
\newcommand{\ran}{\rangle}
\newcommand{\hh}{{\frac{1}{2}}}
\newcommand{\texthh}{{\tfrac{1}{2}}}

\newcommand\Mod[1]{\lVert #1 \rVert}
\newcommand\qu{/\kern-.7ex/} 
\newcommand\mO{\mathcal{O}}

\newcommand\Op{\mathcal{O}p}
\newcommand\bS{\mathbb{S}}
\newcommand\bL{\mathbb{L}}
\newcommand\bD{\mathbb{D}}
\newcommand\bT{\mathbb{T}}
\newcommand\bt{{\bf t}}

\newcommand\txi{\tilde{\xi}}

\newcommand{\pre}{{pre}}
\newcommand{\pa}{ \partial_a}
\newcommand{\Lutt}{\on{Lutt}}
\newcommand{\Symp}{\on{Symp}}
\newcommand{\Diff}{\on{Diff}}
\newcommand{\Map}{\on{Map}}

\newcommand{\ssum}{{\textstyle \sum}}
\newcommand{\wind}{\on{wind}}
\newcommand{\level}{\on{level}}
\newcommand{\can}{{can}}
\newcommand{\image}{\on{image}}
\newcommand{\cyl}{{\on{cyl}}}
\newcommand{\tw}{{\on{tw}}}

\newcommand{\oSig}{\Sig^\circ}
\newcommand\seps{\epsilon}
\newcommand\cpt{{\on{cpt}}}
\newcommand\inj{{\on{inj}}}
\newcommand{\Ree}{\gamma}

\newcommand{\Hof}{{\on{Hof}}}
\newcommand{\hor}{{\on{hor}}}

\begin{document}
\title{Symplectic foliated fillings of sphere cotangent bundles}
\author{Francisco Presas}
\email{fpresas@icmat.es}
\address{Instituto de Ciencias Matem\'aticas
  CSIC-UAM-UC3M-UCM, C. Nicol\'as Cabrera, 13-15, 28049 Madrid,
  Spain.}  

\author{Sushmita Venugopalan}
\email{sushmita@imsc.res.in}
\address{Institute of Mathematical
  Sciences, CIT Campus, Taramani, Chennai 600113, India.}

\begin{abstract} We classify symplectically foliated fillings of
  certain foliated manifolds with a contact structure on the
  leaves. We show that for the foliated sphere cotangent bundle of the
  Reeb foliation in $\bS^3$, the corresponding foliated disk cotangent bundle
  is the unique strong symplectic foliated filling up to blowups and
  symplectic deformation equivalence.  En route to the proof, we study
  another foliated manifold, namely the product of a circle and an
  annulus with an almost horizontal foliation.  In this case, the
  foliated filling of the foliated sphere cotangent bundle is not
  unique.  We show that any such filling is a foliated Lefschetz
  fibration, and is determined up to symplectic deformation
  equivalence, by combinatorial invariants arising from the singular
  locus of the Lefschetz fibration.
\end{abstract}

\maketitle

\tableofcontents

\section{Introduction}
Sphere cotangent bundles have a canonical contact structure, and the
corresponding disk cotangent bundles are natural symplectic fillings.
For a manifold $M$ with a smooth foliation $\F$, the cotangent bundle
$T^*\F$ of the leaves of $\F$, called the {\em foliated cotangent
  bundle} is itself a foliated manifold with a symplectic form on its
leaves.  The {\em foliated sphere (or unit) cotangent bundle}
\[\bS(T^*\F)  \subset T^*\F\]
is the subset of unit co-vectors of the foliated cotangent bundle, and
is a foliated manifold with a contact structure on its leaves.
Analogous to the unfoliated case, the foliated disk cotangent bundle
is a natural symplectic filling of the foliated sphere cotangent
bundle on every leaf.  We show that for the $3$-sphere with Reeb
foliation, the foliated disk cotangent bundle is the unique minimal
filling of the foliated unit cotangent bundle up to symplectic
deformation equivalence.

 The Reeb foliation $\F_{Reeb}$ on the sphere $\bS^3$, described in
 Figure \ref{fig:Reeb}, contains the
 torus $\bT^2$ as a leaf.
 The torus divides $\bS^3$ into two solid
tori, each of whose interior is foliated by an $\bS^1$-family of
non-compact leaves.
One of the results of the paper is that if a
foliated contact manifold with three dimensional leaves has a compact
leaf, then in a foliated filling the compact leaf has a compact
filling. Using this result, any filling $W_{Reeb}$ of the foliated
unit cotangent bundle of $(\bS^3,\F_{Reeb})$ can be split along the
compact leaf to yield $W_{Reeb}^+$, $W_{Reeb}^-$. The two pieces are
fillings of the foliated unit cotangent bundle of the solid torus with
Reeb foliation. By performing a surgery, we reduce either of the pieces
$W_{Reeb}^\pm$ to a filling of the foliated sphere cotangent bundle of
an almost horizontal foliation.

An almost horizontal foliation $\F_{ah}$ on a two-dimensional annulus
$A^2$ consists of two compact leaves which are the boundary of the
annulus, and an $\bS^1$-family of non-compact leaves with one end
asymptotic to each of the leaves, see Figure \ref{fig:test}.  We
analyse the fillings of the foliated sphere cotangent bundle of the
product $\bS^1 \times (A^2,\F_{ah})$.  Using holomorphic curve
techniques, we show that any filling is a a Lefschetz fibration on a
foliated base manifold. The fillings are classified up to symplectic
deformation equivalence by the homotopy class of the locus of singular
values of the Lefschetz fibration, which is a combinatorial invariant
of the filling.

The various fillings in the almost horizontal case give a new class of
examples of foliations with a strong symplectic form.  In
\cite{v:nov}, the second author gives an example of a five-manifold
that violates a Novikov property that was expected to be true for
strong symplectic foliations, and the example closely resembles the
fillings in the almost horizontal case in this paper.

\subsection{Motivating the problem} 
 The problem of classifying symplectic fillings is related to the
 classical question \cite{wall1960determination} of understanding
 cobordism classes of closed differentiable manifolds.  A {\em symplectic
 filling} $W$ is an oriented cobordism from an empty set to a contact
 manifold $M$, that is, $M$ is the convex boundary of the symplectic
 filling $W$. We assume without mentioning that fillings are `strong'
 in the sense that a neighborhood of the boundary $\partial W$ is
 symplectomorphic to $((-\eps,0] \times M, d(e^t \alpha))$ where
 $\alpha$ is a contact one-form on $M$.

 Several contact $3$-folds have been shown to possess unique minimal
 symplectic fillings. Examples include $\bS^3$ \cite{Eliash:holdiscs},
 $\R\P^3$ (which is the unit cotangent bundle of $\bS^2$)
 \cite{Hind:rp3}, and Lens spaces $L(p,1)$ \cite{Hind:lens}, except
 for the special case of $L(4,1)$ which has two minimal symplectic
 fillings. The unit cotangent bundle of the $2$-torus, which is the
 $3$-torus, also has a unique minimal symplectic filling up to
 deformation equivalence \cite{Wendl}.  In all the cases, the proofs
 are adaptations of Eliashberg's technique in \cite{Eliash:holdiscs}
 where the filling of the contact $3$-sphere is foliated by
 holomorphic disks.  Other holomorphic foliations have been used, for
 example, in \cite{Wendl} the filling of the $3$-torus is foliated by
 holomorphic cylinders. However, for the technique to work, curves in
 the holomorphic foliation must have genus zero; otherwise the moduli
 space of holomorphic curves is not transversely cut out.  For the
 fillings of unit cotangent bundles of higher genus surfaces, only
 partial results \cite{sivek2017fillings} are available which say that
 any filling is homotopy equivalent to the disk cotangent bundle.

Fillability properties of contact foliations  resemble those of unfoliated contact manifolds.
For example, 
 a necessary condition for a contact manifold to possess a filling is
 that it is not overtwisted (see \cite{borman2015existence},
 \cite{eliashberg1989classification}, \cite{nieder2006plastik}).
 The
 techniques in \cite{Pino:weinstein} imply that a contact foliation
 with an overtwisted leaf does not admit a foliated filling.

A natural question is whether there are contact foliations whose
symplectic fillings can be classified. The list of unfoliated examples
with unique fillings suggests that we look at foliated unit cotangent
bundles of a foliated $3$-manifold, and the Reeb-foliated $3$-sphere
appeared to be an interesting example.

\begin{figure}
  \centering
  \begin{subfigure}{.5\textwidth}
    \centering \includegraphics[width=.8\linewidth]{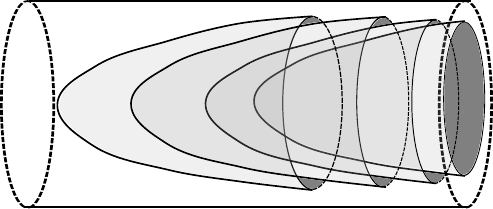}
  \end{subfigure}%
  \begin{subfigure}{.5\textwidth}
    \centering \includegraphics[width=.4\linewidth]{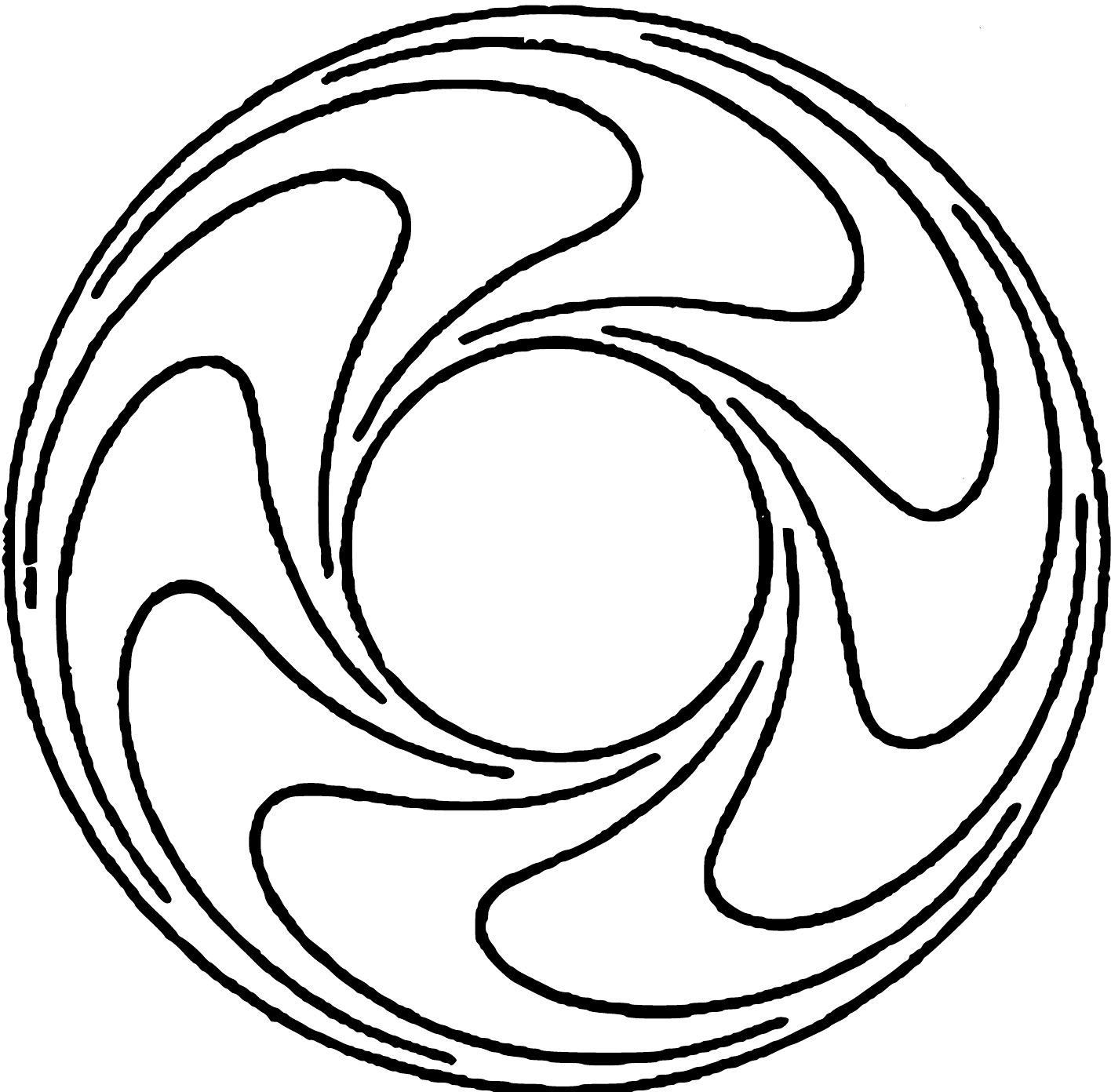}
  \end{subfigure}
  \caption{ {\sc Left}: Reeb foliation on $\R \times \bD^2$. This
    foliation is $\R$-invariant, and descends to a foliation on the
    solid torus $\bS^1 \times \bD^2$.  {\sc Right}: Reeb foliation on
    $\bS^1 \times \bD^2$ restricted to equatorial annulus $A^2$.  The
    Reeb foliation on $\bS^3$ has a compact leaf, which is the
    Clifford torus. It divides $\bS^3$ into two solid tori, each of
    which has the Reeb foliation. }
  \label{fig:Reeb}
\end{figure}

\begin{figure}
  \begin{center}
    \scalebox{.75}{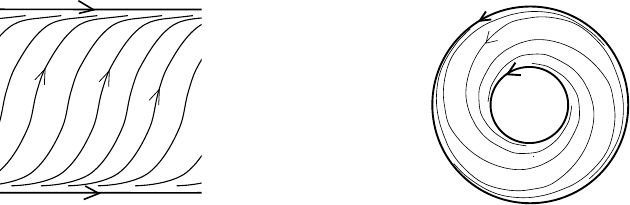}
  \end{center}
  \caption{{\sc Left}: Almost horizontal foliation on the strip
  $\R \times [0,1]$.  {\sc Right:} Almost horizontal foliation on
  the annulus $A^2\simeq \bS^1 \times [0,1]$ obtained by quotienting
  the foliated strip on the left.}
  \label{fig:test}
\end{figure}

\subsection{Statement of results}
Our results cover both weak and strong symplectic foliations, which
are defined as follows.  A codimension-one foliation $\F$ in an
odd-dimensional manifold $M$ is called {\em strong symplectic} if
there is a closed two-form on $M$ that restricts to a symplectic form
on the leaves of the foliation. In case of a {\em weak symplectic
  foliation}, there is a two-form that restricts to a symplectic form
on leaves, but is not required to be closed.

A filling is {\em minimal} if it does not have any embedded symplectic
sphere in a leaf whose self-intersection in the leaf is $-1$.  In case
of strong symplectic foliations, an exceptional sphere is part of an
$\bS^1$-family of exceptional spheres, and it can be blown
down (Proposition \ref{prop:blowdown}). Consequently an arbitrary filling can be transformed to a
minimal filling.  A corresponding result is not available in the weak
symplectic case.

%
\begin{theorem}\label{thm:reeb}
  Let $(\bS^3,\F_{Reeb})$ be the Reeb foliation.  Any minimal strong
  symplectic filling of the sphere cotangent bundle
  $\bS(T^*\F_{Reeb})$ is symplectic deformation equivalent to the disk
  cotangent bundle $\bD (T^*\F_{Reeb})$.
\end{theorem}

Theorem \ref{thm:reeb} is 
proved in Section \ref{subsec:cohinv}. The proof uses the next  result (proved in Section \ref{subsec:cptleaf}), which says that in
a number of cases, a compact leaf in the contact foliation bounds a
compact leaf in the symplectic filling.
\begin{theorem}
  \label{thm:cptleaf}
  Suppose $(M^n,\F^3)$ is a compact foliated manifold with a leafwise contact
  structure and  $W^{n+1}$ is a compact weak
  symplectic foliated filling of $M$. Furthermore, suppose $L_M$ is a leaf of $M$ that is
  contactomorphic to either $\bS^3$, $\R \P^3$, a Lens space $L(p,1)$,
  or $\bT^3$ with the standard contact structure. Then, $L_M$ bounds a
  compact leaf $L$ of $W$.  The
  leaf $L$ in the filling is the canonical filling of $L_M$, except in the case when
  $L_M=L(4,1)$, where $L$ can be one of the two possible minimal
  fillings of $L(4,1)$ (see \cite{Hind:lens}).
\end{theorem}
This result illustrates a contrast between smooth foliations and
symplectic foliations.  If we have a smooth foliated
filling, a compact leaf intersecting the boundary can typically be
destroyed by perturbing the foliation as follows.  Suppose, in a
foliation $(W,\F)$ of codimension one, $L$ is a compact leaf that
intersects the boundary $\partial W$ transversely, and suppose a loop
$\gamma$ transverse to $\F$ intersects the leaf $L$. Then, a
neighborhood of the loop can be replaced by a Reeb component. In the
new foliation $\F'$, all the leaves of $\F$ that intersected $\gamma$
become open and the new compact leaf does not touch the boundary
$\partial W$.

 As pointed out earlier, Theorem
\ref{thm:cptleaf} implies that the filling 
$W_{Reeb}$ of the foliated sphere cotangent bundle $\bS T^*(\bS^3,\F_{Reeb})$
contains a compact leaf, which disconnects $W_{Reeb}$ 
into two pieces $W^+_{Reeb}$, $W^-_{Reeb}$, each of which is a foliated filling of the solid torus with the Reeb foliation.
Eliashberg's technique of constructing a holomorphic subfoliation by rational curves does not naturally work for the filling $W^\pm_{Reeb}$. Therefore, we transform $W^\pm_{Reeb}$
 via a surgery into a
filling of the foliated unit cotangent bundle
\[M_{ah}:=\bS T^*(\bS^1 \times (A^2,\F_{ah})),\]
where $\F_{ah}$ is an almost horizontal foliation on $A^2$.  The
filling of $M_{ah}$ is analyzed via a holomorphic subfoliation by cylinders.

Suppose $W_{ah}$ is a filling of the almost horizontally foliated
manifold $M_{ah}$. In order to analyze $W_{ah}$ using holomorphic
curves, we attach the symplectization
$\R_{\geq 0} \times M_{ah}$ to the boundary of $W_{ah}$, and  obtain a
symplectically foliated manifold $W^\infty_{ah}$ with a cylindrical
end. We prove that the manifold $W^\infty_{ah}$ is a Lefschetz
fibration
\[\pi : W^\infty_{ah} \to \R \times (A^2, \F_{ah})\]
with a foliated base space $\R \times (A^2,\F_{ah})$, 
whose regular fibers are cylinders. Singular points of the fibration
are circles transverse to the foliation, so that on any leaf the set
of singular points is discrete.  There are no singular values on the
boundary $\R \times \partial A^2$.

The combinatorial data underlying a filling that is a strong symplectic foliation differs slightly from the one with a weak symplectic foliation.
The {\em combinatorial type of a strong symplectic foliated 
  filling} consists of the following data.
\begin{enumerate}
\item {\rm (Dehn twists on boundary leaves, $k_\pm \in \Z$)} The
  boundary leaves $\pi^{-1}(\R \times \partial A^2)$ are Lefschetz
  fibrations with no singular points. Denote the outer and inner
  boundaries of $A^2$ by $\partial_+A^2$ and $\partial_-A^2$ respectively.
  There are canonical trivializations of the fibration near the ends
  $\{\pm \infty\} \times M$ obtained via the identification
  $\R_{\geq 0} \times M_{ah} \to W^\infty_{ah}$.  Therefore the
  trivialized bundles $\pi^{-1}((-\infty,\eps) \times \partial_\pm A^2)$
  and $\pi^{-1}((-\eps,\infty) \times \partial_\pm A^2)$ are glued by an
  element of the mapping class group of the cylinder, namely
  $k_\pm \in \pi_1(\R \times \bS^1) \simeq \Z$.
\item {\rm(Combinatorial data of singular loci)}
\label{part:csl}
  Each connected
  component of the singular point set of the Lefschetz fibration
  projects to a closed embedded loop $\gamma$ in the solid torus
  $\R \times A^2$. Let $\Gamma$ 
  be the collection of connected components $\gamma \subset W^\infty_{ah}$ of
  singular points.  The {\em combinatorial data of the singular loci}
  is the multi-set of {\em braid types}
  $\{[\pi(\gamma)] :\gamma \in \Gamma\}$ of the loops of singular values. By
  braid type, we mean the homotopy equivalence class of transverse embedded loops
  in the solid torus $\R \times A^2$. 
\end{enumerate}
We remark that in \eqref{part:csl} above, for a connected component
$\gamma$ of singular values, the projection $\pi(\gamma)$ is
embedded. However for distinct components
$\gamma_1, \gamma_2 \in \Gamma$, the projections $\pi(\gamma_1)$,
$\pi(\gamma_2)$ may intersect. (See Remark \ref{rem:link}.)

If the leaves in $(A^2,\F_{ah})$ are oriented by the arrows
in Figure \ref{fig:test}, then, a monodromy calculation yields the
relation
\begin{equation}
  \label{eq:dehnsum}
  k_+-k_-= \sum_{\gamma \in \Gamma}\wind (\gamma),
\end{equation}
where $\wind(\gamma) \in \Z_+$ is
the degree of the map $\bS^1 \stackrel{\gamma}{\to} \R \times A^2 \stackrel{\pi_1}{\to} \bS^1$,
 or in other words, the number of strands in the braid $\gamma$.

 In case of a weak symplectic foliated filling, a loop of
 singular values $\pi(\gamma)$ in $\R \times A^2$ is allowed to have
 self-intersections.  Consequently, a deformation of a weak symplectic
 foliated filling allows us to unknot the connected components of the
 braids.  Therefore the {\em combinatorial type of a weak symplectic
   foliated filling} consists only of
 \begin{enumerate}
 \item the Dehn twist data on boundary leaves, namely the integers $k_+, k_-$, and 
 \item the winding numbers $\wind(\gamma) \in \Z_+$ of the loops $\gamma$ of singular values. 
 \end{enumerate}

\begin{theorem}\label{thm:ah}
  Minimal strong (or weak) symplectic foliated fillings of $\bS(T^*(\bS^1 \times (A^2,\F_{ah})))$ 
  are classified by their combinatorial type up to strong (or weak) symplectic deformation equivalence.
\end{theorem}

\noindent The proof of Theorem \ref{thm:ah} is given in Section \ref{subsec:modelah}.

It is significant that Theorems \ref{thm:cptleaf} and \ref{thm:ah}
apply to both strong and weak symplectic foliated fillings, because in
general, weak symplectic forms exist for a much wider class of
foliations than the strong counterpart.  For example, a compact strong
symplectic foliation is {\em taut} in the sense that any point has a
transversal loop passing through it. So, the Reeb foliation on
$\bS^3$ does not have a strong symplectic form but it has a weak
symplectic form.  In higher dimensions, Mitsumatsu \cite{Mit} constructed a
weak symplectic foliation on $\bS^5$ that contains a compact
separating leaf, and is therefore, not taut. There is no known strong
symplectic foliation on $\bS^5$.

\textbf{Acknowledgements.} The authors thank Klaus Niederkr\"uger,
Dishant Pancholi, \'Alvaro del Pino and Chris Wendl for
discussions. FP is supported by the Spanish Research Projects
CEX2019-000904-S.554 and PID2022-142024NB-I00. SV is grateful
for the support of the grant 612534 MODULI within the 7th European
Union Framework Programme.

\section{Background}
In this section, we introduce the necessary background on symplectic and contact foliations. We extend the notions of cotangent bundles and fillings to the foliated setting.

\subsection{Symplectic and contact foliations}
A $k$-dimensional {\em foliation} on a manifold $M^n$ is an integrable
distribution $\F \subset TM$ of rank $k$.  We assume that manifolds
are oriented and foliations are co-oriented, and that all foliations
have codimension one, that is, $k=n-1$.  A {\em strong symplectic
  form} on a foliated manifold $(M, \F)$ is a closed two-form
$\om \in \Om^2(M)$ whose restriction to leaves $\om|_\F$ is
symplectic. The pair $(\F,\om)$ is called a {\em strong symplectic
  foliation} on $M$.  Two such forms $\om_0$ and $\om_1$ are
equivalent if $\om_0|_\F=\om_1|_\F$. The equivalence class of a strong
symplectic form $\om_0$ is an affine space. Indeed, if $\om_0$ and
$\om_1$ are equivalent strong symplectic forms then
$(1-t)\om_0 + t\om_1$ lies in the same equivalence class for all
$t \in \R$. A {\em trivial strong symplectic foliation} is a product
of a $(n-k)$-dimensional manifold and a $k$-dimensional symplectic
manifold, and the strong symplectic form is defined by pullback. As
opposed to a strong symplectic form, a {\em weak symplectic form} is a
two-form $\om \in \Om^2(M)$ whose restriction to the leaves $\om|_\F$
is symplectic. The form $\om$ is not required to be closed in $M$. The notion of
equivalence between weak symplectic forms is defined in the same way as above. 

On a foliated manifold $(M,\F)$, a {\em contact structure} is a
sub-distribution $\xi \subset \F$, whose rank is one less than the
rank of $\F$, and that is a contact structure on the leaves. A {\em
  contact form} is a one-form $\alpha \in \Om^1(M)$ that satisfies
$\ker \alpha \cap \F=\xi$. The {\em Reeb vector field} of a foliated
contact form on $(M,\F)$ is a vector field $R_\alpha : M \to T\F$
whose restriction to any leaf is a Reeb vector field
on the leaf.

A {\em foliated symplectomorphism} is a foliated diffeomorphism
$\phi:(W_0,\F_0,\om_0) \to(W_1,\F_1,\om_1)$ between (strong or weak)
symplectic foliations that satisfies
$\phi^*\om_1|_{\F_0}=\om_0|_{\F_0}$. A {\em foliated contactomorphism}
is a foliated diffeomorphism
$\phi:(M_0,\F_0,\xi_0) \to (M_1,\F_1,\xi_1)$ satisfying
$\phi^*\xi_1=\xi_0$. The word `foliated' is sometimes dropped when it
is clear from the context.


\subsection{Symplectic foliated filling}\label{sec:sympfill}
The concept of symplectic fillings of contact manifolds extends to the
foliated setting.  Given a contact foliation $(M,\F_M,\xi)$, a {\em
  strong (or weak) symplectic foliated filling} consists of a strong (or weak) symplectic
foliation $(W,\F,\om)$ whose boundary is transverse to the foliation,
and a foliated diffeomorphism
\begin{equation}
  \label{eq:bdryid}
  i:(M,\F_M) \to (\partial W, \F|_{\partial W})
\end{equation}
such that the leaves of $W$ are symplectic fillings of the
leaves of $M$
in the sense that there is a vector field, called the {\em Liouville vector field},
\begin{equation*}
  Y \in \Gamma(\Op(\partial W),T\F), \quad Y \pitchfork \partial W,
\end{equation*}
which is outward pointing along $\partial W$ and satisfies
\begin{equation}\label{eq:liou}
  L_Y\om|_\F=\om|_\F, \quad (\ker (\iota_Y\om) \cap \F)|_M =\xi.
\end{equation}
The condition \eqref{eq:liou} means that the fillings is `strong' in every leaf 
(as in \cite[Definition 5.1.1 (c)]{Geiges:book}), but this word is suppressed in our terminology.
In this paper the word `strong' only stands for the closedness of the two-form in $W$.

Fillings can also be defined for contact foliated manifolds
$(M,\F_M,\xi)$ with boundary if the boundary is tangent to the
foliation. The filling $(W,\F,\om)$ is a manifold with corners. Its
codimension one boundary splits into two components
\begin{equation*}
  \partial_1W=\partial_\F W \cup \partial_\pitchfork W,
\end{equation*}
where the first is tangent to the foliation $\F$ and the second is
transverse to the foliation.  The contact foliation $M$ can be
identified to the transverse boundary $\partial_\pitchfork W$.  In
addition to the properties of fillings in the previous paragraph, we
additionally require that the tangential boundary $\partial_\F W$ is a
filling of $\partial M$. Thus the two components of the codimension
one boundary of $W$ intersect in a codimension two corner, which is
$\partial M$.

Cylindrical ends can be attached to a symplectic filling as follows.
Given a foliated filling $(W,\F,\om)$ of a contact foliation
$(M,\F_M,\xi)$ and a Liouville vector field $Y$, $\alpha:=\iota_Y \om$
is a foliated contact form on $M$. For a small $\eps>0$, there is a
symplectic embedding
\begin{equation*}
  ((-\eps,0] \times M, d(e^t\alpha)) \to W, \quad (t,m) \mapsto \phi_Y^t(m),
\end{equation*}
where $\phi_Y^t$ is the time $t$ flow of $Y$.  Attaching cylindrical
ends to the filling $W$ produces the foliated symplectic manifold
\begin{equation}
  \label{eq:winfty}
  W^\infty := W \sqcup_{\partial W \simeq \{0\} \times M} [0,\infty) \times M,
\end{equation}
which we call the {\em extended symplectic filling}. The
identification $i:\partial W \to M$ extends naturally to a projection
\begin{equation}
  \label{eq:proj2M}
  i : W^\infty \bs K \to M, 
\end{equation}
where $K$ is a compact set.

The next lemma shows that any two extended fillings of $M$ are related
by a foliated symplectomorphism outside a compact set.
\begin{lemma}\label{lem:phiend}
  Suppose $(W_0,\om_0)$ and $(W_1,\om_1)$ are fillings of the contact
  foliated manifold
  $(M,\F_M,\xi)$.
  Then, there are compact sets $K_l \subset W^\infty_l$, for $l=0,1$,
  and a foliated symplectomorphism
  $\phi:W_0^\infty \bs K_0\to W_1^\infty \bs K_1$ satisfying
  $i_1 \circ \phi = i_0$, where $i_l : W^\infty_l \bs K_l \to M$ is
  the projection as in \eqref{eq:proj2M}.
\end{lemma}
\begin{proof}
  Suppose, for $l=0,1$, $\alpha_l \in \Om^1(M)$ is the contact form
  induced by the Liouville vector field in $W_l$. Since $\alpha_0$,
  $\alpha_1$ represent the same contact structure on $M$, there is a
  map $F:M \to \R$ such that
  $\alpha_1|_{\F_M}=e^F\alpha_0|_{\F_M}$. Then, $\phi$ can be defined
  on the cylindrical ends of $W_1^\infty$ as
  \begin{equation*}
    \phi: W_1^\infty \bs W_1 \to  W_0^\infty \bs W_0, \quad 
    (a,m) \mapsto (a+F(m),m).
  \end{equation*}
\end{proof}
The (strong or weak) foliated fillings $(W_0,\F_0, \om_0)$ and $(W_1,\F_1, \om_1)$ are
{\em symplectic deformation equivalent relative to ends} if the map
$\phi$ in Lemma \ref{lem:phiend} extends to a foliated diffeomorphism
$\phi:W_0^\infty \to W_1^\infty$ and there is a path of (strong or weak) 
symplectic forms $\ol \om_t$ on $W^\infty_0$ and a compact set
$K \subset W^\infty_0$ such that
\begin{equation}\label{eq:defrel}
  \ol \om_0=\om_0, \quad \ol \om_1=\phi^*\om_1, \quad \ol \om_t|_{\F_0}=\om_0|_{\F_0} \quad \forall t \text{ on }W^\infty_0 \bs K.
\end{equation}
The fillings $(W_0,\om_0)$ and $(W_1,\om_1)$ are {\em symplectomorphic
  relative to ends} here if $\phi$ extends to a foliated
symplectomorphism $\phi:W_0^\infty \to W_1^\infty$. The terms
`symplectomorphism' and `symplectic deformation equivalence', when
used in the context of non-compact manifolds, mean that they are
relative to non-compact ends.

\subsection{Foliated cotangent bundle}\label{subsec:tstar}
We now give a class of examples for symplectic and contact foliations.
Let $X$ be a manifold of dimension $n$, and $\F$ is a foliation of
dimension $k$. The {\em foliated cotangent bundle} $T^*\F$ is a vector
bundle $\pi:T^*\F \to X$ with pullback foliation $\pi^*\F$.
We construct a closed two-form on $T^*\F$ that is equal to the canonical
symplectic form on the leaves.
By choosing a splitting of the tangent bundle $TX=T\F \oplus T\F^\perp$,
a cotangent vector on the foliation $\xi \in T^*\F$ extends to a cotangent vector $\hat \xi$ in $T^*X$.
%
%
The one-form $\lam \in \Om^1(T^*\F)$ defined as 
\begin{equation*}
   \lam_\xi(v):= \hat \xi(d\pi(v)) \quad \text{for }\xi \in T^*\F,  v\in T(T^*\F)
 \end{equation*}
 restricts to the Liouville form on each leaf.
 The two-form
 \[\om_\can:=d\lam\]
 restricts to the
canonical symplectic form on the leaves of $T^*\F$. When the foliation
on $X$ is unambiguous, the foliated cotangent bundle is denoted by
$FT^*X$.

In a similar way, the {\em sphere cotangent bundle} $\bS(T^*\F)$ is a
foliated contact manifold. Choose a leafwise Riemannian metric $g$ on
$(X,\F)$. Then,
\begin{equation*}
  \bS(T^*\F):=\{v \in T_x^*\F:x \in X, \Mod{v}_g=1\},
\end{equation*}
with contact form
\[\alpha:=\lam|_{\bS(T^*\F)}\in \Om^1(\bS(T^*\F)).\]
The contact structure on
$\bS(T^*\F)$ is independent of the choice of the leafwise metric $g$,
because $\bS(T^*\F)$ can equivalently be defined as the double cover
of the projectivization $\P(T^*\F)$ with contact structure $\xi$ given
by
\begin{equation*}
  \xi_p:=\{v \in T_p\F_{\P(T^*\F)}:p(\pi_* v)=0\}, \quad p \in \P(T^*\F).
\end{equation*}
The disk bundle $\bD(T^*\F)$ associated to the cotangent space $T^*\F$
is a leafwise strong filling for the unit sphere cotangent bundle
$\bS(T^*\F)$.

\subsection{Foliated blow-ups and blow-downs}\label{subsec:blowdown}
The discussion in Sections \ref{subsec:blowdown} and \ref{subsec:ham}
applies only to strong symplectic foliations. In this section, we
extend the notion of symplectic blowup and blow-down to the foliated
setting, which essentially means, carrying out these operations in an
$\bS^1$-family.  We restrict our dicussion to the case of four-dimensional leaves. Let $(M,\F,\om)$ be a five-manifold with a strong
symplectic foliation of codimension one.

\begin{definition}
  Let $(M,\F)$ be a codimension one foliation, and let $I=[0,1]$ or
  $\bS^1$. 
  A {\em foliated embedding} is an embedding $f : I \times U \to M$ such that
  $\{point\} \times U$ maps to a leaf of $\F$ in $M$. 
\end{definition}

We recall that in 
 a four-dimensional symplectic manifold $(M,\om)$, an {\em
  exceptional sphere} $E \simeq \P^1 \subset M$ (also called a
$(-1)$-sphere) is a symplectic submanifold whose self-intersection is
$-1$. A {\em foliated family of exceptional spheres}
 in a foliated symplectic manifold $(M,\F,\om)$ 
is a foliated embedding
$i:\bS^1 \times \P^1 \hra M$ such that $\{t\} \times \P^1$ is an
exceptional sphere in a leaf. 

Symplectic blow-up can be carried out at a point in the symplectic
manifold with dimension at least four. Analogously on a foliated
symplectic manifold, blow-up is carried out along a closed
transversal, i.e. a loop $\gamma :\bS^1 \to M$ that is transverse to
leaves. Indeed, by Proposition \ref{prop:openstd} below, the foliated
symplectic form $\om$ can be altered without changing $\om|_\F$ so
that the transversal is tangent to the line field $\ker \om$.  A
neighbourhood of the loop is now just a product $\bS^1 \times \bD^4$
with the trivial symplectic foliation, and the blow-up can be carried
out fiberwise.  This process is called {\em foliated blow-up}, and it
can be carried out on any strong symplectic foliation. Therefore,
uniqueness results for foliated symplectic fillings can only be
obtained by restricting attention to {\em minimal fillings}, i.e ones
which do not contain any symplectic exceptional sphere.

In the non-foliated case (see McDuff \cite{mcduff:immersed}),
exceptional spheres can be blown down, and the next result shows that the same is true in a strong symplectic foliation.

\begin{proposition}
  \label{prop:blowdown}
{\rm(Foliated blow-down)} 
  Let $(W^\infty,\F,\om)$ be a strong symplectic foliation
  that is a filling of the unit cotangent bundle of a compact contact foliation $M$, and let
  $S \subset W^\infty$ be
  an exceptional sphere in a leaf.
  \begin{enumerate}
  \item \label{part:bda} Then, $S$ is part of a foliated $\bS^1$-family of exceptional
    spheres given by the foliated embedding $\bS^1 \times S \to M$.
  \item \label{part:bdb} The $\bS^1$-family of exceptional spheres can be blown down to produce a strong symplectic foliation $(M',\F',\om')$ in the following sense: A foliated blow-up of $(M',\F',\om')$ is foliated symplectomorphic to
    $(M,\F,\om)$. 
  \end{enumerate}
\end{proposition}
\begin{proof}
  We choose a leafwise tame almost complex $J$ for which the given
  sphere is holomorphic. By automatic transversality in four
  dimensions, the moduli space of embedded holomorphic $(-1)$-spheres
  is 1-dimensional.  By Gromov compactness in the foliated setting,
  the moduli space is a finite number of copies of $\bS^1$. For any
  copy of $\bS^1$ in the moduli space, each sphere in the family is
  embedded in a leaf.  Furthermore, by Lemma
  \ref{lem:exceptional-embed}, isolated intersections between two
  spheres are ruled, and consequently the family $\bS^1 \times S$ of
  exceptional spheres is embedded in $W^\infty$, which proves
  \eqref{part:bda}. By Proposition \ref{prop:p1nbhd} below, there is a
  foliated symplectic form $\om_1$ such that $\om_1|_\F=\om|_\F$, and
  for which the restriction of $(\F,\om_1)$ to a neighbourhood of the
  spheres is a trivial symplectic foliation (as in Definition
  \ref{def:trivfol}). Blow-down can then be performed fiberwise, and
  \eqref{part:bdb} follows.
\end{proof}

\subsection{Hamiltonian perturbation of a foliated symplectic form}
\label{subsec:ham}
In this section, we show that a family of exceptional spheres in a strong symplectic foliation has a standard product neighborhood, up to a foliated symplectomorphism. In particular, we perturb the strong symplectic form without changing its restriction on the leaves in order to achieve the standard form in the neighborhood of the family of exceptional spheres.


\begin{definition}\label{def:trivfol}
  \begin{enumerate}
  \item \label{part:tfa}
    {\rm(Fibrations as foliations)} A foliation $(M,\F)$ with a
    leafwise symplectic form $\om \in \Om^2(T\F)$ is a {\em symplectic
      fibration} if $M \to M/\F$ is a fiber bundle, whose fibers are
    symplectomorphic. We denote the fiber by $F$.  Such a foliation
    $(M,\F)$ is a {\em trivial symplectic fibration} if there is a
    foliated diffeomorphism $M \simeq (M/\F \times F)$.
  \item \label{part:tfb}
    {\rm(Trivial symplectic foliation)} A strong symplectic
    foliation $(M,\F,\om)$ is {\em trivial} if it is a trivial
    symplectic fibration, that is, $M \simeq (M/\F \times F)$, and
    $\om$ is the pullback of a symplectic form on $F$.
  \end{enumerate}
\end{definition}

The following proposition says that if a subset of a strong symplectic
foliation is a trivial fibration, there is a strong symplectic form
that is equal on the leaves, for which the subset is a trivial
symplectic foliation.

\begin{proposition}\label{prop:openstd}
  Suppose $(M,\F,\om)$ is a codimension one strong symplectic
  foliation, and $i:I \times U \to M$ is a foliated embedding where
  $I=\bS^1$ or $[0,1]$ and $U$ is a simply connected open manifold of
  dimension $\dim(\F)$.  Suppose $i_t^*\om$ is independent of $t$ for
  all $t \in I$. Then, given a compact subset $V \subset U$, there is
  a strong symplectic form $\om_1 \in \Om^2(M)$ with
  $\om_1|_\F=\om|_\F$ such that $i^*(\F,\om_1)$ is a trivial strong
  symplectic foliation on the product $I \times V$.
\end{proposition}
\begin{proof} Choose any $t \in I$, and let
  $\tilde \om \in \Om^2(I \times U)$ be the pullback of
  $i_t^*\om \in \Om^2(U)$. The forms $\tilde \om$ and $i^*\om$ agree
  on the foliation. So, we can write
  \begin{equation*}
    \tilde \om - i^*\om=\alpha \wedge dt, \quad \text{where } \alpha:=i_{\ppt}(\tilde \om - i^*\om).  
  \end{equation*}
  By the closedness of the forms $\tilde \om$, $i^*\om$, we obtain
  $d\alpha \wedge dt=0$. This condition implies $d\alpha|_\F=0$. Since
  $U$ is simply connected, there is a function $f:I \times U \to \R$
  so that $df|_\F=\alpha|_\F$, and consequently,
  \begin{equation*}
    \tilde \om - i^*\om=df \wedge dt \quad \text{on } I \times U.
  \end{equation*}
  Let $\eta:U \to [0,1]$ be a cut-off function that is $1$ on $V$ and
  $0$ outside a neighbourhood of $V$ in $U$. The required form $\om_1$
  is defined as $\om+(i^{-1})^*(d(\eta f) \wedge dt)$ on the image of
  $i$, and as $\om$ outside $\im(i)$.
\end{proof}

 Let $\Bl_\eps \C^2$ be the point blow-up of $\C^2$ equipped a
K\"ahler form $\om_\eps$ that integrates to $\eps$ on the exceptional
sphere.  By the symplectic neighborhood theorem, the neighborhood of
an exceptional sphere $E$ in any manifold $(M,\om)$ is
symplectomorphic to a neighborhood of the exceptional divisor in
$\Bl_\eps\C^2$, where $\eps:=\int_E \om$. The following is a family
version of this symplectomorphism.

\begin{proposition}\label{prop:p1nbhd}
{\rm(Family of exceptional spheres has a product neighbourhood)}
Suppose $(M^5,\F^4,\om)$ is a strong symplectic foliation. Let
$Z \subset M$ be such that the restriction $(\F,\om)|_Z$ is a
fibration by symplectic spheres, and each of the spheres is an
embedded $(-1)$-sphere in the leaf containing it and has area
$\eps>0$.  Then, there is a closed form $\om_1 \in \Om^2(M)$ with
\[\om_1|_\F=\om|_\F,\]
and a neighborhood $\Op(Z) \subset M$ of $Z$ such that
$(\Op(Z), \F, \om_1)$ is a trivial symplectic foliation whose leaves
are symplectomorphic to a neighborhood of the exceptional divisor in
$\Bl_\eps\C^2$.
\end{proposition}
\begin{proof}
  By the hypothesis, $(Z, \F|Z,\om|_{\F \cap Z})$ is a symplectic
  fibration with fiber $\P^1$ and symplectic form $\sig$. The
  fibration is trivial because $\Symp(\P^1,\sig)$ is homotopic to
  $\Diff(\P^1)$, the space of orientation-preserving diffeomorphisms,
  which is connected. Next, we claim that a neighborhood
  $\Op(Z) \subset M$ is leafwise symplectomorphic to a neighborhood of
  the exceptional spheres in
  $\bS^1 \times (\Bl_\eps \C^2 , \om_\eps)$.  For a single leaf of
  $\Op(Z)$, the symplectomorphism is via the symplectic neighbourhood
  theorem. The isomorphism for the $\bS^1$-bundle follows from the
  fact that an $\bS^1$-family of rank two symplectic vector bundles
  $E \to \P^1$ is just the product $\bS^1 \times E$, since the space
  $\Aut(E)$ of bundle isomorphisms of a rank two symplectic vector
  bundle is connected. \footnote{The space of symplectic bundle
    isomorphisms is homotopy equivalent to the subspace of
    isomorphisms which are also metric preserving. The latter space is
    $\Map(\P^1,S^1)$.}  Finally, by Proposition \ref{prop:openstd},
  there is strong symplectic form $\om_1$ with $\om|_\F=\om_1|_{\F}$
  which makes $\Op(Z)$ a trivial symplectic foliation.
\end{proof}

\section{Holomorphic foliations}\label{sec:holfol}
We define holomorphic curves on foliated symplectic manifolds with
cylindrical ends and analyze symplectic fillings in the almost
horizontal case.

\subsection{Holomorphic curves} \label{subsec:holc} Almost complex
structures are defined on the leaves of the foliation and holomorphic
curves lie on the leaves. We use the same terminology as the
corresponding notions in unfoliated manifolds.  Later in this section,
we define the notion of holonomy transport on a foliated anifold. To
define moduli spaces of holomorphic curves, we will restrict attention
to holomorphic curves on which the holonomy transport is trivial.

\begin{definition}\label{def:Jcyl}
  {\rm(Cylindrical almost complex structures)}
  \begin{enumerate}
  \item On a foliated symplectic manifold $(W,\F,\om)$, an {\em almost
      complex structure} is a bundle automorphism $J:T\F \to T\F$
    compatible with $\om|_\F$ and satisfying $J^2=-\Id$.
  \item {\rm(Almost complex structures on symplectizations)} Suppose
    $(M,\F,\xi)$ is a foliated contact manifold, and let
    $\alpha \in \Om^1(M)$ be a contact form whose Reeb vector field is
    $R_\alpha$.  In the symplectization $\R \times M$, the tangent
    space of the leaves admits a splitting
    $T(\R \times \F)=\R \oplus R_\alpha \oplus \xi$.  Let $\pa$ denote
    the unit vector field in the $\R$-direction.  The space of {\em
      cylindrical almost complex structures}, denoted by
    $\J_\alpha(\R \times M)$, consists of compatible almost complex
    structures $J$ on $\R \times M$ that satisfy
    \begin{enumerate}
    \item $J$ is invariant under $\R$-translations in $\R \times M$,
    \item $J\pa=R_\alpha$,
    \item $J\xi=\xi$, and $J|\xi$ is compatible with $d\alpha$.
    \end{enumerate}
  \item {\rm(Almost complex structures on manifolds with cylindrical
      end)} A non-compact foliated symplectic manifold
    $(W^\infty,\F,\om)$ possesses {\em cylindrical ends} if its non-compact
    ends are symplectomorphic to
    $(\R_{\geq 0} \times M,d(e^t\alpha))$, where $(M,\F_M,\alpha)$ is
    a foliated contact manifold.  Given a cylindrical almost complex
    structure $J_0$ on the symplectization $\R \times M$, we denote by
    \[\J_\om(W^\infty,J_0)\]
    the space of almost complex structures on $W^\infty$ that restrict
    to $J_0$ on the cylindrical ends, and that are $\om$-compatible on
    $W^\infty \bs (\R_{\geq 0} \times M)$
  \end{enumerate}
\end{definition}

We consider pseudoholomorphic maps defined on punctured Riemann
surfaces that are tangent to the foliation and which are
asymptotically close to Reeb cylinders in the neighborhoods of
punctures.  A {\em Reeb cylinder} is a holomorphic curve in
$\R \times M$ that projects to a Reeb orbit $\gamma: \R/T\Z \to M$;
up to translation in the domain, such a Reeb cylinder is of the form
\[\tilde \gamma: \R \times \R/\Z \to \R \times M, \quad (s,t) \mapsto (Ts, \gamma(Tt)). \]
A {\em punctured Riemann surface} $\oSig$ is the
complement $\Sig \bs \{z_1,\dots,z_k\}$ of a finite set of points in a
compact Riemann surface $(\Sig, j_\Sig)$.
Given an almost structure
$J$ on $(W^\infty,\F)$ that is cylindrical on the end
$\R_{\geq 0} \times M$, a {\em $J$-holomorphic map}
$u: \oSig \to W^\infty$ is tangent to the foliation $\F$, is
$(j,J)$-holomorphic, and is asymptotically close to Reeb cylinders near punctures in the following sense:
\begin{definition} {\rm(Asymptotically cylindrical)} A map
  $u: \oSig \to W^\infty$ lying in a leaf $L$ is {\em asymptotically
    cylindrical} if at any puncture point $z_i$ of $\oSig$ with
  complex coordinates
  $(s_i,t_i): \Op(z_i) \bs \{z_i\} \to \R_{\geq 0} \times S^1$ if
  there is a Reeb orbit $\gamma: \R/\Z\to M$ with period $T > 0$, a
  point $a \in \R$, and constants $c, \delta>0$ such that
  \begin{equation}
    \label{eq:exp-decay}
    d_\F(u(s_i,t_i), (a + Ts_i, \gamma(T t_i)) \leq ce^{-\delta s},  
  \end{equation}
  where $d_\F$ is a leafwise distance metric on $W^\infty$. (Note that
  the Reeb orbit $\gamma$ lies in the same leaf as $u$.)
\end{definition}

We assume without mention that holomorphic maps are asymptotically
cylindrical. This condition is equivalent to the Hofer energy being
finite; the equivalence is proved by Proposition \ref{prop:energyhomology} and Theorem
\ref{thm:rem-fol} in Section \ref{sec:hol-prop}.
The holomorphic maps we consider also have trivial holonomy transport, as defined next.

\begin{definition}\label{def:holonomy}
  Let $W$ be a manifold with a foliation $\F$ of codimension one. A
  path $\tau:(-\eps,\eps) \to W$ that is transverse to the foliation
  $\F$ is called a {\em transversal}.
  \begin{enumerate}
  \item {\rm(Foliated homotopy of transversals)} \label{part:hol1}
    Two transversals
    $\tau_0, \tau_1:(-\eps,\eps) \to W$ are related by {\em
      foliated homotopy} if there exists a smooth map 
    $\Gamma : [0,1] \times (-\eps,\eps) \to W$ such that
    $\Gamma(0,\cdot)=\tau_0$, $\Gamma(1,\cdot)=\tau_1$ and for any
    $t \in (-\eps,\eps)$, the image of $\Gamma(\cdot,t)$ is
    contained in a leaf.
  \item {\rm(Holonomy transport)} \label{part:hol2}
    Suppose $\gamma:\bS^1 \to L$ is a
    loop in the leaf $L$ of a foliated manifold $(W,\F)$, where the
    foliation is of codimension one. Let $\tau$ be a transversal
    through $\gamma(0)$ with $\tau(0)=0$.  For a small enough
    $\eps>0$, there is a foliated homotopy of transversals
    $\{\tau_t:(-\eps,\eps) \to W \}_{t \in [0,1]}$ satisfying
    $\tau_t(0)=\gamma(t)$ for all $t \in [0,1]$, whose starting point
    is $\tau_0:=\tau|(-\eps,\eps)$, and final point is
    $\tau_1=\tau \circ h$ for an embedding $h:(-\eps,\eps) \to \R$
    with $h(0)=0$.  The map $h$ is independent of $\gamma$ up to
    homotopy, and therefore, we obtain a group homeomorphism
    \begin{equation}
      \label{eq:def-ht}
      \pi_1(L,x_0) \to \on{Homeo}(\tau), \quad [\gamma] \mapsto h,  
    \end{equation}
    which is called the {\em holonomy transport along $\gamma$}. Here,
    $\on{Homeo}(\tau)$ is the group of germs of self-homeomorphisms of
    $\tau$ that fix the point $L \cap \tau$. See p.141 in Calegari
    \cite{Calegari}.
  \end{enumerate}
\end{definition}

We will work with holomorphic curves in leaves of a symplectic
foliation. The space of such maps is a foliated manifold as we now
explain.  Let $\B:=\Map_\F(\Sig,W)$ denote the space of asymptotically cylindrical
smooth maps $u:\Sig \to W$ whose images are tangent to the foliation 
and have trivial
holonomy transport.  Consider a map $u \in \B$. Let $\bt$ be a
transverse coordinate defined in the neighbourhood of the image of
$u$, whose level sets are tangent to the foliation. Such a coordinate
exists because of the trivial holonomy condition.  There is a
splitting of the tangent space
\begin{equation}
  \label{eq:splitB}
  T_u\B=\Gamma(\Sig,u^*T\F) + \R\partial_{\bt}.
\end{equation}
In a neighbourhood of a map $u \in \B$, we can define a transverse
coordinate $\bt_\B$ on $\B$ as $\bt \circ u$.  After extending $\B$ to
Sobolev completions of maps (defined in Section \ref{subsec:fred}), it
is a foliated Banach manifold.
The leafwise linearized Cauchy-Riemann operator of a holomorphic map
$u \in \B$ is defined between the spaces
\begin{equation*}
  D_u^\F:\Gamma(\Sig,u^*T\F) \to \Om^{0,1}(\Sig,u^*T\F).
\end{equation*}
A foliated $J$-map $u$ is {\em unobstructed} if $D_u^\F$ is
surjective. Results about the manifold structure of the moduli space
of holomorphic curves in the non-foliated setting carry over to this
setting.  We remark that if a map $u:\Sig \to (W,\F)$ does not have
trivial holonomy, it can not be `transported' to neighbouring leaves,
and the moduli space can not be a foliated manifold.

\subsection{Holomorphic curves in
  symplectizations}\label{subsec:symplectization}

In this section, we study holomorphic curves in the symplectization of
the foliated unit cotangent bundle of a three-manifold with an almost horizontal foliation. 

We consider the three-manifold
\begin{equation}
  \label{eq:defX}
  X:=\bS^1 \times A^2,  
\end{equation}
where $A^2 \subset \R^2$ is an annulus.  The foliation on $X$ is the
pullback of the {\em almost horizontal foliation} on the annulus $A^2$,
which is uniquely described by two conditions (see for example \cite{Rei}), namely that
\begin{enumerate*}
\item the two boundary components of $A^2$ are compact leaves, and all
  the other leaves are non-compact,
\item there is an orientation of the foliation that extends the
  counter-clockwise orientation of the boundary leaves.
\end{enumerate*}
The foliation $\F_{ah}$ on the annulus can be concretely described as
follows.  Define a foliation on the strip $[-1,1] \times \R$ whose
leaves are $\{y+c= \tan \frac {\pi x} 2\}_{c \in \R}$, and
$\{x=\pm 1\}$. Then, $(A^2,\F_{ah})$ is the quotient of the strip by
the translation action of $\Z$ on the second coordinate. This
foliation has the property that there is a coordinate
\begin{equation*}
  q_2:A^2 \to \R/\Z
\end{equation*}
whose level sets are transverse to leaves, which we call the {\em leaf
  coordinate}.  In our concrete description, we can take $q_2:=y$.
The foliation on the $3$-manifold $X$ is the pull-back of $\F_{ah}$,
which by abuse of notation, we call an {\em almost horizontal
  foliation} and denote by $\F_{ah}$.

A {\em holomorphic foliation} of a foliated symplectic manifold
$(W,\F)$ is a subfoliation of $\F$ by punctured holomorphic curves
that are asymptotically cylindrical at punctures.  The almost complex
structure is assumed to be cylindrical on the ends, and as a result,
the holomorphic subfoliation can be chosen to be $\R$-invariant on the ends.

\begin{proposition}\label{prop:folcylinder}{\rm(Holomorphic foliation on symplectization)}
  Let $(X,\F):=\bS^1 \times (A^2,\F_{ah})$. There is a contact form
  $\alpha$ on the foliated unit cotangent bundle $M:=\bS(FT^*X)$ and a
  cylindrical almost complex structure $J \in \J_\alpha(\R \times M)$
  for which $\R \times M$ has an $\R$-invariant holomorphic foliation by
  cylinders.
 Furthermore, for any leaf in the holomorphic foliation that is not a Reeb cylinder, the ends are asymptotic to Reeb orbits in the positive end, namely $\{\infty\} \times M$.
\end{proposition}
\begin{proof}
  As discussed in Section \ref{subsec:tstar}, a leafwise contact form
  on the sphere cotangent bundle $\bS(FT^*X)$ is determined by a
  choice of leafwise Riemannian metric on $X=\bS^1 \times A^2$.
  Let
  \[(q_1,q_2,\tau) : \bS^1 \times A^2 \to \bS^1 \times \bS^1 \times [0,1]\]
  be a diffeomorphism. Then, $(q_1,q_2)$ induce local coordinates on
  neighborhoods in leaves, and we choose the leafwise Riemannian
  metric to be $dq_1^2 + dq_2^2$. We remark that the level sets of
  $\tau$ are not tangent to the foliation.

  We next choose suitable coordinates on the sphere cotangent bundle.
  Let $(p_1,p_2)$ be the co-ordinates on the fibers of $T^*\F$
  corresponding to the coordinates $(q_1,q_2)$ on the leaves of $X$,
  so that the symplectic form on $T^*\F$ is $\sum_i dp_i \wedge dq_i$.
  The unit cotangent bundle $M$ is the hypersurface
  $\{p_1^2+p_2^2=1\}$ of $T^*\F$.  Thus $M$ has coordinates
  $(q_1,q_2,\theta,\tau)$, where $\theta$ is given by
  $(\cos \theta,\sin \theta)=(p_1,p_2)$, and $(q_1,q_2,\theta)$ induce
  local coordinates on neighborhoods in leaves.  Using the Liouville
  vector field $v=\sum_i p_i\partial_{p_i}$, we obtain a contact form
  and Reeb vector field on $M$:
  \begin{equation*}
    \alpha=\cos \theta dq_1 + \sin \theta dq_2, \quad R_\alpha=\cos \theta \partial_{q_1} + \sin \theta \partial_{q_2}.
  \end{equation*}
  Here, we note that the vector fields
  $\partial_{q_1}, \partial_{q_2}, \partial_\theta \subset T\F$ and
  are obtained from the local coordinates $(q_1,q_2,\theta)$ on
  neighborhoods in leaves. The Reeb orbits are confined to level
  sets of $\theta$. For two of the level sets $\theta=0,\pi$, all Reeb
  orbits are closed with period $2\pi$; they are Morse-Bott
  submanifolds as in Definition \ref{def:mb} and are denoted by
  \begin{equation}
    \label{eq:mp0}
    \mP_0:=\{\theta=0\}, \quad \mP_\pi:=\{\theta=\pi\} \subset M.
  \end{equation}
  We now describe the foliation by holomorphic curves.  Choose an
  almost complex structure $J_0 \in \J_\alpha(M)$ as
  \begin{equation}\label{eq:cylJ}
    J_0\partial_a=R_\alpha, \quad J_0\partial_\theta=-\sin \theta \partial_{q_1} + \cos \theta \partial_{q_2}.
  \end{equation}
  The leaves of the holomorphic foliation are the connected components
  of the fibers of the map
  \begin{equation}\label{eq:holfib}
    \pi_{hol}: \R \times M \to \R \times (A^2,\F_{ah}), \quad (a,(q_1,q_2,\theta,\tau)) \mapsto (e^a\sin \theta, (q_2,\tau)).
  \end{equation}

  To examine the asymptotic behavior of the leaves of the holomorphic
  foliation, we use a different set of coordinates.  The
  symplectization $\R \times M$ may be viewed as the cotangent bundle
  $FT^*X$ minus the zero section, and there are coordinates on
  $\R \times M$ which induce a homeomorphism
    \begin{equation}
      \label{eq:coord}
       ((p_1,p_2),(q_1,q_2,\tau)) : \R \times M \to (\R^2 \bs \{0\}) \times X, 
    \end{equation}
    given by $p_1=e^a \cos \theta$, $p_2=e^a \sin \theta$, and  $J_0 \partial_{p_i}=\partial_{q_i}$ for $i=1,2$. 
The map $\pi_{hol}$ can be rewritten as 
\[\R \times M \ni (p_1,p_2, q_1,q_2, \tau) \mapsto (p_2,q_2,\tau) \in \R \times A^2.\]
The inverse image of $\{p_2=0\}$ consists of two components, each of them is a family of Reeb cylinders that project to Reeb orbits in $\mP_0$ and $\mP_\pi$. 
If $p_2 \neq 0$, the fiber over $(p_2,q_2,\tau)$ is
a holomorphic cylinder with a parametrization
\[u_{(p_2,q_2,\tau)}: \R \times \R/2\pi \Z \to (\R^2 \bs \{0\}) \times X, \quad (s,t) \mapsto ((p_1=s,p_2),(q_1=t,q_2,\tau)).\]
For each of these cylinders, the limits $\lim_{s \to \pm \infty}u_{(p_2,q_2,\tau)}(s,t)$ are Reeb orbits contained in $\{\infty\} \times \mP_0$ and $\{\infty\} \times \mP_\pi \subset \ol \R \times M$.  
%

  
\end{proof}

\begin{remark}{\rm(A geometric interpretation of the Reeb vector
    field)}
  A geodesic on a Riemannian manifold lifts to a flow line of the Reeb
  vector field on the unit cotangent bundle, see Geiges \cite[Theorem
  1.5.2]{Geiges:book}. The leaves of the 3-manifold $X$ have a
  Riemannian metric $dq_1^2 + dq_2^2$. Geodesics are curves
  $\gamma:I \to X$ lying on a leaf for which the ratio
  $\frac {dq_2(\gamma')} {dq_1(\gamma')}$ is constant. A geodesic
  $\gamma$ for which the ratio is $\lam$ lifts to a Reeb flow line
  $(\gamma,\theta=\tan^{-1}(\lam))$ in the foliated unit cotangent
  bundle $M$. On a non-compact leaf of $M$, the only flow lines that
  close up are those for which $q_2=constant$, and so, $\theta=0, \pi$.
\end{remark}

\begin{remark} \label{rem:sympmod}
  \begin{enumerate}
  \item {\rm (On the moduli space of cylinders in the
      symplectization)} Let $\M^{\R \times M}$ denote the leaf space
    of the above holomorphic foliation minus the Reeb cylinders.  The
    holomorphic foliation is invariant under $\R$-translation on
    $\R \times M$. The moduli space $\M^{\R \times M}/\R$ has two
    components, namely $\M^{\R \times M}_\pm/\R$, consisting of maps
    with $\theta$ coordinate in the intervals $(0,\pi)$ and
    $(-\pi,0)$. Further, by \eqref{eq:holfib} each cylinder projects
    to a point in $(A^2,\F_{ah})$, and in fact the maps
    \begin{equation*}
      \M^{\R \times M}_\pm/\R \to (A^2,\F_{ah})
    \end{equation*}
    are foliated diffeomorphisms.  Finally, we observe that cylinders
    in $\M^{\R \times M}$ project to embedded cylinders in $M$, and
    this projection is invariant under the $\R$-action on
    $\M^{\R \times M}$. The cylinders foliate $M - \mP$ where $\mP$ is
    the union of Morse-Bott submanifolds $\mP_0$ and $\mP_\pi$ from \eqref{eq:mp0}.
    Thus, there is a map
    \begin{equation}
      \label{eq:fibratem}
      \pi:M -\mP \to \M^{\R \times M}/\R,    
    \end{equation}
    whose fibers are cylinders.
  \item {\rm(Sections of the holomorphic fibration)}
    \label{part:sec}
    For the fibration \eqref{eq:fibratem}, we can choose sections
    \begin{equation}
      \label{eq:spm}
      s_\pm:\M^{\R \times M}_\pm/\R \to M,
    \end{equation}
    such that $\theta(s_+) \in (0,\eps)$ and
    $\theta(s_-) \in (-\eps,0)$. Further, they can be chosen so that
    the maps $s_\pm:(A^2,\F_{ah}) \to M$ are homotopic as foliated maps.
  \end{enumerate}
\end{remark}

\begin{corollary}\label{cor:folcyl}
  Given $a \in \R$, the holomorphic foliation of Proposition \ref{prop:folcylinder} has a  leaf
  that is contained in $[a,\infty) \times M$.
\end{corollary}
\begin{proof}
  Any leaf $L$ of the holomorphic foliation that is not a Reeb
  cylinder is contained in $[a_0,\infty) \times M$ for some $a_0 \in \R$, since the leaf is
  a cylinder whose ends asymptote to Reeb orbits in
  $\{\infty\} \times M$. Since the holomorphic foliation of
  $\R \times M$ is $\R$-invariant, translating $L$ by $(a-a_0)$ in the
  $\R$ direction gives a leaf that is required by the corollary.
\end{proof}

\subsection{Holomorphic curves in the filling}\label{subsec:holfill}

In this section, we describe a family of holomorphic curves on the
foliated filling of a sphere cotangent bundle.
The holomorphic curves
lie on leaves of the foliation, and we will show that they form a
singular subfoliation of the foliated filling.

We first describe the foliated manifold studied in this section.  We
will study a filling of the foliated sphere cotangent bundle of the
$3$-manifold
\[X:=\bS^1 \times (A^2,\F_{ah}), \]
where the annulus $A^2$ has the almost horizontal foliation $\F_{ah}$
as described following \eqref{eq:defX}.  The foliated sphere cotangent
bundle is denoted by
\begin{equation}
  \label{eq:defM}
  M:=(\bS(FT^*X), \alpha).   
\end{equation}
The contact form $\alpha$ is as in Proposition \ref{prop:folcylinder}.
Let $(W,\om)$ be a (weak or strong) symplectic filling of $M$, and let $W^\infty$ be
the extended filling, i.e. $W^\infty$ is obtained by attaching
cylindrical ends to $W$.  Recall that Proposition
\ref{prop:folcylinder} gives a tame cylindrical almost complex structure
$J_0 \in \J_\alpha(\R \times M)$ on the symplectization $\R \times M$.
Let
\begin{equation}
  \label{eq:Jom}
  \J_\om(W^\infty,J_0):=\{J \in \J(W^\infty,\om) | \exists a : J|_{[a,\infty)
    \times M}=J_0, \text{$J$ is $\om$-compatible}\}   
\end{equation}
be the space of compatible almost complex structures on $W^\infty$
that are equal to $J_0$ on the cylindrical ends.

Next, we describe a single holomorphic curve in $W^\infty$.  Starting
from this curve, we will later obtain a family of curves that will
foliate the manifold.  Consider any almost complex structure
$J \in \J_\om(W^\infty,J_0)$. There exists $a>0$ such that
$ J|_{[a,\infty) \times M}=J_0$. By Corollary \ref{cor:folcyl}, 
 the $J_0$-holomorphic foliation of
 $\R \times M$ from Proposition \ref{prop:folcylinder} contains a leaf
 $u_0:\R \times \bS^1 \to \R \times M$ that is contained in
$[a,\infty) \times M$, and thus it is also an embedded $J$-holomorphic
map
\[u_0 : \R \times \bS^1 \to W^\infty. \]
Let $\ol W^\infty$ be the compactification of $W^\infty$ obtained by
adding $\{\infty\} \times M$ to the cylindrical end.  The ends of the
cylinder $u_0$ are asymptotic to Reeb orbits in $\{\infty\} \times M$:
\[\lim_{s \to \infty}u_0(s,t) \subset \mP_\pi \subset \{\infty\}
  \times M, \quad \lim_{s \to -\infty}u_0(s,t) \subset \mP_0 \subset
  \{\infty\} \times M. \]
Thus, $u_0$ represents a relative homology class
\begin{equation}
  \label{eq:betadef}
  \beta \in H_2(\ol W^\infty, \mP_0 \cup \mP_\pi), \quad
  \beta:=(u_0)_*[\ol \R \times \bS^1, \{\pm \infty\} \times \bS^1 ] 
\end{equation}
Define a moduli space of maps
\begin{multline}
  \label{eq:moduli}
  \M(\beta, J):=\left \{u: \R \times \bS^1 \to W^\infty : \text{$u$ is
    $J$-holomorphic},
  \lim_{s \to \infty}u(s,t) \in \mP_0, \right.\\ \left.
  \lim_{s \to -\infty}u(s,t) \in \mP_\pi, u_*(\R \times \bS^1)=\beta \in
  H_2(W^\infty, \mP_0 \cup \mP_\pi)\right \}/\Aut(\R \times \bS^1).
\end{multline}
Here, the group $\Aut(\R \times \bS^1)$ of domain automorphisms is
$\C^\times$.  The following is the main result of this section. 
\begin{proposition}\label{prop:folcylend}{\rm(Holomorphic foliation on
    fillings)}
  Let
  $W$ be a (weak or strong) symplectic foliation that is a filling of the foliated sphere cotangent bundle
  $M$ from \eqref{eq:defM}. Let 
  $J \in \J(W^\infty, J_0)$ be a cylindrical almost complex
  structure that is generic in $W$, and let $\beta$ be the homology
  class in \eqref{eq:betadef}.
  \begin{enumerate}
  \item \label{wa} The moduli space $\M:=\M(\beta,J)$ is a manifold of
    dimension $3$ and possesses a codimension one foliation. Every
    curve in $\M$ is embedded and no two curves in $\M$ intersect.
  \item \label{wb} The moduli space $\M$ has a compactification
    $\ol \M$, and the foliation $\F_\M$ extends to the boundary. The
    boundary $\ol \M \bs \M$ consists of
    \begin{enumerate}
    \item \label{wb1} a compact foliated 2-dimensional manifold, where
      each point represents a building with an empty main level and
      one nontrivial upper level that is a leaf in the holomorphic
      foliation of $\R \times M$.
    \item \label{wb2} A compact 1-dimensional manifold, denoted by
      $\M_{nodal}$, that is transverse to the foliation $\F_\M$, where
      each point represents a nodal curve consisting of two embedded
      index zero curves in $W^\infty$. Each of the nodal curves is
      disjoint from the curves in $\M$.
    \end{enumerate}
  \item \label{wc} The collection of curves in $\M$ and the nodal
    curves in $\M_{nodal}$ form a foliation of $W^\infty$ outside of
    the one-dimensional set of nodal points.  Any nodal
    point is a transversal intersection of two leaves;
    these are the nodes of the nodal curves in $\ol \M \bs \M$.
  \end{enumerate}
\end{proposition}
\begin{proof}
  All curves in the moduli space $\M$ have trivial holonomy transport,
  because the curves are asymptotic to a Reeb orbit at the cylindrical
  ends, and the Reeb orbit has trivial holonomy.
  Therefore by Proposition \ref{prop:transv}, 
  the moduli space of curves is a foliated manifold if leafwise unobstructedness holds.
  For a generic $J$, leafwise unobstructedness holds by the following paragraph. 

A neighborhood of the moduli space of maps is unaffected by
non-compactness of leaves. Indeed, each of the curves is asymptotic to
a Reeb orbit in the same leaf, and the projection
\begin{equation}
    \label{eq:piM}
    \pi_M(\on{Im}(u) \cap (W^\infty \bs W))   
  \end{equation}
  in $M$ is compact. Therefore, the local description of the moduli
  space is the same as in the unfoliated case. The arguments in the
  proof of \cite[Proposition 7]{Wendl} imply that the intersection
  number $i(u_1,u_2)$ between curves $u_1, u_2 \in \M$
  vanishes \footnote{If $u_1$ and $u_2$ are asymptotic to the same
    Reeb orbit, the number $i(u_1,u_2)$ also includes a non-negative
    contribution from intersections at infinity in the sense of
    Siefring \cite{Sie}. The arguments for the vanishing of this number are also the same as in the unfoliated case.}, 
  and punctures
  are odd. Since limit Reeb orbits are simply covered, these facts
  imply that the curves in $\M$ are embedded, linearized operators are
  unobstructed. Furthermore, a neighborhood in a leaf of $\M$ consists
  of curves that foliate a neighborhood of a leaf in $W^\infty$.

  By Theorem \ref{thm:folcpt}, the moduli space $\M$ has a
  compactification $\ol \M$. Theorem \ref{thm:folcpt} is indeed
  applicable, because by Proposition \ref{prop:energyhomology} and
  Remark \ref{rem:HEbd}, there is a uniform bound on the Hofer energy
  of curves in $\M$.  A curve $u \in \ol \M \bs \M$ may have multiple
  levels in a leaf of the foliation or it may be a nodal curve.
  First, consider the case that $u$ is a multi-level curve. By the
  arguments in Step 3 in the proof of \cite[Proposition 7]{Wendl}, $u$
  only consists of an upper level, where it is a curve occurring in
  the holomorphic foliation of the symplectization $\R \times M$. If
  $u \in \ol \M \bs \M$ is a nodal curve, each node in $u$ lowers the
  expected dimension of the moduli space by $2$. That is, for the
  moduli space $\M_u$ of nodal curves whose components are homologous
  to $u$, the expected dimension is $\dim(\M) - 2(\# nodes)$.  Moduli
  spaces with negative expected dimension are empty, because curves
  are not multiply covered since they are asymptotic to simple Reeb
  orbits, $J$ is generic in $W$, and a component of a curve can not
  lie in $W^\infty \bs W$ since by Step 2 in the proof of
  \cite[Proposition 7]{Wendl}, the component can not have isolated
  intersections with an element of $\M(\R \times M)$.  Therefore a
  nodal curve $u$ in $\ol \M \bs \M$ has exactly one node. Since the
  filling is minimal, a component of $u$ can not be a cover of an
  exceptional sphere.  The only other possibility is that $u$ consists
  of two disks connected at a node.  By Proposition \ref{prop:transv},
  the moduli space $\M_u$ is a one-dimensional manifold transverse to
  the foliation $\F_W$, and part \eqref{wb} follows. Part \eqref{wc}
  follows from the fact that distinct curves in $\ol \M$ do not have
  intersections.
\end{proof}

\begin{remark}
  The moduli space $\ol \M$ is a manifold with corners. It has four
  (codimension one) boundary components:
  \begin{itemize}
  \item two of them are transverse to the foliation, consisting of
    height two curves as in \eqref{wb1}. We call them the {\em top}
    and {\em bottom} boundary, and they are canonically identified to
    $\R$-equivalence classes in the symplectization, i.e. there are
    standard foliated diffeomorphisms
    \begin{equation}
      \label{eq:topbot}
      (A^2,\F_{ah}) \to \partial_\pm \ol \M \to \M^{\R \times M}_\pm/\R.
    \end{equation}
  \item The other two boundary components are tangent to the foliation
    $\F$ and consist of curves foliating the fillings of $\bT^3$, the
    boundary components of $\bS(FT^*(\bS^1 \times (A^2,\F_{ah})))$.
    This is a consequence of the definition of fillings of manifolds
    with boundary, see Section \ref{sec:sympfill}.  We call these the
    {\em side boundaries} of $\ol \M$, and denote them by
    $\partial_{side,\pm}\ol \M$.
  \end{itemize}
\end{remark}
\begin{proposition}\label{prop:moduliann}
  There is a foliated diffeomorphism
  $(A^2,\F_{ah}) \times [0,1] \to \ol \M$.  that extends the canonical
  diffeomorphism \eqref{eq:topbot} on the top and bottom boundaries.
\end{proposition}
\begin{proof}
  We first show that there are maps from the annulus to the top and
  bottom boundaries, that are homotopic in $\ol \M$. There is a
  foliated map
  \begin{equation}
    \label{eq:piw}
    \pi: W^\infty \to \ol \M
  \end{equation}
  that sends a point to the curve in $\ol \M$ on which it lies.  Let
  $\ol W^\infty$ denote the compactification of $W^\infty$ by adding
  $\{\infty\} \times M$ to the cylindrical end.  We recall that the
  top and bottom boundary components $\partial_\pm \ol \M$ can be
  canonically identified to $\R$-equivalence classes of maps
  $\M^{\R \times M}/\R$ in the symplectization, see \eqref{eq:topbot}.
  Therefore, the map \eqref{eq:piw} extends to
  \begin{equation}
    \label{eq:projcpt}
    \pi :\ol W^\infty\bs \mP \to \ol \M,
  \end{equation}
  where $\mP \subset \{\infty \} \times M$ is the union of the
  Morse-Bott submanifolds $\mP_0$, $\mP_\pi$. Here
  $\pi|_{\{\infty\} \times M}$ is defined by composing the fibration
  $M -\mP \to \M^{\R \times M}/\R$ in \eqref{eq:fibratem} with the
  identification $\M^{\R \times M}_\pm/\R \to \partial_\pm \ol \M$ in
  \eqref{eq:topbot}.  The maps
  $s_\pm:\partial_\pm \ol \M \to \ol W^\infty \bs \mP$ in
  \eqref{eq:spm} are sections of the fibration
  \eqref{eq:projcpt}. Since these maps are homotopic in $W^\infty$,
  the homotopy $(A^2,\F_{ah}) \times [0,1] \to \ol W^\infty \bs \mP$ can
  be projected to $\ol \M$ to obtain a foliated map
  \begin{equation}
    \label{eq:sa}
    s:(A^2,\F_{ah}) \times [0,1] \to \ol \M.
  \end{equation}
  By construction, the inverse images of $\partial_+\ol \M$ and
  $\partial_-\ol \M$ are $\{1\} \times A^2$ and $\{0\} \times A^2$,
  and  the restriction $s|(\{0,1\} \times A^2)$ is a diffeomorphism onto the
  top and bottom boundary of $\partial_\pm \ol \M$. We can extend this
  statement to a neighbourhood of the boundary -- there are neighborhoods
  $\Op(\partial_\pm \ol \M)$, $\Op(\{0,1\} \times A^2)$ such that
  \[s^{-1}(\Op(\partial_\pm \ol \M))=\Op(\{0,1\} \times A^2),\]
and 
  \begin{equation}
    \label{eq:topbott}
    s:\Op(\{0,1\} \times A^2) \to \Op(\partial_\pm \ol \M)
  \end{equation}
  is a foliated diffeomorphism.

  In the rest of the proof, we show that $s$ can be homotoped via
  foliated maps to a foliated diffeomorphism.  We perform the
  deformation relative to the neighborhood $\Op(\{0,1\} \times A^2)$
  of the top and bottom boundaries, where by \eqref{eq:topbott}, $s$
  is already a diffeomorphism.  We start with the side boundary.  By
  Lemma \ref{lem:cyldiff}, the side boundary of the moduli space
  $\partial_{side}\ol \M$ is a cylinder, and the restriction of $s$ to
  the side boundary can be homotopically deformed, relative to the
  corners, to a diffeomorphism
  $\partial A^2 \times [0,1] \to \partial_{side}\ol \M$.  Further,
  $s^{-1}(\partial_{side}\ol \M)=\partial A^2 \times [0,1]$, because
  $s$ is a foliated map and so, it is submersive in the leaf
  direction. Consequently an interior leaf can not map to a boundary
  leaf.  The deformation of $s$ on the lateral boundary leaf can be
  extended to the interior so that the restriction
  \begin{equation*}
    s:s^{-1}(\Op(\partial_{side}\ol \M)) \to \Op(\partial_{side} \ol \M)
  \end{equation*}
  is a foliated diffeomorphism onto its image, and
  $s^{-1}(\Op(\partial_{side}\ol \M))$ is a neighbourhood of
  $[0,1] \times \partial A^2$.

  So far, we have a foliated map $s : (A^2,\F_{ah}) \times [0,1] \to \ol \M$ that is a diffeomorphism in a neighborhood of the boundary. That is, there are neighborhoods of the boundaries such that
  \[s^{-1}(\Op(\partial \ol \M))=\Op(\partial(A^2 \times [0,1])),\]
  and
  \[s : \Op(\partial(A^2 \times [0,1])) \to \Op(\partial \ol \M)   \]
  is a foliated diffeomorphism.
  The complement $A^2 \times [0,1] \bs \Op(\partial(A^2 \times [0,1]))$ is equal to $\bS^1 \times \bD^2$ with the trivial foliation, and 
  \begin{equation}
    \label{eq:sdisk}
    s : \bS^1 \times \bD^2 \to \ol \M \bs \Op(\partial \ol \M)  
  \end{equation}
  is a foliated map that maps the boundary $\bS^1 \times \partial \bD^2$ diffeomorphically to the boundary of $ \ol \M \bs \Op(\partial \ol \M)$. Suppose that for any $t \in \bS^1$, the image $s(\{t\} \times \bD^2)$ is contained in the leaf $\bL_t \subset \ol \M \bs \Op(\partial \ol \M)$.
  Then $\bL_t$ is a surface with a single boundary component $s(\{t\} \times \partial \bD^2)$ that is contractible in $\bL_t$, and therefore, $\bL_t \simeq \bD^2$. Consequently, by \cite[Theorem 6.1.5]{Candel:book},
  $\ol \M \bs \Op(\partial \ol \M)$ is foliated diffeomorphic to $\bS^1 \times \bD^2$ with the trivial foliation, and
  $\ol \M \simeq (A^2,\F_{ah}) \times [0,1]$.
  We may rewrite the map $s$ in \eqref{eq:sdisk} as a foliated map
  \[ s : \bS^1 \times \bD^2 \to \bS^1 \times \bD^2,\]
  where both sides
  have the trivial foliation and $s$ maps the boundary
  $\bS^1 \times \partial \bD^2$ identically to itself. Then, $s$ can be
  deformed to the identity map via a foliated homotopy because the set
  $\Map_\partial(\bD^2,\bD^2)$ of maps from $\bD^2$ to itself which
  are identity on the boundary is convex.
\end{proof}

The following result was used in the proof of Proposition \ref{prop:moduliann}.
\begin{lemma}\label{lem:cyldiff}
  Suppose $X$ is a two-manifold with boundary, and
  $s: [0,1] \times \bS^1\to X $ is a smooth map such that
  $s(\{0,1\} \times \bS^1) \subset \partial X$, and
  $s|_{\Op(\{0,1\} \times \bS^1)}$ is an embedding. Then, $X$ is a
  cylinder and $s$ can be homotoped relative to boundaries to a
  diffeomorphism.
\end{lemma}
\begin{proof}
  The manifold $X$ is a cylinder because $s$ gives a homotopy between
  two of its boundary components $s(\{0\} \times \bS^1)$,
  $s(\{1\} \times \bS^1)$.  We fix a diffeomorphism
  $X \to [0,1] \times \bS^1$, so that $s$ is an identity map near the
  ends.  By projecting to $\bS^1$, the map $s$ represents a loop in
  the space of degree $1$ maps from $\bS^1$ to $\bS^1$, called
  $\on{Map}_1(\bS^1,\bS^1)$. The space $\on{Map}_1(\bS^1,\bS^1)$ is
  homotopy equivalent to $\bS^1$, and the loop represented by $s$ is
  classified by an integer $k \in \Z$.  The homotopy class
  corresponding to any $k$ has a representative that is a
  diffeomorphism.
\end{proof}


\begin{proposition} {\rm(No nodal curves on lateral
    boundary)} \label{prop:nodein} Let $\M_{nodal} \subset \ol \M$
  denote the one-dimensional moduli space of nodal curves as in
  Proposition \ref{prop:folcylend} \eqref{wb2}.  If $W$ is a minimal
  filling, $\M_{nodal}$ is a finite union of circles, which do not
  intersect the lateral boundaries $\M$ ($\partial_{side,\pm}\ol
  \M$). Each point in this 1-dimensional manifold represents a nodal
  curve which is a union of two disks.
\end{proposition}
\begin{proof}
  Index zero spheres are ruled out because the filling is minimal, and the only other
  possible nodal curve is a union of disks. Neither of these disks are
  contained in the end $[R,\infty) \times M$, because the only curves
  of $\ol \M$ that are contained in the end $[R,\infty) \times M$ are
  those that occur in the symplectization. Further, the curves are
  somewhere injective because of their asymptotic behaviour.  Since
  $J$ is generic away from the ends, both components of the nodal
  curve have index $\geq -1$. By the reasoning in Proposition
  \ref{prop:folcylend}, the curves have even index, and therefore the
  index is zero. So, the moduli space of nodal curves is a
  one-dimensional manifold transverse to the foliation on $\ol \M$.

  Next, we claim that $\M_{nodal}$ does not intersect the lateral
  boundaries. The argument is as in Wendl \cite{Wendl}. Indeed,
  $W^\infty \to \ol \M$ is a Lefschetz fibration which is singular
  along $\M_{nodal}$. The monodromy map on a loop in
  $\partial_{side,\pm}\ol \M$ enclosing all the points of
  $\M_{nodal} \cap \partial_{side,\pm}\ol \M$ is trivial, because the
  bundle is trivial on the ends of the moduli space
  $\partial_{side,\pm} \ol \M \simeq \R \times \bS^1$. The mapping
  class group of the cylinder is $\Z$, and is generated by a single
  element. So the product of positive Dehn twists cannot be identity.

\end{proof}

Finally, we show that by adjusting the almost complex structure on
$W^\infty$, we obtain a holomorphic subfoliation that is standard in
the non-compact ends.  The standard filling of the sphere cotangent
bundle $M:=\bS (FT^* X)$ is the disk cotangent bundle
$W_{std}:=\bD (FT^* X)$. By attaching cylindrical ends, we obtain
the cotangent bundle $W^\infty_{std}=FT^*X$. There is a foliated
cylindrical almost complex structure on $J_0$ on $FT^*X$ for
which the fibers of the projection
\begin{equation}\label{eq:stdfol}
  FT^*X \to \R \times (A^2,\F_{ah}), \quad (q_1,q_2,p_1,p_2,\tau) \mapsto (p_2,(q_2,\tau))
\end{equation}
are holomorphic. Indeed, on the cylindrical ends
$[0,\infty) \times M$, $J_0$ can be defined as in
\eqref{eq:cylJ}. We refer to the holomorphic foliation in
\eqref{eq:stdfol} as the {\em standard holomorphic foliation} on the
cylindrical ends $[0,\infty) \times M$.

\begin{proposition}
  \label{prop:endstd} {\rm(Holomorphic foliation is standard in ends)}
  Suppose $W$ is a filling of $M:=\bS (FT^* X)$ that is a weak or strong symplectic foliation.  There is a
  cylindrical almost complex structure $\hat J$ on $W^\infty$ that is
  \begin{enumerate}
  \item equal to $J_0$ on $[R,\infty) \times M$ for a large $R$.
  \item Suppose $\ol \M_{\hat J}$ is the moduli space of
    $\hat J$-holomorphic curves given by Proposition
    \ref{prop:folcylend}. The holomorphic foliation induced by the
    curves in $\ol \M_{\hat J}$ on $[R,\infty) \times M$ is the
    same as the standard holomorphic foliation on
    $[R,\infty) \times M$.
  \end{enumerate}
\end{proposition}
\begin{proof}
  We start with an arbitrary generic cylindrical almost complex
  structure $J$ that is equal to $J_0$ on $[0,\infty) \times
  M$. Let $\ol \M_J$ be the moduli space of $J$-holomorphic
  curves produced by Proposition \ref{prop:folcylend}. Take $R>0$ such
  that there are no nodes in $[\frac R 2,\infty) \times M$.  The
  holomorphic foliation in $[\frac R 2,\infty) \times M$ is
  homotopic to the standard foliation \eqref{eq:stdfol} in that
  region.  In the region $[\frac R 2, R] \times M$, homotopically
  deform the holomorphic foliation induced by $\ol \M_J$, leaving it
  unchanged near $\{\frac R 2\} \times M$, and making it agree
  with the standard holomorphic foliation near $\{R\} \times
  M$. By taking $R$ large enough, the deformation can be carried
  out such that the leaves are symplectic.  Define an almost complex
  structure $\hat J$ on $[\frac R 2, R] \times M$ such that the
  leaves of the new foliation are $\hat J$-holomorphic.  Extend
  $\hat J$ to all of $W^\infty$ by defining $\hat J=J$ outside
  $[\frac R 2, R] \times M$. The holomorphic foliation induced by
  $\ol \M_{\hat J}$ agrees with the standard one on
  $[R,\infty) \times M$.
\end{proof}

The following result is used in Section \ref{subsec:blowdown} to show that
 a family of exceptional spheres is necessarily embedded.

\begin{lemma}\label{lem:exceptional-embed}
  {\rm(Exceptional spheres do not intersect)} Let $W$ be a strong symplectic foliation that is a 
  filling of a compact contact foliation $M$, and let $W^\infty$ be
  its extension.  Let $J$ be a compatible almost complex structure on
  $W^\infty$ that is cylindrical in the ends and generic in $W$. Let
  $u_1, u_2 : \P^1 \to W^\infty$ be embedded $J$-holomorphic spheres
  with self-intersection $(-1)$ and distinct images. Then, $u_1$ and
  $u_2$ do not intersect.
\end{lemma}
\begin{proof}
  Assume for the sake of contradiction, that $u_1$ and $u_2$ intersect
  at $p \geq 1$ points, and they lie in a leaf $L \subset W^\infty$ of the
  foliation.  
  By genericity of $J$, we may assume that the
  intersection is transverse. Gluing at one of the intersection
  points, we obtain a holomorphic immersed sphere $u_0: \P^1 \to L$ (which
  is embedded if $p=1$) with $c_1(u) \geq 2$. By automatic
  transversality in four-manifolds, the linearized leafwise
  Cauchy-Riemann operator $D_{u_0}^\F$ is surjective, and the moduli space
  $\M_L$ of maps in $L$ containing $u_0$ is smooth.

  First, consider the case that the intersection number $i(u_0,u_0)$
  is positive.  Then, any curve $u_0$ in the moduli space $\M_L$
  intersects a fixed curve $u_0' \in \M_L$.  By Proposition \ref{prop:bddia}
  (which is applicable by the compactness of $W$), there is a compact
  set $K$ in the leaf $L$ that contains the images of the maps in
  $\M_L$.  Therefore, there is a compactification $\ol \M_L$, and the
  images of the maps in the compactified moduli space $\ol \M_L$ cover
  the leaf $L$. The arguments in McDuff \cite{mcduff:immersed} imply
  that the leaf $L$ is $\P^2$ with a set of point blow-ups, and in
  particular, is simply connected.  However, if a foliated manifold
  $W^\infty$ has a compact simply connected leaf, then $W^\infty$ is
  an $S^1$-fibration with fibers diffeomorphic to $L$. This
  contradicts the fact that $W^\infty$ is non-compact, and therefore,
  this case does not occur.

  It remains to consider the case that $i(u_0,u_0)=0$ and therefore,
  $c_1(u_0)=2$, $u_0$ is embedded and has a trivial normal bundle.
  Consider the moduli space $\M$ of holomorphic spheres
  $u: \P^1 \to W^\infty$ tangent to $\F$ that are of the same homology
  class as $[u_0]$. Since $u_0$ has trivial holonomy, by a version of Proposition \ref{prop:transv}
  for compact curves, the moduli space $\M$ is a foliated manifold. 
  A map $u \in \M$ can not be contained be
  contained in the end $W^\infty\bs W$ because the symplectic form is
  exact in the end. Consequently, by a version of the monotonicity
  theorem (such as Proposition \ref{prop:bddia}), there is a compact
  set $K \subset W^\infty$ that contains the images of all maps in
  $\M$. Therefore, there is a compactification $\ol \M$ of $\M$, and
  the curves in $\ol \M$ cover $W^\infty$. We have arrived at a
  contradiction because $\ol \M$ is compact and $W^\infty$ is
  non-compact, leading to the proof of the Lemma.
  
\end{proof}

\section{Model fillings}
Suppose $W_{ah}$ is a filling of the foliated unit cotangent bundle of
$\bS^1 \times (A^2,\F_{ah})$. The results of Section \ref{sec:holfol}
can be summarized by saying that the extended filling
$W^{\infty}_{ah}$ is a foliated Lefschetz fibration over
$\R \times (A^2,\F_{ah})$ that has a standard structure on non-compact
ends.  The regular fibers are cylinders $\R \times \bS^1$.  In this
section, we prove that such Lefschetz fibrations can be classified up
to symplectic deformation equivalence using combinatorial data.

A part of the combinatorial data comes from the boundary leaves of
$W^\infty$. The boundary of a Lefschetz fibration on
$\R \times (A^2,\F_{ah})$ has two boundary leaves, each of which is a
Lefschetz fibration over a cylinder.  These are extended fillings of
$\bT^3$, and were studied by Wendl \cite{Wendl}.

\subsection{Model filling of $\bT^3$}
The manifold $\bT^3$ has a standard contact structure, by viewing it
as the unit cotangent bundle of $\bT^2$.  The standard filling of
$\bT^3$ is the disk cotangent bundle $\bD( T^*\bT^2)$. By a Luttinger
surgery of the zero section $\bT^2 \subset T^*\bT^2$, one obtains a
$\Z^2$-family of fillings. Wendl \cite{Wendl} proved that any filling
of $\bT^3$ is symplectic deformation equivalent to one of the elements
in this family. The $\Z^2$-parameter associated to a filling is called
the {\em Luttinger constant} of the filling.  In this section, we
first describe the Luttinger surgery for a choice of parameter
$(k,k_b) \in \Z^2$. We then prove a result about which Luttinger
constants occur on the boundary of the almost horizontal filling.


The description of the Luttinger surgery is on the lines of Wendl
\cite{Wendl}, who in turn, followed Auroux-Donaldson-Katzarkov
\cite{ADK}.  To describe the surgery, we choose an additional
parameter $c \in \R_{>0}$, but later it will be shown that various
choices lead to symplectomorphic manifolds. We denote
$\sig=(c,k,k_b) \in \R_{>0} \times \Z^2$.  The surgery is performed on
the cotangent bundle
$$T^*\bT^2 = \{(q_1,q_2,p_1,p_2)\in (\bS^1)^2 \times (\R^2) \}, \quad \om_0=\sum_{i=1,2} dq_i \wedge dp_i$$
along the zero section $\{p_1=p_2=0\}$. For any $r>0$, let
$K_r:=\{(q,p) \in T^*\bT^2:|p_1|,|p_2|<r\}$. There is a
symplectomorphism $\psi_\sig:K_{2c} \bs K_c \to K_{2c} \bs K_c$
defined as
\begin{equation*}
  \psi_\sig(q_1,q_2,p_1,p_2):=  
  (q_1+k\chi(p_2)\beta(\frac {p_1} c), q_2+k_b\chi(p_1)\beta(\frac {p_2} c), p_1, p_2)),
\end{equation*}
where $\chi:\R \to \{0,1\}$ is $0$ on $\R_{<0}$ and $1$ on
$\R_{\geq 0}$, and $\beta:\R \to [0,1]$ is a smooth cut-off function
that is $0$ on $(-\infty,-1]$, $1$ on $[1,\infty)$ and satisfies
$\int_{-1}^1t\beta'(t) dt=0$. The output of the Luttinger surgery is
the symplectic manifold
\begin{equation*}
  (\Lutt_\sig(T^*\bT^2),\om_\sig):=(T^*\bT^2 \bs K_c,\om_0) \cup_{\psi_\sig} (K_{2c},\om_0).
\end{equation*}

\begin{remark}
  \label{rem:noneqt3}
  For a non-trivial Luttinger parameter $\sig$, the spaces $T^*\bT^2$
  and $\Lutt_\sig (T^*\bT^2)$ are symplectomorphic. The
  symplectomorphism is given by global coordinates $(Q_1,Q_2,P_1,P_2)$
  on $\Lutt_\sig(T^*\bT^2)$ defined as follows: $(Q,P)$ is the
  standard coordinate on $K_{2c}$, and thus on the overlap, we have
  $(q,p)=\psi_\sig(Q,P)$. The coordinates $(Q,P)$ extend to all of
  $T^*\bT^2\bs K_{2c}$.  However, the restriction of the
  symplectomorphism to the ends $\R_{\geq 0} \times \bT^3$. is not
  isotopic to identity.  So $T^*\bT^2$ and $\Lutt_\sig(T^*\bT^2)$ are
  not symplectic deformation equivalent relative to ends.
\end{remark}

\begin{proposition} {\rm (Classification of $\bT^3$-fillings
    \cite[Proposition 5.6]{Wendl})} \label{prop:lutt} Suppose
  $(W^\infty,\om)$ is a filling of $M=(\bT^3,\alpha_0)$ with infinite
  ends attached, so there is a symplectomorphism to the non-compact
  ends $F:[R,\infty) \times M \to W^\infty$ for some $R>0$.  There is
  a Luttinger surgery parameter $\sig$ such that the map $F$ extends
  to a symplectic deformation equivalence $F:W_\sig \to W^\infty$,
  i.e. $F$ is a diffeomorphism such that $(1-t)\om_\sig + tF^*\om$ is
  symplectic on $W_\sig$ for all $t \in [0,1]$.
\end{proposition}

\begin{remark}\label{rem:lutt}
  {\rm(Interpretations of Luttinger surgery parameters)} There is a
  cylindrical almost complex structure $J_0$ on $T^*\bT^2$ such that
  the fibers of the map
  \begin{equation*}
    \pi:T^*\bT^2 \to \bS^1 \times \R, \quad (q_1,q_2,p_1,p_2) \mapsto (q_2,p_2)
  \end{equation*}
  are holomorphic cylinders.  For any Luttinger parameter $\sig$,
  there is an $\om_\sig$-compatible almost complex structure on
  $W_\sig$ that is equal to $J_0$ on the ends
  $[R,\infty) \times \bT^3$, such that the fibers of the map
  \begin{equation}
    \label{eq:cylfib}
    \pi: (Q_1,Q_2,P_1,P_2) \mapsto (Q_2,P_2)
  \end{equation}
  are $J_\sig$-holomorphic.  The Luttinger parameters $(k,k_b)$ have
  the following interpretations.
  \begin{enumerate}
  \item {\rm(As offsets of ends of cylinders)} The cylinder
    $u^{(\rho,\eta)}:=\{(P_2,Q_2)=(\rho,\eta)\}$ has one end
    asymptotic to $\{\theta=\pi, q_2=\eta\}$ and the other end
    asymptotic to $\{\theta=0, q_2=\eta+k_b\beta(\rho/c)\}$.
  \item {\rm(As Dehn twists)} Let us restrict ourselves to the case
    when $k_b=0$. We will see in Proposition \ref{prop:luttah} that
    $k_b$ indeed vanishes in the almost horizontal case. If $k_b=0$,
    then $(Q_2,P_2)=(q_2,p_2)$.  Use the coordinates in $K_{2c}$
    (which in this case are the same as the coordinates in $\{p_2<2c\}$) to
    trivialize the bundle $\pi^{-1}\{p_2<2c\}$, and the coordinates on
    $T^*\bT^2\bs K_c$ to trivialize $\pi^{-1}\{p_2>c\}$. The Luttinger
    surgery has the effect of gluing these trivial bundles with a Dehn
    twist in the overlap, i.e. cylinder in second trivialization is
    the cylinder in first trivialization with a positive $k$ Dehn
    twist applied to it.
  \end{enumerate}
\end{remark}
%
%

\begin{proposition} {\rm(Luttinger constants for almost horizontal
    filling)} \label{prop:luttah} Suppose $W^\infty_{ah}$ is a
  foliated filling of $M_{ah}:=\bS(FT^*X)$, where
  $X:=\bS^1 \times (A^2,\F_{ah})$, and $(A^2,\F_{ah})$ is the annulus
  with an almost horizontal foliation.  Denote by $\partial_+ A^2$ and
  $\partial_- A^2$ the inner and outer boundaries of the annulus.
  Suppose $(k_b^\pm,k^\pm)$ are the Luttinger parameters of the
  fillings of $\bS (FT^*(\bS^1 \times \partial_\pm A))$ (obtained by
  applying Proposition \ref{prop:lutt}). Then,
  \begin{enumerate}
  \item \label{part:kb} $k_b^\pm=0$,
  \item and $k^+ - k^- \geq 0$ is the number of nodal curves
    (Proposition \ref{prop:folcylend} \eqref{wb2}) in each non-compact
    leaf of the filling of $M_{ah}$.
  \end{enumerate}
 
\end{proposition}
\begin{proof}
  There is an identification of the cylindrical end
  $F:[R,\infty) \times M_{ah} \to W^\infty_{ah}$ for a large $R>0$.
  By Proposition \ref{prop:lutt}, the restrictions
  $F_\pm:=F|_{\bS(FT^*(\bS^1 \times \partial_\pm A))}$ extend to
  diffeomorphisms
  \begin{equation}
    \label{eq:bdryext}
    F_\pm:\Lutt_{(k_b^\pm,k^\pm,c^\pm)} (T^*\bT^2) \to \partial_\pm W^\infty_{ah},
  \end{equation}
  where $\partial_\pm W^\infty_{ah}$ are boundary components of
  $W^\infty_{ah}$.  Suppose $(q_1,q_2,\theta)$ are coordinates on
  $\bT^3 \simeq \bS(T^*\bT^2)$ defined as in the proof of Proposition
  \ref{prop:folcylinder} -- i.e. the unit cotangent fiber has
  coordinate $\theta$, the $\bS^1$ factor in $X$ has coordinate $q_1$,
  and the boundary of the annulus has coordinate $q_2$.  We denote
  loops in these directions as $\gamma_{q_1}$, $\gamma_{q_2}$ and
  $\gamma_\theta$.  The diffeomorphisms \eqref{eq:bdryext} implies
  that the loop
  $\gamma_{\theta}+ k^\pm\gamma_{q_1} + k_b^\pm\gamma_{q_2}$ is
  contractible in $\partial_{side,\pm}W_{ah}$. The loops
  $\gamma_{q_1}$ and $\gamma_{\theta}$ have trivial holonomy transport
  and $\gamma_{q_2}$ has non-trivial holonomy transport. This implies
  $k_b^\pm=0$.

  The moduli space of nodal curves $\M_{nodal}$ intersects each of the
  non-compact leaves in $\ol \M$ at a finite number of points. Any
  non-compact leaf $\L$ is of the form $\R \times [0,1]$. The
  monodromy map on a loop that encloses all the points in
  $\M_{nodal} \cap \L$ is $k_1^+ -k_1^-$, indeed the fibration is
  trivial in a neighbourhood of $\R \times \{0\}$ and
  $\R \times \{1\}$, and going down along $\{t\} \times [0,1]$ for a
  large $t$, one picks up a monodromy of $k_1^+$, and if $t$ is a
  large negative number, the monodromy is $k_1^-$, see Remark
  \ref{rem:lutt}.
\end{proof}

\subsection{Combinatorial type}
\label{subsec:combi}
In Section \ref{subsec:holfill}, we showed that the filling of the
foliated unit cotangent bundle of $\bS^1 \times (A^2,\F_{ah})$ is a
Lefschetz fibration over $\R \times (A^2,\F_{ah})$ with a standard
structure on the non-compact ends.  In this section, we recall the
definition of a smooth Lefschetz fibration, and we give combinatorial
invariants for these fibrations that provide a classification up to
symplectic deformation equivalence.
The total space of the fibration is allowed to be a strong or weak symplectic foliation, the combinatorial data associated to the fibration are slightly different in the weak and strong cases.

We consider Lefschetz fibration
whose regular fibers 
are cylinders
\begin{equation*}
  F:=\{(p_1,q_1) \in \R \times \R/\Z\}, \quad \om_F:=dq_1 \wedge dp_1.
\end{equation*}

\begin{definition}\label{def:folTLT}
A {\em Lefschetz fibration on a foliated base manifold} with fiber $F$
consists of a foliated five-manifold $(X,\F_X)$ and a map to a foliated three
manifold $(B,\F_B)$
\begin{equation*}
  \pi:(X,\F_X) \to (B,\F_B)
\end{equation*}
that is a map of foliations in the sense that $\F_X:=\pi^*\F_B$ and which satisfies 
 the following: For any $x_0 \in X$,
\begin{enumerate}
\item either $d\pi_{x_0}$ is surjective, or
\item there are coordinates $(z_1,z_2,t) \in \C^2 \times \R$ in a neighborhood $U_{x_0} \subset X$ of $x_0$ and coordinates $(w,\tau) \in \C \times \R$ on a neighborhood $U_{\pi(x_0)} \subset B$ of $\pi(x_0)$ such that $t$, $\tau$ are constant on leaves of $\F_X$, $\F_B$ respectively, and
  \[((w,\tau) \circ \pi)=(z_1^2 + z_2^2, t) \]
  on $U_{x_0}$. The point $x_0$ is a {\em singular point} and $\pi(x_0)$ is a {\em singular value}.
\end{enumerate}
 For a regular value $b \in B$, $\pi^{-1}(b) \simeq F$ and for a singular value $b$, for which $\pi^{-1}(b)$ has $k$ singular points, the fiber is a nodal curve consisting of two open disks and $(k-1)$-spheres. connected to each other at nodal points.
 This ends the definition.
\end{definition}

A Lefschetz fibration $\pi:X \to \R^2$ is {\em compatible with a
  symplectic form $\om$} on $X$,
 if the regular fibers of $\pi$ are symplectic, and in case of singular fibers, the irreducible components are symplectic.

We consider Lefschetz fibrations over the foliated base manifold
\begin{equation*}
  B:=\R \times (A^2,\F_{ah}), \quad p_2:B \to \R, \quad q_2 : A^2 \to \bS^1.
\end{equation*}
Here $\F_{ah}$ is the almost horizontal foliation on the annulus
$A^2$, $p_2$ is the projection to the first coordinate, and $q_2$ is a
$\bS^1$-valued leaf coordinate. The two-form
\begin{equation*}
  \om_B:=dq_2 \wedge dp_2
\end{equation*}
is a strong symplectic form on $B$.

We study foliated symplectic Lefschetz fibrations
$\pi:X \to \R \times A^2$ whose structure is {\em standard on
  ends}. That is, there is a foliated symplectomorphism
\begin{equation}
  \label{eq:ends}
  i = (\pi,(p_1,q_1)):X \bs K \to ((\R \times A^2) \times F,\om_0:=\om_B \oplus \om_F )\bs K_0
\end{equation}
that lifts the identity map on the base $\R \times A^2$, and
$K \subset X$, $K_0 \subset \R \times A^2 \times F$ are compact sets.
Since the Lefschetz fibration is standard on the non-compact ends as
in \eqref{eq:ends}, there are no singular fibers on the boundary
$\R \times \partial A^2$; the reason is as in the second half of the
proof of Proposition \ref{prop:nodein}, and is a monodromy
calculation.

We now show that the $p_1$-coordinate can be extended to the whole space $X$. Later the `level' of a singular point is defined as the $p_1$-coordinate;
the levels of singular points are part of the symplectic Lefschetz data of a fibration.
The $p_1$-coordinate is extended via parallel transport on the symplectic foliated fibration $X \to B$ using the two-form  $\om \in \Om^2(X)$.  The form
$\om$ induces a splitting of $T_xX$ for any $x \in X$ as
\begin{equation}\label{eq:cylconn}
  T_xX = TF_x \oplus H_x, \quad H_x:=\{\xi \in T_xM: \om(\xi,\tau)=0 \enspace\forall \tau \in TF_x\}.
\end{equation}
On a path $\gamma:[0,1] \to B$ not passing through singular values of
the fibration, the resulting parallel transport map
\begin{equation}
  \label{eq:par-transp}
  \Phi_t^\gamma: F_{\gamma(0)} \to F_{\gamma(t)}
\end{equation}
is a symplectomorphism for any $t$ if $\gamma$ lies on a leaf. If 
$\om$ is a strong symplectic form, $\Phi_t^\gamma$ is a symplectomorphism for any path $\gamma$ in $B$.

\begin{lemma}\label{lem:volinmid}
  Suppose $b \in B$ and $F_b$ is a fiber of $X \to B$ for which the
  $p_1$ coordinate from \eqref{eq:ends} is well-defined in $\{p_1>a_+\} \cup \{p_1<a_-\}$
  for some constants $a_-\leq a_+$. Then, the compact region of $F_b$
  between the levels $a_+$, $a_-$ has volume $a_+-a_-$.
\end{lemma}
\begin{proof}
  Take a point $b_0 \in B$ in the same leaf as $b$ for which the fiber
  $F_{b_0} \subset X\bs K$ and thus has the standard form, and the
  $p_1$ coordinate is well-defined for all of $F_{b_0}$. If $F_b$ is a
  regular fiber, the lemma is proved by parallel transporting along a
  path from $b_0$ to $b$ not passing through a singular
  point. Parallel transport is also well-defined on regions of the
  singular fiber away from the singular point, and so the lemma can be
  proved in a similar way.
\end{proof}

Definition \ref{def:folTLT} implies that singular points form 
transverse one-dimensional manifolds in $X$. By \eqref{eq:wind},  
each open leaf in
$X$ has a finite number of singular points, and therefore, the set of
singular points is a collection of embedded transverse loops in $X$.

We now define the {\em level} of a singular point of the foliated
Lefschetz fibration as its $p_1$-coordinate. In the case when the
foliated fibration has a strong symplectic form, the level is constant
on any connected component of singular points.

\begin{definition}
  {\rm(Level of a singular point)} Suppose $x$ is a singular point of
  the Lefschetz fibration $\pi:X \to B$, and $b:=\pi(x)$. Assume $x$
  is the only singular point in the fiber $F_b$, and so, the fiber
  consists of two disks -- an upper disk $F^u_b$ and a lower disk
  $F^l_b$. Suppose the $p_1$ co-ordinate is well-defined on $F_b$ for
  $p_1>a_+$ and $p_1<a_-$. Then, the {\em level} of the singular point
  $x$ is
  \begin{equation*}
    \level(x):=a_+ - \int_{F^u_b \cap \{p_1 \leq a_+\}}\om. 
  \end{equation*}
  As a consequence of Lemma \ref{lem:volinmid}, the level of $x$ can
  equivalently be defined as
  \begin{equation*}
    \level (x):=a_- + \int_{F^l_b \cap \{p_1 \geq a_-\}}\om. 
  \end{equation*}
  In case of multiple singular points on a fiber, the definition of
  level can be extended in an obvious way.
\end{definition}

\begin{remark}
  By Lemma \ref{lem:volinmid}
  there is an extension of the 
  definition of $p_1$ to
  \begin{equation}
    \label{eq:p1x}
    p_1:X \to \R
  \end{equation}
  such that for any fiber $F_b$, and constants $b_- \leq b_+$,
  \begin{equation*}
    \int_{F_b \cap \{b_-\leq p_1 \leq b_+\}}\om_F=b_+ - b_-,
  \end{equation*}
  and on a singular fiber, the level sets of $p_1$ are connected. As a
  result, on a singular fiber, a level set of $p_1$ is either a circle
  contained in one of the irreducible components or a singular
  point. The $p_1$-coordinate of a singular point is equal to its
  level.  The extension of $p_1$ is uniquely defined only up till
  Hamiltonian diffeomorphism of the fiber.
\end{remark}

{\bf In the rest of the section, we assume that the foliated filling has a strong symplectic form. The case of the weak symplectic form is discussed in Remark \ref{rem:weak-data}.}

The next lemma shows that parallel transport preserves the
$p_1$-coordinate up to Hamiltonian diffeomorphism of the fiber.
\begin{lemma}\label{lem:levelset}
  Suppose $X \to B$ is a strong symplectic foliated Lefschetz fibration,
  and $F_{b_0}$, $F_{b_1}$ are regular fibers of $X \to B$ over
  $b_0, b_1 \in B$. Let $\Phi^\gamma:F_{b_0} \to F_{b_1}$ be the
  parallel transport map along a path $\gamma$ in $B$ from $b_0$ to
  $b_1$.  Then, there is a Hamiltonian diffeomorphism of $F_{b_1}$
  that maps the loop $\Phi^\gamma(\{p_1=a\} \cap F_{b_0})$ to
  $\{p_1=a\} \cap F_{b_1}$.
\end{lemma}
\begin{proof}[Proof of Lemma \ref{lem:levelset}]
  For the proof of the lemma, we define an offset function along a
  path, which measures the amount by which parallel transport changes
  the $p_1$ coordinate. Suppose $\gamma :[0,1] \to B$ is a path of
  regular values, with a parallel transport map
  $ \Phi_t^\gamma:F_{\gamma(0)} \to F_{\gamma(t)}$ (as in \eqref{eq:par-transp}).
  The {\em offset}
  function $s^\gamma:[0,1] \to \R$ is defined by the following
  condition: for any $a \in \R$ the loop $\Phi_t^\gamma(\{p_1 =a \})$ is
  Hamilton isotopic to $\{p_1=a+s^\gamma(t)\} \subset F_{\gamma(t)}$.
  The value of $s^\gamma(t)$ does not depend on $a$ because parallel
  transport preserves volume on the fiber.

  If the path $\gamma$ lies on a leaf, then $s^\gamma\equiv 0$. Indeed,
  the symplectic form on the leaf is standard on the cylindrical ends,
  and parallel transport by the standard symplectic form preserves the
  $p_1$-coordinate in the ends. This is not true if $\gamma$ does not
  lie on a leaf because the identification $i$ (from \eqref{eq:ends})
   on the non-compact ends
  is only a foliated symplectomorphism, and not an equivalence of
  strong symplectic two-forms.

  The offset function is preserved by foliated homotopy of paths:  If two
  paths $\gamma_0$ and $\gamma_1$ in $B$ are related by foliated
    homotopy (as in Definition \ref{def:holonomy} \eqref{part:hol1}), then $s^{\gamma_0} \equiv s^{\gamma_1}$
  because the parallel transport map induced by a contractible loop is
  a Hamiltonian isotopy, see Theorem 6.21 in \cite{MS:intro}, and the
  offset function is constant for a path lying on a leaf.

  We now prove the lemma by showing that the offset function vanishes
  for any path.  
  Consider a path $\gamma:[0,1] \to \R \times A^2$ such that
  $\gamma(0)$ is in the outer boundary leaf $\R \times \partial_+A^2$
  and the rest of the path maps to leaves in the interior of
  $\R \times A^2$.  There is a loop
  $\delta: \bS^1 \to \R \times \partial_+ A^2$ with non-trivial
  holonomy. As a result, for any small $\eps_0>0$, the path $\gamma$
  is foliated homotopy equivalent to a path $\gamma':[0,1] \to A^2$,
  such that $\gamma'$ is a reparametrization of $\gamma|[0,\eps]$ for
  some $\eps<\eps_0$. Indeed, in the foliated homotopy from $\gamma$
  to $\gamma'$, the starting point $\gamma(0)$ goes around the loop
  $\delta$ sufficiently many times.  Since the offset function is
  continuous $s^\gamma=s^{\gamma'}$, we conclude that
  $s^\gamma \equiv 0$. Any path in $\R \times A^2$ that does not intersect the
  inner boundary leaf $\R \times \partial_-A^2$ is foliated homotopy
  equivalent to a subset of the path $\gamma$ or $-\gamma$. Therefore,
  the offset function vanishes along any path, concluding the proof of
  the lemma.
\end{proof}

\begin{lemma}
  In a strong symplectic foliated Lefschetz fibration $X \to B$, the level is
  constant on any connected component of singular points.
\end{lemma}
\begin{proof}
  Let $\gamma:[0,1] \to B$ be a path of singular values of the
  Lefschetz fibration. We assume that the fibers $F_{\gamma(t)}$ has
  one singular point, the proof of the general case is similar.  Each
  singular fiber consists of two disks. Let $F^u_{\gamma(t)}$ be the
  upper disk. Then, $F^u_\gamma:=\cup_tF^u_{\gamma(t)}$ is a disk
  bundle over $[0,1]$.  Suppose
  $\Phi_t:F^u_{\gamma(0)} \to F^u_{\gamma(t)}$ is the parallel
  transport map in $F^u_\gamma$ in the direction
  $\ker(\om|_{F^u_\gamma})$.  Suppose the level set $\{p_1=a\}$ is a
  loop in each of the disks $F^u_{\gamma(t)}$, which we denote by
  $\mu_{t,a}$.  Then, by Lemma \ref{lem:levelset}, $\Phi_t(\mu_{0,a})$
  is Hamilton-equivalent to the loop $\mu_{t,a}$ in the fiber
  $F_{\gamma(t)}$.  So, both loops enclose an equal area in the disk
  $F_{\gamma(t)}$. Since, parallel transport preserves volume in the
  fiber, we can conclude that the quantity
  $\int_{F^u_{\gamma(t)}\cap \{p_1 \leq a\}}\om$ is
  $t$-independent. So, $\level(\gamma(t))$ is $t$-independent.
\end{proof}

We now describe the data associated to a strong symplectic foliated Lefschetz
fibration $X \to B$ with a standard structure on ends.
%
To each boundary leaf $\pi^{-1}(\R \times \partial_\pm A)$, we
associate an integer $k_\pm$, which is the Dehn twist of the fiber
along a line $\{c\} \times \R$, $c \in \partial_\pm A^2$. This quantity
is indeed meaningful, because via the map $i$, the fibers at the ends
of the line $\{c\} \times \R$ have fixed identifications to $F$.

Suppose $\Gamma$ is the set of connected components of singular points
in the fibration $X \to B$. Each component $C_j$, $j \in \Gamma$,
projects to a loop $\gamma_j$ in the base $\R \times A^2$. The loop
$\gamma_j$ does not have self-intersections because the level is
constant on the singular value component $C_j$.  Let $\wind(\gamma_j)$
be the winding number of $\gamma_j$ in $\R \times A^2$. Then, by a
holonomy calculation,
\begin{equation}\label{eq:wind}
  k_+ - k_-= \sum_{j \in \Gamma}\wind(\gamma_j). 
\end{equation}

\begin{definition}
  {\rm(Symplectic Lefschetz datum)}
\label{def:sld}
  The {\em symplectic Lefschetz
    datum} for a foliated strong symplectic  Lefschetz fibration $X \to B$ that is standard on
  ends consists of
  \begin{enumerate}
  \item a collection of embedded loops
    $\{\gamma_j:\bS^1 \to B\}_{j \in \Gamma}$
    of singular values 
    transverse to the
    foliation, that are indexed by a finite set $\Gamma$;
  \item Dehn twists $k_\pm \in \Z$ on the boundary leaves satisfying
    \eqref{eq:wind};
  \item $\level_{\gamma_j} \in \R$ for each loop $j \in \Gamma$. If
    two loops $\gamma_{j_1}$, $\gamma_{j_2}$ intersect in
    $\R \times A^2$, then, their levels cannot coincide :
    $\level_{\gamma_{j_1}} \neq \level_{\gamma_{j_2}}$.
  \end{enumerate}
\end{definition}

Two embedded loops $\gamma_0, \gamma_1 :\bS^1 \to \R \times A^2$ have
the same {\em braid type} if they are connected by a homotopy of
embedded loops. Two loops with the same braid type necessarily have
the same winding number.

\begin{definition}
  {\rm(Combinatorial type)} The combinatorial type of a strong symplectic foliated Lefschetz
  fibration $\pi:X \to B$ is
  \begin{enumerate}
  \item the braid type of each component of singular points, and
  \item the Dehn twists $k_\pm$ on the boundary leaves.
  \end{enumerate}
\end{definition}

\begin{remark}\label{rem:link}
  The linking between the singular value loops is not preserved in a
  homotopy of symplectic Lefschetz data. However, the braid type of a
  single loop is preserved. This is because the level is constant on a
  single loop of singular values, and so, the loop is embedded in the
  base space $\R \times A^2$. However, two different components of
  singular values can intersect in the base space if the levels are
  different on the two components.
\end{remark}

\begin{proposition}\label{prop:homo-sympdata}
  Suppose $\Phi_0$, $\Phi_1$ are symplectic Lefschetz data for
  fibrations on $\R \times (A^2,\F_{ah})$  whose regular fibers are
  cylinders. Suppose $\Phi_0$, $\Phi_1$ have the same combinatorial
  type. Then, there is a $[0,1]$-family $\Phi_t$ of symplectic
  Lefschetz data that connects $\Phi_0$ to $\Phi_1$.
\end{proposition}
\begin{proof}
  Suppose, for $k=0,1$, $\Phi_k$ consists of the collection of
  embedded loops $\{\gamma_j^k\}_{j \in \Gamma}$. For each $j$,
  $\gamma_j^0$ is homotopic to $\gamma_j^1$ via a family of embedded
  loops $\{\gamma_j^t\}_{t \in [0,1]}$.  Define the symplectic
  Lefschetz data $\Phi_t$ as consisting of the loops
  $\{\gamma_j^t\}_{j \in \Gamma}$ and level data $\level_j^t$ defined
  such that if $\gamma^t_{j_1}$ intersects $\gamma^t_{j_2}$, then,
  $\level_{j_1}^t \neq \level_{j_2}^t$. 
\end{proof}

\begin{proposition} {\rm(From symplectic Lefschetz data to smooth fibrations)}
  \label{prop:sm}
  \begin{enumerate}
  \item {\rm(Existence)}
    \label{part:sm1}
    Corresponding to a symplectic Lefschetz datum $\Phi$, there exists
    a smooth Lefschetz fibration $\pi: X\to B$ that has a standard
    identification \eqref{eq:ends} on its ends, and a coordinate
    $p_1: X \to \R$ that extends the $p_1$ coordinate on the ends, and 
    whose restriction $p_1|F_b$ to any fiber $F_b$, $b \in B$
    has critical points
   only at singular points of the Lefschetz fibration. 
  \item {\rm(Uniqueness)}
    \label{part:sm2}
    For $k=0,1$, suppose $\pi_k : X \to B$ is a smooth Lefschetz
    fibration corresponding to a symplectic Lefschetz datum $\Phi$,
    with an identification $i_k$ on ends and a level coordinate
    $p_{1,k} : X_k \to \R$. Then, there is a diffeomorphism
    $\phi : X_0 \to X_1$ satisfying $\pi_0=\pi_1 \circ \phi$,
    $p_0=p_1 \circ \phi$ on $X_0$, and $i_0=i_1 \circ \phi$ in the
    complement of a compact set of $X_0$.
  \item {\rm(A path of data)}
    \label{part:sm3}
    Suppose $\{\Phi_t\}_{t \in [0,1]}$ is a family of symplectic
    Lefschetz data.  Then, there is a corresponding family of smooth
    Lefschetz fibrations $\pi_t : X_t \to B$, each with an
    identification $i_t$ on ends and a level coordinate
    $p_{1,t} : X_t \to \R$ as in \eqref{part:sm1}. 
  \end{enumerate}
\end{proposition}
\begin{proof}
  For the proof of \eqref{part:sm1}, we delete a compact set from a trivial $F$-bundle to obtain 
  \[((\R \times A^2) \times F)\bs K_0,\]
  where $p_1$ is the $\R$-coordinate on the fiber $F$, $p_2$ is the $\R$-coordinate on the base $\R \times A^2$,
  and $K_0:=\{|p_1|, |p_2| \leq R\}$, for a large $R$. On the boundary
  $\R \times \partial_\pm A^2$, we glue in the bundle
  $[-R,R] \times \bS^1 \times F$, where the gluing is trivial along
  $\{-R\} \times \bS^1$, and with $k_\pm$-Dehn twist at
  $\{R\} \times \bS^1$. The side boundaries can be thickened so that
  the bundle is now defined on
  $(\R \times A^2)\bs (\bS^1 \times \mO)$. Here $\mO \subset \R^2$ is
  a large ball, $\bS^1 \times \mO$ has the product foliation, it does
  not intersect the boundary of $\R \times A^2$, and it contains all
  the singular points of the Lefschetz data. On any leaf, the
  monodromy of the fibration along $\partial \mO$ is $(k_+-k_-)$,
  which is equal to $k$, and therefore one may glue in a fibration
  with $k$ singular points with the prescribed values on the base
  given by $\gamma_j$, and $p_1$-level given by $\level_{\gamma_j}$.

  Conversely, for any two fibrations with the same datum, there is a
  diffeomorphism in a neighborhood of the boundary
  $\R \times \partial_\pm A^2$, since the both fibrations have the
  same Dehn twists. The diffeomorphism extends to the interior in an
  obvious way, since the fibration is locally trivial in the
  complement of singular values, and both $X_0$, $X_1$ have the same
  singular values and levels of the singular points.
  The proof of \eqref{part:sm3} is left to the reader.
\end{proof}

\begin{remark}\label{rem:weak-data}
  {\rm(Weak symplectic foliations)}
  In a weak symplectic foliated Lefschetz fibration, the level is not constant
  on a connected component of singular points. 
  Therefore, for a weak symplectic foliated Lefschetz fibration, the symplectic
  Lefschetz data (compare to Definition \ref{def:sld}) consists of
  immersed loops of singular values
  $\{\gamma_j:\bS^1 \to B\}_{j \in \Gamma}$ transverse to the
  foliation, Dehn twists $k_\pm \in \Z$ on the boundary leaves
  satisfying \eqref{eq:wind}, and level data consisting of a smooth
  map
  \[\level_{\gamma_j} : \bS^1 \to \R \]
  for each loop $j \in \Gamma$. If two loops $\gamma_{j_1}$,
  $\gamma_{j_2}$ intersect in $\R \times A^2$, then their levels
  cannot coincide at the point of intersection. The combinatorial type
  of a weak symplectic foliated Lefschetz fibration is defined before
  Theorem \ref{thm:ah} in the Introduction. With these definitions, it
  is seen in a straightforward way that the analogues of Propositions
  \ref{prop:homo-sympdata} and \ref{prop:sm} hold for weak symplectic
  foliations.
\end{remark}

\subsection{Symplectic forms on Lefschetz fibrations}\label{subsec:modelah}
The question of existence and uniqueness of compatible symplectic
forms on Lefschetz fibrations has been addressed by Gompf
\cite{Gompf_book}. However our situation differs from the standard
treatment firstly because of non-compact fibers. The second point of
difference is that the base manifold consists of an $\bS^1$-family of
open leaves $\R^2$ that accumulate on closed leaves $\R \times \bS^1 $
at the ends; this imposes a certain condition on the non-compact ends
of the open leaves. In this section, we construct a symplectic
Lefschetz fibration corresponding to any given symplectic Lefschetz
data. We also show that Lefschetz fibrations with the same
combinatorial data are symplectic deformation equivalent. We first
prove the results for strong symplectic foliations. The case of weak
symplectic foliations is discussed in Remark \ref{rem:wsf}.

\begin{proposition}{\rm(Existence of symplectic form on Lefschetz
    fibrations)}
  \label{prop:exist}
  Corresponding to any symplectic Lefschetz fibration data, there
  exists a foliated strong symplectic fibration $X$ on
  $B:=(A \times \R,\F_{ah},\om_B)$, with fiber $F$, and which is
  standard on ends (see \eqref{eq:ends}).
\end{proposition}
\begin{proof}
  By Proposition \ref{prop:sm} \eqref{part:sm1}, there is a smooth fibration $\pi : X \to B$ corresponding to the given
  data, that is equipped with a diffeomorphic identification of
  non-compact ends as in \eqref{eq:ends}. 

  We adapt Gompf's method \cite[Theorem 10.2.18]{Gompf_book} to
  construct a foliated strong symplectic form. In that method, one needs a
  closed two-form which represents the right cohomology class on
  fibers. In our set-up this closed two-form $\zeta \in \Om^2(X)$ is
  required to satisfy the following conditions:
  \begin{enumerate}
  \item On any fiber, the form $\zeta-i^*\om_0$ vanishes outside a
    compact set, where $\om_0$ is the standard symplectic form on the trivial $F$-bundle on the ends \eqref{eq:ends}.  
  \item For any regular fiber $F$, there is a constant $a_0$ such that
    for all $a \geq a_0$,
    \begin{equation*}
      \int_{F \cap \{|p_1| \leq a\}}\zeta=2a.
    \end{equation*}
  \item\label{part:sing} On a singular fiber $F$ with singularities at
    levels $l_1<\dots<l_k$ and components $F_0,\dots,F_k$ (here $F_i$
    is compact if $1 \leq i \leq k-1$), we require
    \begin{equation*}
      \begin{split}
        \int_{F_i}\zeta = l_i - l_{i-1}, \quad 1 \leq i \leq k-1,\\
        \int_{F_0 \cap \{p_1 \geq -a\}}\zeta=l_1+a, \quad \int_{F_k
          \cap \{p_1 \leq a\}}\zeta=a-l_k,
      \end{split}
    \end{equation*}
    for any sufficiently large $a>0$.
  \end{enumerate}
  Let $F_K \subset F$ be a compact subset such that
  $\R \times A^2 \times (F \bs F_K)$ is contained in
  $(\R \times A^2 \times F)\bs K_0$. Choose $\zeta_F \in \Om^2(F)$
  satisfying the following conditions:
  \begin{enumerate*}
  \item $\zeta_F=0$ on $F_K$,
  \item $\om_F-\zeta_F$ is compactly supported and
    $\int_F \om_F-\zeta_F=0$.
  \end{enumerate*}
  The two-form $\zeta_{pre}:=i^*\pi_2^*\zeta_F$ satisfies the first
  two necessary conditions mentioned above. We add a compactly
  supported exact form to $\zeta_{pre}$ in order to satisfy the third
  condition. We demonstrate this in case $\gamma$ is a loop of
  singular points in $X$, and each of the singular fibers contains
  only one singularity at level $l$. Let $\gamma_B:=\pi(\gamma)$ be
  the projection of the loop, and let
  $\Op(\gamma_B) \subset \R \times A^2$ be a neighbourhood. Let
  $A_l:=B_{2\eps}(l) \bs B_\eps(l) \subset \R$ be a small annulus
  centered at $l$. The fibration
  $\pi^{-1}(\Op(\gamma_B)) \cap \{p_1 \in A_l\}$ is trivial, and is
  diffeomorphic to $\Op(\gamma_B) \times (A_l \times \bS^1)$. There
  exists an exact two-form $d\eta_F$ compactly supported in
  $A_l \times \bS^1$ such that $\zeta_{pre} + d\eta_F$ satisfies the
  condition \eqref{part:sing} for the singular fibers in
  $\gamma$. Next, define a cut-off function
  $\eta_B:\Op(\gamma_B) \to [0,1]$ that is $1$ in a neighbourhood of
  $\gamma_B$. The form $d(\eta_B\eta_F)$ extends by zero to all of
  $X$. By applying similar adjustments for all loops of singular
  points, we obtain the required two-form $\zeta$.

  Next, we describe a covering $\cup_\alpha U_\alpha$ of the base
  $\R \times A^2$, and a symplectic form $\om_\alpha$ on
  $\pi^{-1}(U_\alpha)$.  \setlist[description]{font=\normalfont\space}
  \begin{description}
  \item [Type 1] Two of the sets in the covering are
    $\{ \pm p_2>R\} \subset \R \times A^2$. On these, the symplectic
    form is just $i^*\om_0$.
  \item [Type 2] Amoros et al. (\cite{Amoros}) constructs a symplectic form
    in a small neighbourhood of a singular fiber of a Lefschetz
    fibration. This procedure can be replicated in a family
    (parametrized by an interval) to yield symplectic forms on
    neighbourhoods of singular fibers that agree with $i^*\om_0$ at
    the ends of the fibers.
  \item [Type 3] The rest of the base space can be covered by
    contractible open sets. On these sets the trivialization on the
    ends of the fiber \eqref{eq:ends} can be extended to a
    trivialization of the whole fiber. Thus on each of the sets, we
    can produce a foliated symplectic form agreeing with $i^*\om_0$ on
    the ends.
  \end{description}

  Next, we glue the forms together. For this, we show that in every
  chart $U_\alpha$, there is a compactly supported one-form
  $\eta_\alpha \in \Om^1_c(\pi^{-1}(U_\alpha))$ such that
  $\om_\alpha-\zeta=d\eta_\alpha$ on $\pi^{-1}(U_\alpha)$. Here we
  assume $U_\alpha$ is closed so that the `compact support' condition
  only concerns the fiber direction. Firstly, observe that the forms
  $\om_\alpha-\zeta$ are compactly supported by construction. If
  $U_\alpha$ is of type 1, then, $\om_\alpha-\zeta$ is the pullback of
  the form $\om_F - \zeta_F \in \Om^2_c(F)$. Since this form
  integrates to zero, it is zero in $H^2_c(F)$. Therefore, there is a
  one-form $\eta_F \in \Om^1_c(F)$ such that $d\eta_F=\om_F-\zeta_F$
  which can be pulled back to $\pi^{-1}(U_\alpha)$. If $U_\alpha$ is
  of type 2, we observe that $\pi^{-1}(U_\alpha)$ is contractible. By
  Poincar\'{e} duality $H^2_c(\pi^{-1}(U_\alpha))=0$, and so,
  $\om_\alpha-\zeta$ has a primitive.  If $U_\alpha$ is of type 3, we
  claim that $[\om_\alpha-\zeta]=0$ in
  $H^2_c(\pi^{-1}(U_\alpha)$. This is because the forms integrate to
  zero on the fiber, and since $U_\alpha$ is compact and contractible,
  $H^2_c(F \times U_\alpha) \simeq H^2_c(F) \tensor H^0(U_\alpha)
  \simeq H^2_c(F)$.  With the primitives in hand, we define
$$\om_{pre}:=\zeta+\ssum_\alpha d(\rho_\alpha \eta_\alpha),$$
where $\eta_\alpha:\R \times A^2 \to [0,1]$ is a partition of unity.
This form agrees with $i^*\om$ outside a compact set, and it is
symplectic on fibers. Indeed, it is equal to $\om_\alpha$ on fibers as
$d\rho_\alpha$ vanishes in the fiber direction.

It remains to modify the two-form $\om_{pre}$ to make it
non-degenerate in the base direction. Unlike the compact case, we
cannot add a large multiple of a symplectic form on the base. Instead,
we carry out an {\em inflation} argument, wherein we add a base area
form where needed and adjust it elsewhere so that the form is not
disturbed outside a compact set.  We assume that the compact set $K$
(from \eqref{eq:ends}) is $\{|p_1| \leq R\} \cap \{|p_2| \leq R\}$.
Recall that $\om_B=dq_2 \wedge dp_2$. Let $\eta_1:\R \to [-c_1,1]$ be
a compactly supported function that is $1$ on $[-R,R]$, and
$\int_\R \eta_1(s)ds=0$. The constant $c_1>0$ is small and is to be
determined. Then, there is a compactly supported function $f:\R\to \R$
such that $df=\eta_1(p_1)dp_1$.  Secondly, define a compactly
supported cut-off function $\eta_2:\R \to [0,1]$ that is $1$ on
$[-R,R]$, and $|\eta_2'|<c_2$ where $c_2$ is a small constant that
will be determined later.
Define a form
\begin{equation}
  \label{eq:inflate}
  \om:=\om_{pre}+Cdq_1 \wedge d(\eta_2(p_2) f(p_1))=\om_{pre}+Cdq_1 \wedge (fd\eta_2 + \eta_1(p_1)\eta_2(p_2)dp_1).
\end{equation}
The form $\om-\om_{pre}$ is compactly supported. We now determine the
constants $C$, $c_1$ and $c_2$ so that $\om$ is symplectic.  The
constant $C$ is chosen large enough so that $\om_{pre}+C\om_B$ is
symplectic on the compact set $\{|p_1| \leq R, |p_2|\leq R\}$. This
can be done because $\om_{pre}$ is symplectic in the fiber direction.
The constants $c_1$, $c_2$ are chosen so as not to disturb the
non-degeneracy of the standard symplectic form on $X\bs K$. In
particular, we choose $c_1<\frac 1 {10C}$ to control the last term in
\eqref{eq:inflate}. This fixes $f$. Then, we choose
$c_2<1/(10C\Mod{f}_{L^\infty})$ to control the second to last term in
\eqref{eq:inflate}.
\end{proof}
\begin{remark}\label{rem:existfamily}{\rm(Parametric existence
    result)} For a $[0,1]$-family of symplectic Lefschetz data, there
  is a family $\{X_t\}_{t \in [0,1]}$ of Lefschetz fibrations, where
  each element of the family is standard in the non-compact ends.  The
  proof is the same as the proof of Proposition \ref{prop:exist}.
\end{remark}
\begin{proposition}\label{prop:unique}
  {\rm(Uniqueness of symplectic form on Lefschetz fibrations)} Suppose
  $\pi_k:(X_k,\om_k) \to (\R \times A^2,\F_{ah})$ ($k=0,1$) are simple
  strong symplectic foliated Lefschetz fibrations with fiber $F=\R \times \bS^1$ and
  identification of ends $i_k$, and of the same combinatorial
  type. Then, they are strong symplectic deformation equivalent. That is,
  there is a diffeomorphism $\phi:X_0 \to X_1$ that satisfies
  \begin{enumerate}
  \item $i_1\circ \phi = i_0$ wherever $i_0$ and $i_1$ are defined.
  \item There is a family of strong symplectic forms
    $\{\ol \om_t: t \in [0,1]\}$ on $X_0$ such that $\ol \om_0=\om_0$,
    $\ol \om_1=\phi^*\om_1$, and $\ol \om_t|_\F$ is $t$-independent
    outside a compact subset of $X_0$.
  \end{enumerate}
\end{proposition}
\begin{proof}
  We first prove the result assuming that $X_0$ and $X_1$ have the
  same symplectic Lefschetz data. Then, by Proposition \ref{prop:sm}
  \eqref{part:sm2}, there is a diffeomorphism $\phi:X_0 \to X_1$
  satisfying the first condition in the proposition and
  $\pi_1 \circ \phi=\pi_0$. By a parametric Moser argument
  \cite[Theorem 7.4]{lsg}, we can deform $\phi$, so that it is
  additionally a symplectomorphism on fibers. We remark that to apply
  the Moser argument on singular fibers, we use the fact that the
  levels of the singular values match in $X_0$ and $X_1$. For any
  $t \in [0,1]$, the two-form $(1-t)\om_0 + t\phi^*\om_1$ is
  non-degenerate in the fiber direction, but may not be non-degenerate
  in the leaf tangent space $T\F$. For a large enough constant $C>0$,
  \begin{equation}\label{eq:Cdef}
    (1-t)\om_0 + t\phi^*\om_1 + C\om_B
  \end{equation}
  is strong symplectic for all $t \in [0,1]$.
  In the absence of the requirement that the two-form is standard on the ends, 
  the path of symplectic forms, obtained by taking convex combination at each of the 3 steps
  \[\om_0 \to \om_0 + C \om_B \to \phi^*\om_1 + C \om_B \to \phi^*\om_1\]
  gives a strong symplectic deformation equivalence between $\om_0$ and $\phi^*\om_1$. However, since we need the deformation to be via a path of strong symplectic forms that is standard on the ends, we  use an inflation argument as follows: 
  We choose constants
  $c_1$, $c_2$ and functions $\eta_1$, $\eta_2$, $f:\R \to \R$ as in
  the proof of Proposition \ref{prop:exist} so that
  \begin{equation}\label{eq:path1}
    (1-t)\om_0 + t\phi^*\om_1 + Cdq_1 \wedge d(f(p_1)\eta_2(p_2))
  \end{equation}
  is strong symplectic for all $t \in [0,1]$, and $C$ is from \eqref{eq:Cdef}. 
  Further, we observe that the
  forms
  \begin{equation}\label{eq:path2}
    \om_0 + tdq_1 \wedge d(f(p_1)\eta_2(p_2)), \quad \phi^*\om_1 + tdq_1 \wedge d(f(p_1)\eta_2(p_2))  
  \end{equation}
  are strong symplectic for all $t \in [0,C]$. Thus, we have produced a path
  of strong symplectic forms connecting $\om_0$ and $\phi^*\om_1$.

  Next, we consider the case that $X_0$ and $X_1$ have the same
  combinatorial type, but not the same symplectic Lefschetz data. 
  By Proposition \ref{prop:homo-sympdata}, there
  is a $[0,1]$-family $\Phi_t$ of symplectic Lefschetz data whose
  end-points are the data of $X_0$ and $X_1$.  By the parametric
  version of the existence result (see Remark \ref{rem:existfamily}),
  there is a family $(X_t',\om'_t) \to B$ of symplectic Lefschetz
  fibrations that are standard on ends whose Lefschetz data is
  $\{\Phi_t\}_t$.  There is a family of diffeomorphisms
  $\phi_t':X_0' \to X_t'$ that respect the identification in the
  ends. The family of symplectic forms $(\phi_t')^*\om_t'$ on $X_0'$
  provides a symplectic deformation equivalence between
  $(X_0',\om_0')$ and $(X_1',\om_1')$.  By the previous two
  paragraphs, for $i=0, 1$, the Lefschetz fibrations $(X_i,\om_i)$ and
  $(X_i',\om_i')$ are strong symplectic deformation equivalent. Composing the
  three equivalences, the proposition is proved.
\end{proof}

\begin{proof}
  [Proof of Theorem \ref{thm:ah}] Suppose $W$ is a strong symplectic foliated filling of
  $\bS FT^*(\bS^1 \times (A^2,\F_{ah}))$, and let $W^\infty$ be the
  manifold obtained by attaching infinite ends.  There is a Lefschetz
  fibration $W^\infty \to \R \times (A^2,\F_{ah})$ -- this is a
  consequence of Propositions \ref{prop:folcylend},
  \ref{prop:moduliann}, \ref{prop:nodein} and
  \ref{prop:endstd}. Lefschetz fibrations with a standard structure on
  ends are classified up to symplectic deformation equivalence -- this
  is a consequence of Propositions \ref{prop:exist} and
  \ref{prop:unique}. The case of the weak symplectic foliated filling is proved in a similar way and is discussed below in Remark \ref{rem:wsf}.
\end{proof}

\begin{remark} \label{rem:wsf}
  {\rm(On weak symplectic foliations)} Given a combinatorial datum for a
  weak symplectic foliated Lefschetz fibration, the proof of the
  existence and uniqueness (up to symplectic deformation equivalence)
  of a weak symplectic foliated Lefschetz fibration $\pi:X \to B$
  carries over verbatim from the strong case. Therefore the weak symplectic version
  of Theorem \ref{thm:ah} holds. 
  The only modification
  occurs in the proof of Proposition \ref{prop:exist} in the
  construction of a two-form $\om_\alpha$ on an open
  $\pi^{-1}(U_\alpha)$ containing a 
  family of singular values ($U_\alpha$ is thus a Type 2 open set). The construction of Amoros et
  al. \cite{Amoros} can be carried out in a family, even when the
  level is not constant on a connected component of singular points.
\end{remark}

\section{Fillings in the Reeb case}

In this section we prove Theorem \ref{thm:reeb}, which says that the
only strong symplectic foliated filling of the sphere cotangent bundle of the Reeb
foliation on $\bS^3$ is the foliated disk cotangent bundle.

\subsection{Compact leaf bounds a compact leaf}\label{subsec:cptleaf}

The first step in analyzing a filling of the foliated sphere cotangent
bundle in the Reeb case is to show that the filling of the compact
leaf is compact.  As part of Theorem \ref{thm:cptleaf}, we prove a
broader result wherein we allow the compact leaf in the contact
foliation to be $\bS^3$, $\R\P^3$, $\bT^3$ or any Lens space. Note
that the result also holds for weak symplectic foliated fillings.

The proof of Theorem \ref{thm:cptleaf} requires a monotonicity result,
where we need the following notion of tameness.  Let $(M,g)$ be a
Riemannian manifold, and let $K \subset M$ be a closed subset
containing the boundary $\partial M$. The complement $M \bs K$ is {\em
  tamed} if there is a positive lower bound $\inj$ on the injectivity
radius for points in $M \bs K$, and there is a constant $0<r_0 < \inj$
such that there is a uniform bound on the differential of the
exponential map and its inverse on $B_{r_0}(m)$ for all
$m \in M \bs K$. 

\begin{proposition}
  \label{prop:bddia} {\rm(Bound on the diameter in terms of area
    \cite[Proposition 4.4.1]{Aud:JH})}
  Suppose $(W,\om,J)$ is a symplectic manifold with a compatible
  almost complex structure. Let $K \subset W$ be a closed set
  containing $\partial W$ such that the complement $W \bs K$ with
  metric $\om(\cdot, J \cdot)$ is tamed in the sense defined in
  preceding paragraph.  Then, there is a constant $C>0$ for which the
  following is satisfied: For any compact set $K \subset W$, and any
  compact connected $J$-holomorphic curve $f:\Sig \to W$ whose image
  intersects $K$ and $f(\partial \Sig) \subset K$, the image of $f$ is
  contained in a neighbourhood of $K$ of size $C(1+\int_\Sig f^*\om)$.
\end{proposition}


\begin{proof}[Proof of Theorem \ref{thm:cptleaf}]
  For each of the three-manifolds in this theorem, the filling is
  foliated by holomorphic curves. See Eliashberg
  \cite{Eliash:holdiscs} for $\bS^3$, Hind \cite{Hind:rp3} for
  $\R\P^3$, Hind \cite{Hind:lens} for Lens spaces $L(p,1)$ and Wendl
  \cite{Wendl} for $\bT^3$.  The holomorphic curves in the filling
  intersect the boundary $L_M$, and have a bound on symplectic area by
  Proposition \ref{prop:energyhomology}. The result is proved by
  showing that given an energy bound, the image of the curve is
  contained in a fixed diameter of the boundary. This is a consequence
  of applying the monotonicity result Proposition \ref{prop:bddia} to
  the leaf $L$, with $K$ be a closed tubular neighborhood of $\partial W=M$. 
  The result is indeed applicable on $L$
  because the tameness conditions follow from the compactness of the filling $W$. 
\end{proof}

\subsection{A cohomological invariant for fillings of
  $\bT^3$} \label{subsec:cohinv}

In this section, we prove Theorem \ref{thm:reeb} though the proof of the main technical lemma is postponed to Section \ref{subsec:surgery}. In this section, as part of the proof of Theorem \ref{thm:reeb}, we carry out two preparatory steps,
each of which involves a deformation of the strong symplectic form on the filling. The first of these deformations ensures that the identification of the cylindrical end to the symplectization is an equivalence of two forms, and not just a symplectomorphism on leaves. The second of these deformations ensures that the closed leaf is symplectomorphic to $T^*\bT^2$ with the canonical symplectic form. The second deformation involves changing a cohomological invariant of the filling, which we  explain in the following remark.



\begin{remark} \label{rem:motiv} {\rm(Motivation)} Suppose $W$ is a
  filling of $\bT^3$ and there is a diffeomorphism
  \begin{equation*}
    \psi:(T^*\bT^2,\om_\can) \to (W^\infty,\om) 
  \end{equation*}
  that is standard on the ends, and such that the path
  $\{t\psi^*\om + (1-t)\om_\can:t \in [0,1]\}$ is a symplectic
  deformation equivalence.  Here, $\om_\can$ is the canonical
  symplectic form on $T^*\bT^2$.  The map $\psi$ can not be deformed
  to a symplectomorphism relative to ends if the compactly supported cohomology class
  $[\psi^*\om - \om_\can] \in H^2_c(T^*\bT^2)$ is non-trivial. By the
  Poincar\'{e} lemma,
$$H^2_c(T^*\bT^2,\R) \simeq H^0(\bT^2,\R)\simeq \R.$$
For any $\kappa \in \R$, viewed as a class in $H^2_c(T^*\bT^2,\R)$, we
construct a representative as follows.  Let
$(p_1,p_2): T^*\bT^2 \to \R^2$ be the coordinates on the cotangent
fibers, and let $\beta \in \Om^2_c(\R^2)$ be a compactly supported form
that integrates to $1$. Then,
\begin{equation}
  \label{eq:omk}
  \om_\kappa:=\om_\can + \kappa \beta(p_1,p_2) 
\end{equation}
is a symplectic form on $T^*\bT^2$. There exists $\kappa \in \R$ such
that $[\psi^*\om - \om_\kappa]=0$ in $H^2_c(T^*\bT^2)$, and
$(\psi^*\om - \om_\kappa)$ has a compactly supported
primitive. Furthermore, if the support of $\beta$ is chosen to be in a
region where $\psi^*\om=\om_\can$, then
$\{t\psi^*\om + (1-t)\om_\kappa:t \in [0,1]\}$ is a path of symplectic
forms. Therefore, by Moser's theorem \cite[Theorem 7.4]{lsg}, the map $\psi$ can be deformed
to a symplectomorphism. Thus $\kappa$ is a cohomological invariant of
the filling $W$. This ends the remark. 
\end{remark}

The form $\beta$ defined in the remark above is called a {\em cotangent
  fiber form}: It is a closed two-form on an extended filling $W^\infty$ whose
support is compact and lies in the cylindrical ends, and which
integrates to $1$ on the cotangent fiber.  In general, the cotangent
fiber form can be defined on any extended filling where the coordinates
$(p_1,p_2)$ are well-defined on the ends. For example, for a 
foliated three-manifold $M$ whose foliated cotangent bundle $FT^*M$ is
a trivial $\R^2$-bundle, the cotangent fiber form is
well-defined on any extended strong symplectic foliated filling of the
unit cotangent bundle $\bS(FT^*M)$.

\begin{proof}
  [Proof of Theorem \ref{thm:reeb}] Let $W^\infty$ be an extension of
  a foliated filling $W$ of the sphere cotangent bundle
  $\bS(FT^*\bS^3_{Reeb})$ of the Reeb foliation on the $3$-sphere,
  which means that there is a foliated symplectomorphism
  \begin{equation}
    \label{eq:endsr}
    i: ( [0,\infty) \times \bS(FT^*\bS^3_{Reeb}), \om_{can}) \to (W^\infty \bs W,\om),
  \end{equation}
  %
  We will prove the existence of a strong symplectic foliated
  deformation equivalence
  \[\phi : FT^*\bS^3_{Reeb} \to W^\infty\]
  that is equal to $i$ in the complement of a compact subset.

  First, we modify the strong symplectic form on the extended filling
  $W^\infty$ by a symplectic deformation equivalence, so that the
  identification \eqref{eq:endsr} on the ends is an equivalence of
  two-forms.  We point out that, at the outset, $i^*\om$ equals
  $\om_\can$ only on the leaf tangent space, and the forms are not
  equal. The cohomology class $[i^*\om - \om_\can]$ vanishes since
  $\bS(FT^*\bS^3_{Reeb})$ is diffeomorphic to $\bS^3 \times \bS^1$,
  and therefore, $i^*\om - \om_\can=d\lam$ for a one-form $\lam$ on
  $[0,\infty) \times \bS(FT^*\bS^3_{Reeb})$. Let
  $\eta : [0,\infty) \to [0,1]$ be a cut-off function that is $0$ in
  the neighborhood of $\{0\}$, and $1$ outside a compact set. If the
  derivative of $\eta$ is small enough,
  $\om':=\om - (i^{-1})^*d(\eta \lam) \in \Om^2(W^\infty)$ is a
  symplectic form on $\F$. Consequently
  $\om - \tau (i^{-1})^*d(\eta \lam) \in \Om^2(W^\infty)$ is a strong
  symplectic form for all $\tau \in [0,1]$, and $\om'=i^*\om_\can$ on
  $\{\eta=1\}$.  From now on, we replace $\om$ by $\om'$, so we can
  assume that $i$ preserves the strong symplectic form outside a
  compact subset.

  We will perform another modification of $(W^\infty,\om)$ by a
  symplectic deformation equivalence, and show that for the resulting
  two-form, the extended filling of the compact leaf is
  symplectomorphic to $T^*\bT^2$.  By Theorem \ref{thm:cptleaf}, the
  filling $L \subset W$ of the compact leaf
  $L_M \subset \bS(FT^*\bS^3_{Reeb})$ is compact since
  $L_M=\bS(T^*\bT^2)$.  By Proposition \ref{prop:lutt}, there are
  Luttinger parameters $\sig:=(k_1,k_2) \in \Z^2$ and a diffeomorphism
  \begin{equation}
    \label{eq:tt2}
    \phi: (\Lutt_{(k_1,k_2)}(T^*\bT^2),\om_\sig) \to (L^\infty,\om)   
  \end{equation}
  that is equal to \eqref{eq:endsr} on
  $[R,\infty) \times \bS(T^*\bT^2)$ for some $R$, and such that
  $(1-t)\phi^*\om + t \om_\sig$ is a symplectic form for any
  $t \in [0,1]$.  The Luttinger parameter $\sig$ is $(0,0)$ for the
  following reason: The leaf $L_M$ partitions $\bS(FT^*\bS^3_{Reeb})$
  into two copies $M_1$, $M_2$, each of which is contactomorphic to
  $M:=\bS (FT^*X_{Reeb})$. Applying the reasoning in the proof of
  Proposition \ref{prop:luttah} \eqref{part:kb} on the component
  $M_1$, we can conclude that $k_2=0$. By applying the same reasoning
  on $M_2$, we get $k_1=0$. Therefore, \eqref{eq:tt2} can be rewritten
  as
  \begin{equation}
    \label{eq:tt2-re}
    \phi: (T^*\bT^2,\om_\can) \to (L^\infty,\om)   
  \end{equation}
  By Remark \ref{rem:motiv}, there is cotangent fiber form
  $\beta \in \Om^2_c(W^\infty)$ whose support is compact and contained
  in $[R,\infty) \times \bS(FT^*\bS^3_{Reeb})$, and a constant
  $\kappa \in \R$, such that $(\om - \kappa \beta)$ is a strong
  symplectic form on $W^\infty$ and
  $[\phi^*(\om-\kappa \beta) - \om_\can]=0$ in $H^2_c(T^*\bT^2)$.  We
  also have that for any $t \in [0,1]$,
  $(1-t)\phi^*(\om-\kappa \beta) + t\om_\can$ is a symplectic form on
  $T^*\bT^2$ that is $t$-independent outside a compact subset. By
  Moser's theorem \cite[Theorem 7.4]{lsg},
  $\phi$ can be deformed via an isotopy so that
  \begin{equation}
    \label{eq:tt2-re2}
    \phi: (T^*\bT^2,\om_\can) \to (L^\infty,\om - \kappa \beta)   
  \end{equation}   
  is a symplectomorphism and $\phi \equiv i$ on
  $[R,\infty) \times \bS(T^*\bT^2)$. For the rest of the proof we
  replace $\om$ by $\om - \kappa \beta$ since the latter is a strong
  symplectic form on $W^\infty$ that is deformation equivalent to
  $\om$.

  We now have an identification $i$ (from \eqref{eq:endsr}) on ends
  that is an equivalence of strong symplectic forms, and a
  symplectomorphic extension $\phi$ of $i$ to the closed leaf
  $T^*\bT^2 \subset FT^*\bS^3_{Reeb}$. By Proposition
  \ref{prop:reebexact} stated below, the map $\phi$ extends
  to a foliated symplectomorphism from the foliated cotangent of
  $\bS^3_{Reeb}$ to $W^\infty$, finishing the proof of Theorem
  \ref{thm:reeb}.
\end{proof}

The proof of Theorem \ref{thm:reeb} used the following result which is proved in the next section. 
\begin{proposition}\label{prop:reebexact}
  Suppose $W$ is a strong symplectic foliated filling of $M=\\ \bS(FT^*(\bS^3,\F_{Reeb}))$ with extended filling $(W^\infty,\om)$ and an identification of non-compact ends
  \[i : ([R,\infty) \times M, \om_\can) \to (W^\infty\bs K, \om) \]
  for some $R \geq 0$ and a compact set $K \subset W^\infty$, such that $i^*\om=\om_\can$. 
    Suppose that the compact leaf  $L_M \simeq \bS T^*\bT^2 = \bT^3$  bounds a compact leaf $L$ in $W$ with extension  $L^\infty \subset W^\infty$, and that there is a symplectomorphism
  \begin{equation}\label{eq:cptleafsymp}
  \phi:(T^*\bT^2,\om_\can) \to (L^\infty,\om)
  \end{equation}
  that is equal to $i$ on $L_M \times [R,\infty)$.
  Then, $\phi$ extends to a foliated
  symplectomorphism
  \begin{equation*}
    \phi:(FT^*(\bS^3,\F_{Reeb}),\om_\can) \to W^\infty 
  \end{equation*}
  that extends the identification of the non-compact ends.
\end{proposition}

\subsection{A surgery operation}\label{subsec:surgery}
In this section, we prove 
Proposition \ref{prop:reebexact} which shows the
 uniqueness of strong symplectic fillings in the Reeb case.
First (in Remark \ref{rem:surgeryx})  we describe a surgery operation
 on a foliated $3$-manifold, which when applied to a Reeb
component, transforms it  to a $3$-manifold with almost horizontal foliation. This
surgery is used at the level of foliated cotangent bundles to prove
Proposition \ref{prop:reebexact}. The surgery operation requires the Reeb filling to be a strong symplectic foliation.

\begin{figure} 
  \centering 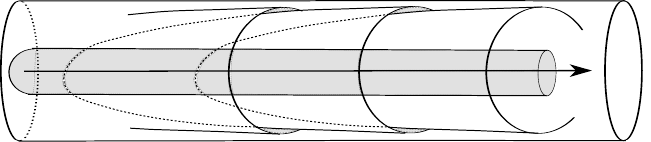
  \caption{The Reeb foliation in a part of the torus, with transversal
    $\gamma$, and a neighbourhood $U_{Reeb}$}
  \label{fig:ureeb}
\end{figure}

\begin{remark} \label{rem:surgeryx} {\rm(From Reeb to almost
    horizontal)} Let $X_{Reeb}:=(\bS^1 \times \bD^2, \F_{Reeb})$ be
  the solid torus with the Reeb foliation. Let
  $\gamma:\bS^1 \to X_{Reeb}$ be a closed transversal that intersects
  every non-compact leaf exactly once. The thickening of the
  transversal, denoted by $U_{Reeb} \simeq \bS^1 \times \bD^2$, is
  equipped with the product foliation, see Figure \ref{fig:ureeb}. Let
  $U_{out}$ be an open set that deformation retracts to
  $X_{Reeb}\bs \ol U_{Reeb}$, and such that
  $U_{Reeb} \cup U_{out}=X_{Reeb}$ as in Figure \ref{fig:surgery}.

  In the almost horizontally foliated manifold
  $X_{ah}:=(A^2 \times \bS^1,\F_{ah})$, the outer boundary
  $\partial_+A \times \bS^1$ has a neighborhood
  $\Op(\partial_+A \times \bS^1) \subset X_{ah}$ that is foliated
  diffeomorphic to $U_{out}$. By an abuse of notation, we call this
  neighborhood $U_{out}$.  The other open set, which is a
  neighbourhood $\Op(\partial_-A \times \bS^1)$ of the inner boundary,
  is denoted by $U_{ah}$. Therefore, $X_{ah}=U_{out} \cup
  U_{ah}$. There is also a diffeomorphism
  \begin{equation*}
    U_\cap:=U_{out} \cap U_{ah} \simeq U_{out} \cap U_{Reeb}.
  \end{equation*}

  Starting from $X_{Reeb}$, the space $X_{ah}$ can be constructed by
  deleting $U_{Reeb}\bs U_\cap$ and gluing in $U_{ah}$, that is,
  \begin{equation*}
    X_{ah}=(X_{Reeb} \bs (U_{Reeb} \bs U_\cap)) \cup_{U_\cap} U_{ah}.
  \end{equation*}
  Such a surgery can be performed at the level of cotangent bundles to
  transform the standard filling of $\bS(FT^*X_{Reeb})$ to a standard
  filling of $\bS(FT^*X_{ah})$. The strong symplectic form on the
  foliated cotangent bundle depends on a choice of splitting of the
  tangent space of the foliated three-manifold into $T\F_X$ and
  $T\F_X^\perp$. In order to perform the surgery on foliated cotangent
  bundles, we need to ensure that the splittings agree on the overlap.
\end{remark}

\begin{figure}
  \begin{center}
    \scalebox{.8}{
\begingroup%
  \makeatletter%
  \providecommand\color[2][]{%
    \errmessage{(Inkscape) Color is used for the text in Inkscape, but the package 'color.sty' is not loaded}%
    \renewcommand\color[2][]{}%
  }%
  \providecommand\transparent[1]{%
    \errmessage{(Inkscape) Transparency is used (non-zero) for the text in Inkscape, but the package 'transparent.sty' is not loaded}%
    \renewcommand\transparent[1]{}%
  }%
  \providecommand\rotatebox[2]{#2}%
  \newcommand*\fsize{\dimexpr\f@size pt\relax}%
  \newcommand*\lineheight[1]{\fontsize{\fsize}{#1\fsize}\selectfont}%
  \ifx\svgwidth\undefined%
    \setlength{\unitlength}{411.38175904bp}%
    \ifx\svgscale\undefined%
      \relax%
    \else%
      \setlength{\unitlength}{\unitlength * \real{\svgscale}}%
    \fi%
  \else%
    \setlength{\unitlength}{\svgwidth}%
  \fi%
  \global\let\svgwidth\undefined%
  \global\let\svgscale\undefined%
  \makeatother%
  \begin{picture}(1,0.46407171)%
    \lineheight{1}%
    \setlength\tabcolsep{0pt}%
    \put(0,0){\includegraphics[width=\unitlength,page=1]{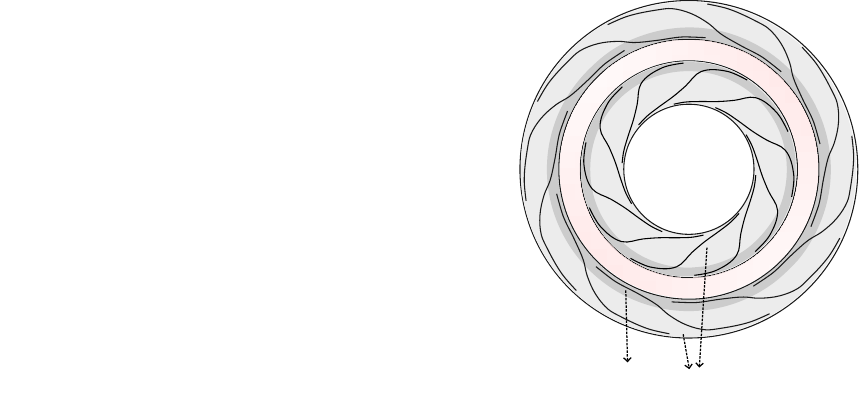}}%
    \put(0.71301011,0.01525577){\makebox(0,0)[lt]{\lineheight{1.25}\smash{\begin{tabular}[t]{l}$U_{ah}$\end{tabular}}}}%
    \put(0,0){\includegraphics[width=\unitlength,page=2]{surgery.pdf}}%
    \put(0.15873174,0.00876264){\makebox(0,0)[lt]{\lineheight{1.25}\smash{\begin{tabular}[t]{l}$U_{reeb}$\end{tabular}}}}%
    \put(0.23089133,0.01348514){\makebox(0,0)[lt]{\lineheight{1.25}\smash{\begin{tabular}[t]{l}$U_{out}$\end{tabular}}}}%
    \put(0.79115951,0.00677532){\makebox(0,0)[lt]{\lineheight{1.25}\smash{\begin{tabular}[t]{l}$U_{out}$\end{tabular}}}}%
    \put(0,0){\includegraphics[width=\unitlength,page=3]{surgery.pdf}}%
  \end{picture}%
\endgroup%
}
  \end{center}
  \caption{Left: A cross-section of a Reeb-foliated solid
    torus. Right: A cross-section of an almost-horizontally foliated
    $A^2 \times S^1$.}
  \label{fig:surgery}
\end{figure}

Using the surgery construction, we prove the following proposition,
which says that a filling of the foliated unit cotangent bundle of
$X_{Reeb}$ is standard.  Since $(\bS^3,\F_{Reeb})$ is made up of two
copies of $X_{Reeb}$, Proposition \ref{prop:reebexact} is then an easy
consequence.

\begin{proposition}\label{prop:reeb}
{\rm(Fillings of Reeb components are standard)} 
Suppose $W$ is a foliated filling of $\bS(FT^*X_{Reeb})$ and
$W^\infty$ is the extended filling with an identification of ends
given by a diffeomorphism
\begin{equation}
  \label{eq:endsreeb}
  \phi:([R,\infty) \times \bS(FT^*X_{Reeb}),\om_\can) \to (W^\infty \bs K, \om)
\end{equation}
for some $R>0$ and a compact subset $K \subset W^\infty$, that
satisfies $\phi^*\om=\om_\can$.  Let $L \subset W$ be the boundary
leaf, and let $L^\infty \subset W^\infty$ be its extension. Suppose
$\phi$ extends to a symplectomorphism
\begin{equation}
  \label{eq:lsymp}
  \phi:(T^*\bT^2,\om_\can) \to L^\infty.
\end{equation}
Then, there is a foliated symplectomorphism
\begin{equation}\label{eq:reebcptsymp}
  \phi:(FT^*X_{Reeb},\om_\can) \to W^\infty
\end{equation}
which extends both \eqref{eq:endsreeb} and \eqref{eq:lsymp}. Here,
$\om_\can$ is a strong symplectic form that is canonical on the leaves
of the foliated cotangent bundle.
\end{proposition}
\begin{proof}
  The proposition is proved by applying the surgery construction of
  Remark \ref{rem:surgeryx} at the level of cotangent bundles.  We
  carry out the proof in two steps -- one of performing the surgery
  and one for going back to the original filling.

  \vskip .1in 
{\sc Step 1}: {\em  There is foliated diffeomorphism
  \begin{equation}
    \label{eq:optt-id}
    \phi_\pre: (\Op(T^*\bT^2), \om_\can) \to (\Op(L^\infty),\om) 
  \end{equation}
  that is equal to $\phi$ in the cylindrical ends \eqref{eq:endsreeb} and the closed leaf \eqref{eq:lsymp} and that   satisfies $\phi_\pre^*\om=\om_\can$.}

We will extend the map in \eqref{eq:endsreeb} and \eqref{eq:lsymp} to
$\Op(T^*\bT^2)$ such that it is a foliated diffeomorphism and maps
integral curves of the line field $\ker(\om_\can)$ to those of
$\ker(\om)$. The latter condition will ensure $\phi_\pre^*\om=\om_\can$
because the flow along $\ker(\om)$ preserves $\om$ and $\phi_\pre$ is a
symplectomorphism on $T^*\bT^2$. In order to prove that the map $\phi_\pre$
can be extended in this manner, it is enough to check that for any
pair of transversals $\tau_0 \subset FT^*X_{Reeb}$ through
$p \in T^*\bT^2$ tangent to $\ker(\om_\can)$ and $\tau_1$ through
$\phi(p) \in L^\infty$ tangent to $\ker(\om)$ there is a map
$\phi : \tau_0 \to \tau_1$ such that the following holds: If, for
$k=0, 1$, $\tau_k$ is related by foliated homotopy to a transversal
$\tau_k'$ lying in the cylindrical end so that
$\phi_\pre: \tau_0' \to \tau_1'$ is defined, then $\phi_\pre$ commutes with the
foliated homotopy. Such a map between transversals can indeed be
uniquely defined because $\phi_\pre$ commutes with holonomy transport
\eqref{eq:def-ht} on $T^*\bT^2$ and $L^\infty$. The last observation
follows from the fact that any loop in $T^*\bT^2$ or $L^\infty$ is
homotopic to a loop in the cylindrical end, where the map $\phi_\pre$ is
already defined.

\vskip .1in 
  {\sc Step 2}: {\em The surgery in Remark \ref{rem:surgeryx} can be performed at the level of cotangent bundles to yield a filling $W_{ah}$ of $\bS(FT^* X_{ah})$.}

  We take the transversal $\gamma \in X_{Reeb}$ close enough to the
  boundary leaf $\bT^2$ so that $FT^*X_{Reeb}|_{\gamma}$ is contained
  in $\Op(T^*\bT^2)$. Then the region in which we perform the surgery
  has the standard foliated symplectic form. By applying Proposition \ref{prop:openstd} to
  the image $\phi_{pre}(FT^*U_{Reeb})$, we assume that
  $\phi_{pre}(FT^*U_{Reeb}) \simeq \bS^1 \times T^*\bD^2$ is a trivial
  symplectic foliation. Consequently, the surgery in Remark
  \ref{rem:surgeryx} can be performed at the level of cotangent
  bundles to produce a foliated symplectic manifold
  \begin{equation}
    \label{eq:wah}
    W_{ah}^\infty:=W^\infty \bs \phi_{pre}(FT^*(U_{Reeb}\bs U_\cap)) \cup_{\bD FT^*U_\cap} \bD FT^*U_{ah}.
  \end{equation}
  By construction, there is a foliated diffeomorphism
  \begin{equation}\label{eq:ahend}
    \phi_{ah}:([R,\infty) \times \bS(FT^*X_{ah}),\om_\can) \to (W^\infty_{ah} \bs W_{ah},\om),
  \end{equation}
  that satisfies $\phi_{ah}^*\om=\om_\can$, 
  where $\om_\can$ is a strong symplectic form that is the canonical
  form on each leaf, and $W_{ah} \subset W^\infty_{ah}$ is a compact set.  Therefore, $W_{ah}$ 
  is a filling of
  $\bS(FT^*X_{ah})$.

  We remark that $W^\infty_{ah}$ has two boundary components -- the
  outer one $\partial_+W_{ah}$ which is also a boundary leaf of
  $W^\infty$, and the inner one $\partial_-W_{ah}$ which is glued in
  by the surgery. Further, $\phi_{ah}$ in \eqref{eq:ahend} extends to
  \begin{equation}
    \label{eq:phiahforms}
    \phi_{ah}: (\Op(\partial_- FT^*X_{ah}),\om_\can) \to (\Op(\partial_-W_{ah}^\infty),\om), 
  \end{equation}
  and in this region, $\phi_{ah}^*\om=\om_\can$.  Note that the
  equality of forms is stronger than foliated symplectomorphism.

\vskip .1in 
  {\sc Step 3}: {\em There is a foliated symplectomorphism
    \begin{equation*}
      \phi_{ah}: (FT^*X_{ah},\om_\can) \to W_{ah}^\infty, 
    \end{equation*}
    that extends $\phi_{ah}$ in \eqref{eq:ahend} and
    \eqref{eq:phiahforms}.}
  
  The boundary leaves $\partial_\pm W_{ah}^\infty$ have zero Dehn
  twist, and so,
by Proposition \ref{prop:unique}, 
  there is a diffeomorphism
  \begin{equation*}
    \tilde \phi_{ah}: (FT^*X_{ah},\om_\can) \to W_{ah}^\infty 
  \end{equation*}
  that extends the identification of the non-compact ends, and such
  that $\tilde \phi_{ah} = \phi_{ah}$ on $\Op(\partial_-FT^*X_{ah})$, and such that $\tilde \phi_{ah}$ is a symplectic deformation equivalence. 
  We will produce a foliated symplectomorphism called $\phi_{ah}$ by deforming 
  $\tilde \phi_{ah}$  relative to $\Op(\partial_- FT^*X_{ah})$, and relative to the cylindrical ends.
  As in the proof of Proposition \ref{prop:unique}, 
  $\tilde \phi_{ah}$ is a symplectic deformation equivalence via three
  paths of symplectic forms
  \begin{equation}\label{eq:paths}
    \om_\can \to \om_\can + \om_{inflate} \to \tilde \phi_{ah}^*\om + \om_{inflate} \to \tilde \phi_{ah}^* \om,
  \end{equation}
  where $\om_{inflate}:=d(f(p_1)\eta_2(p_2))$ is the inflation form
  from the proof of Proposition \ref{prop:unique}.
  We apply the foliated Moser's theorem (Lemma \ref{lem:folmos} \eqref{part:moser-b})
  to each of the arrows in \eqref{eq:paths}. 
  The compactly supported
  primitive $\alpha \in \Om^1_c(FT^*X_{ah})$
  for $\tilde \phi_{ah}^*\om - \om_\can$ (corresponding to the second arrow) 
   vanishes on
  $\Op(\partial_-FT^*X_{ah})$.
%
%
  The
  deformations on $\tilde \phi_{ah}$ produced by Moser's theorem
  applied to the first and third arrows cancel out in
  $\Op(\partial_- FT^*X_{ah})$. Thus the deformation produced by
  Moser's theorem is relative to the region
  $\Op(\partial_- FT^*X_{ah})$. The resulting map is the foliated symplectomorphism $\phi_{ah}$
  on $FT^*X_{ah}$ that extends \eqref{eq:ahend}, 
    \eqref{eq:phiahforms}.

    \vskip .1in 
{\sc Step 4}: {\em The surgery can be reversed to prove the Proposition.} 
  
  We now have a foliated symplectomorphism $\phi_{ah}$ which is
  additionally an equivalence of forms on
  $\Op(\partial_- FT^*X_{ah})$. As a consequence, we can reverse the
  the surgery \eqref{eq:wah}, and we obtain a foliated
  symplectomorphism
  \begin{equation*}
    \tilde \phi_{Reeb}: (FT^*X_{Reeb},\om_\can) \to W^\infty
  \end{equation*}
  that agrees with the identification of ends \eqref{eq:endsreeb}.

  However, the map $\tilde \phi_{Reeb}$ is not equal to the
  symplectomorphism \eqref{eq:lsymp} of the boundary $T^*\bT^2$. On
  the boundary leaf, $\phi^{-1} \tilde \phi_{Reeb}$ is a compactly
  supported symplectomorphism of $T^*\bT^2$. By \cite[Theorem
  5]{Wendl}, the space of compactly supported symplectomorphisms on
  $T^*\bT^2$ is contractible. Further, since $H^1_c(T^*\bT^2)=0$, any
  such map is a Hamiltonian diffeomorphism. The generating Hamilton
  function can be extended to the interior of $FT^*X_{Reeb}$ in a way
  that it vanishes outside a small neighbourhood of the boundary
  leaf. Thus, we obtain a compactly supported foliated Hamilton
  diffeomorphism $\psi$ on $FT^*X_{Reeb}$ that is equal to
  $\phi^{-1}\tilde \phi_{Reeb}$ on the boundary leaf. The Proposition
  is proved by the map $\phi:=\tilde \phi_{Reeb} \psi^{-1}$.
\end{proof}

\begin{proof}[Proof of Proposition \ref{prop:reebexact}]
  We recall that the compact leaf $L_M$ of the unit cotangent bundle
  $\bS(FT^*\bS^3)$ has a compact filling $L$. By the hypothesis of the
  proposition, the extended filling $L^\infty$ is symplectomorphic to
  $T^*\bT^2$ via
  \begin{equation*}
    \phi:(T^*\bT^2, \om_\can) \to (L^\infty ,\om), 
  \end{equation*}
  and $\phi$ agrees with the identification on the ends. The leaf
  $L^\infty$ divides the extended filling $W^\infty$ into two
  components $W^{\infty,\pm}$, each of which is an extended filling of
  $\bS(FT^*X_{Reeb})$. By Proposition \ref{prop:reeb},
  $W^{\infty,\pm}$ are standard fillings, i.e. there are foliated
  symplectomorphisms
  \begin{equation*}
    \phi_\pm:(FT^*X_{Reeb},\om_\can)  \to W^{\infty,\pm}
  \end{equation*}
  that agrees with $\phi$ on $T^*\bT^2$ and on the cylindrical ends.
  The maps $\phi_\pm$ patch to yield an extension of $\phi$, denoted by 
  \begin{equation*}
    \phi_{patch}:(FT^*\bS^3_{Reeb},\om_\can) \to (W^\infty,\om).
  \end{equation*}
  The patched map $\phi_{patch}$ is continuous, but not smooth.  By a
  deformation of $\phi_{patch}$ that is $C^1$-small in leaves, we may
  obtain a foliated diffeomorphism $\tilde \phi$ on $FT^*\bS^3_{Reeb}$
  that is
  \begin{enumerate}
  \item equal to $\phi$ in the closed leaf and cylindrical ends,
  \item an equivalence of two-forms in a small neighborhood $\Op_0(T^*\bT^2)$ ($\phi$ can be extended to such an equivalence of two-forms as in Step 1 of the proof of Proposition \ref{prop:reeb}),
  \item and equal to $\phi_{patch}$ outside a slightly larger neighbourhood $\Op_1(T^*\bT^2)$. 
  \end{enumerate}
  If the deformation is $C^1$-small enough,
  $(1-t)\om_\can + t \tilde \phi^*\om$ is a strong symplectic form for
  all $t \in [0,1]$. By the foliated Moser's theorem (Lemma \ref{lem:folmos}),
  we can correct $\phi_{patch}$
  to yield a smooth foliated symplectomorphism $\phi$ required by the
  Proposition.
\end{proof}
The proof of Proposition \ref{prop:reeb} used the following foliated
version of Moser's theorem. 
%
\begin{lemma}
  {\rm(Foliated Moser stability, \cite[Theorem 2.3]{Hector:Moser}, \cite[Lemma
18]{CPP})}
\label{lem:folmos}
  Suppose $(M,\F)$ is a foliated manifold, whose boundary $\partial M$ is tangent to the foliation $\F$. Suppose $\om_0$, $\om_1$ are strong symplectic forms on $M$ whose difference $(\om_1-\om_0)$ is compactly supported and $[\om_1-\om_0]=0$ in $H^2_c(M)$.
  \begin{enumerate}
  \item \label{part:moser-a} Then there is a flow
    $\{\phi_t\}_{t \in [0,1]}$ on $M$ that is identity outside a
    compact set and that satisfies $\phi_t^*\om_t=\om_0$, where
    $\om_t=(1-t)\om_0 + t\om_1$.
  \item \label{part:moser-b} Additionally, suppose $X \subset M$ is a closed leaf and
    $M \simeq X \times [0,1]$, and $(\om_1-\om_0) \equiv 0$ on
    $X \times [0,\eps]$ for some $\eps>0$. Then the flow
    $\{\phi_t\}_{t \in [0,1]}$ in \eqref{part:moser-a} can be chosen to be identity on $X \times [0,\eps]$.
  \end{enumerate}
\end{lemma}
\begin{proof}
  The proof is exactly as in the unfoliated version of the result
  \cite[Theorem 7.4]{lsg} by defining a compactly supported primitive $\mu$ of $\om_1-\om_0$. The maps $\{\phi_t\}_t$ are
  defined as the flow of the vector fields
  $\{v_t \subset T\F\}_{t \in [0,1]}$ defined by the condition
  \[i_{v_t}\om_t|_\F= \mu|_\F.  \] For \eqref{part:moser-b}, following
  the technique of proof of \cite[Proposition 6.8]{lsg}, the primitive $\mu$ 
  can be chosen so that it vanishes in $X \times [0,\eps]$, and
  therefore $v_t \equiv 0$ on this subset.
\end{proof}
\section{Properties of punctured holomorphic maps}\label{sec:hol-prop}
We extend the proofs of various foundational results of holomorphic
curves in manifolds with cylindrical ends to the setting of foliated
manifolds.
For punctured holomorphic maps to behave reasonably, we require the
contact form on the cylindrical ends to be Morse-Bott as defined below.
Throughout this section, $J$ denotes a cylindrical almost complex structure 
as in Definition \ref{def:Jcyl}. 

\begin{definition} \label{def:mb} {\rm(Morse-Bott)} Let
  $(M,\F,\alpha)$ be a foliated manifold with contact form
  $\alpha \in \Om^1(M)$.  The contact form $\alpha$ is {\em
    Morse-Bott} if the action spectrum
  \[\sig(\alpha):=\{\int_\gamma \alpha : \gamma \text{ is a Reeb orbit} \} \subset \R_{\geq 0}\]
  is discrete,  and 
  for any period $T>0$,
  \[N_T:=\{x \in M : \psi_T(x)=x, \psi_t(x) \neq x\, \forall t \in
    (0,T)\}\]
  is a submanifold of $M$, $\F|N_T$ is a codimension one foliation,
  and $T_x N_T=\ker(d\psi_T(x) - \Id)$.
  
\end{definition}

In the presence of the Morse-Bott condition, punctured holomorphic
curves asymptote to Reeb orbits contained in the same leaf as the map
(Theorem \ref{thm:rem-fol}); and the limit of a converging sequence of
holomorphic curves is a multi-level nodal curve contained in a single
leaf (Theorem \ref{thm:folcpt}).  Without the Morse-Bott condition, at
least one of these conditions may fail as pointed out by Pino-Presas
\cite{Pino:weinstein}.


 \subsection{Fredholm theory and transversality}
 \label{subsec:fred}
 In this section, we introduce a Banach completion for the space of
 asymptotically cylindrical maps with trivial holonomy, and define a
 Cauchy-Riemann operator whose zero set is the moduli space of
 holomorphic maps.
 The moduli space of holomorphic maps in a fixed relative homology class 
 is a foliated manifold if all the maps are unobstructed. 
 We consider maps on a
 punctured Riemann surface $\oSig:=\Sig \bs \{z_1,\dots,z_k\}$ which
 asymptote to Reeb orbits of periods $T_1,\dots,T_k$. Thus, $u$
 represents a relative homology class
 $\beta \in H_2(W, \cup_i N_{T_i}))$.  Let $\M(\beta,J)$ denote the
 moduli space of asymptotically cylindrical $J$-holomorphic maps
 representing the class $\beta$, and which have trivial holonomy.

 We recall that if the image of a map has trivial holonomy in the
 foliation $\F$, then a transverse coordinate can be defined in a
 neighborhood of the image. That is, there is a map
 \begin{equation}
   \label{eq:bt}
   \bt : \Op(\on{im}(u)) \to \R  
 \end{equation}
 on a neighborhood $\Op(\on{im}(u)) \subset W^\infty$ with no critical
 points whose level sets are tangent to the foliation.

 For asymptotically cylindrical maps, there is a natural evaluation
 map at punctures as follows.  Let $\oSig$ be $\P^1$ with punctures
 $z_1,\dots,z_k$.  We fix complex coordinates
 \[(s_i,t_i) : C_i:=\Op(z_i) \bs \{z_i\} \to \R_{\geq 0} \times \bS^1\]
 in the neighborhoods of punctures.  Let $\B(\beta)$ denote the space
 of asymptotically cylindrical smooth (not necessarily holomorphic)
 maps  $u: \oSig \to W^\infty$ of
 class $\beta$.
 We recall the notion of asymptotic cylindricity from \eqref{eq:exp-decay}: For any $u \in \B(\beta)$,
 there is a constant $c$ such that 
 at each puncture $z_i$, 
 \begin{equation}
   \label{eq:expconv}
   d_\F(u(s,t), (a_i+T_is, \gamma_i(T_i t)) \leq c e^{-\delta |s|}
 \end{equation}
 for some $a_i \in \R$ and Reeb orbit $\gamma_i : \R/T_i\Z \to M$.
 The constant $\delta>0$ is from Theorem \ref{thm:rem-fol} and is held fixed throughout this section. Such a constant can indeed be chosen since the maps we consider are asymptotic to Reeb orbits of period $2\pi$. 
 The 
 {\em evaluation at the puncture} $z_i$ is 
 \[\ev_{z_i}(u) :=(a_i,\gamma_i(0)) \in \R \times N_{T_i},\]
 where we recall that $N_{T_i} \subset M$ is the foliated submanifold
 consisting of Reeb orbits of period $T_i$. The image of the
 evaluation map $\ev:=(\ev_{z_i})_i$ lies in
 \begin{equation}
   \label{eq:leafdiag}
   S:=\{(x_1,\dots,x_k) \in \prod_i (\R \times N_{T_i}) : \bt(x_1)=\dots=\bt(x_k)\},  
 \end{equation}
 which is a manifold with a foliation of codimension one.

 We define a Sobolev completion of the space of maps, and show that it
 is a Banach manifold with a foliation of codimension one.  Let $p>2$
 and $0<\lam<\delta$. 
 For 
 any $k \geq 0$, define a weighted Sobolev space 
 \[W^{k,p,\lam}(\R_{\geq 0} \times \bS^1):=\{f(s,t) \in W^{k,p}_{loc}: e^{\lam s}f \in
   W^{k,p}(\R_{\geq 0} \times \bS^1)\}. \]
 The Sobolev completion
 \[\B^{p,\lam}:=\B(\beta)^{p,\lam}\]
 is defined as the set of $W^{1,p}_{loc}$-maps $u: \oSig \to W^\infty$
 such that for any puncture $z_i$, there is a Reeb orbit
 $\gamma_i : \R/T_i \Z \to M$ and a Reeb cylinder
 $\ol \gamma_i(s,t):=(a_i + T_i s, \gamma_i(T_i t))$ such that
 $u(s_i,t_i)=\exp_{\ol \gamma_i}\eta_i$ and
 $\eta_i \in W^{1,p,\lam}(C_i , \ol \gamma_i^*TW^\infty)$.  We will
 show that the space $\B^{p,\lam}$ is a Banach manifold, by describing
 a Banach space structure on a neighborhood of any
 $u \in \B^{p,\lam}$.  A neighbourhood $\Op(\image(u))$ has a transverse coordinate $\bt$ as in \eqref{eq:bt}. In this neighborhood, we choose a metric $g_u=g_\F \oplus d\bt^2$, where $g_\F$ is a metric on the
 leaves of $\F|\Op(\on{im(u)})$, and on the cylindrical end
 $\Op(\on{im(u)}) \cap (\R_{\geq 0} \times M)$,
 $g_\F=da^2 \oplus g_{\F,M}$ and is $\R$-invariant.
 The tangent space
 $T_u\B^{p,\lam}$ is the image of the map
 \begin{align}
   T_{\ev(u)}S \oplus W^{1,p,\lam}(C,u^*T\F) &\to W^{1,p}_{loc}(C,u^*TW^\infty), \\
   \nonumber  ((\xi_{z_i})_i, \xi) &\mapsto \sum_i\beta_i \bT_u^i \txi_{z_i} + \xi,  
 \end{align}
 where $S$ is from \eqref{eq:leafdiag}; $\beta_i:\oSig \to [0,1]$ is a
 cutoff function that is supported in a neighborhood of the puncture
 $z_i$;
 \[\txi_{z_i}(s,t):=(\xi_{z_i,\R}, d\psi_{T_i t}(\xi_{z_i,M}) ) \in
 \Gamma(\R \times \R/\Z, \ol \gamma_i^*T(\R \times M))),\]
%
$\xi_{z_i}=(\xi_{z_i,\R}, \xi_{z_i,M}) \in T_{x_i}(\R \times N_{T_i})$ is a splitting into components; 
and
 $\bT_u^i : \Gamma(C_i, \ol \gamma_i^*T(\R \times M)) \to \Gamma(C_i,
 u^*TW^\infty)$ is parallel transport along geodesics.  The tangent
 space $T_u\B^{p,\lam}$ is identified to a neighborhood of $u$ via the
 exponential map $\xi \mapsto \exp_u\xi$ with respect to the metric $g_u$. Therefore, a neighborhood of
 $u$ is isomorphic to a neighborhood of $(\ev(u),0)$ in
 $S \times W^{1,p,\lam}(\oSig, u^*TW^\infty)$.  The space
 $\B^{p,\lam}$ is thus a foliated Banach manifold, where the foliation is
 pulled back from $S$.

 The moduli space of holomorphic cylinders is cut out as the zero set
 of the Cauchy-Riemann operator.  The non-linear Cauchy-Riemann
 operator is a smooth section $\delbar_J: \B^{p,\lam} \to \E^{p,\lam}$
 of a Banach bundle $\E^{p,\lam} \to \B^{p,\lam}$ whose fiber over a
 point $u \in \B^{p,\lam}$ is
 $\E_u:=L^{p,\lam}(\Om^{0,1}(\oSig, u^*T\F))$. The linearization
 \cite[p27]{wendl:sft} denoted by
 %
 \[  D_u:=d\delbar_J^u: T_u\B^{p,\lam} \to \E^{p,\lam}_u  \]
%
 is a Fredholm operator, and its restriction to the leaf of
$\B^{p,\lam}$ containing $u$ 
 is denoted by 
 \[D_u^\F:=D_u|\B^{p,\lam}_\F, \quad \text{where} \quad \B^{p,\lam}_\F:=\B^{p,\lam} \cap W^{1,p}_{loc}(\Sig,u^*T\F).\]
 We say that the map $u$ is {\em unobstructed} if $D_u$ is
 surjective, and {\em leafwise unobstructed} if $D_u^\F$ is surjective.

 We also define a leafwise Cauchy-Riemann operator and the notion of
 unobstructedness for nodal curves modelled on a fixed tree. Let $\oSig$ be a punctured nodal
 curve with nodes $w_1\dots w_k$, and for each $w_i$, let $w_i^+$,
 $w_i^-$ be lifts of the node. Let $\beta=(\beta_i)_i$ denote the
 relative homology class of each of the components, and we denote by
 $\M(\beta, J)$ the set of $J$-holomorphic asymptotically
 cylindrical maps in the class $\beta$.  For each of the nodal lifts,
 we have an evaluation map
 $\ev_{w_i^\pm} : \M(\beta, J) \to (W^\infty)^2$.
 For $\M(\beta,J)$ to be a smooth manifold in a neighborhood of $u$,
  component-wise moduli
 spaces $\M(\beta_i,J)$ must be transversely cut out at $u$, and the
 evaluation maps $(\ev_{w_i^+}, \ev_{w_i^-})$ of the nodal lifts
 defined on $\prod_i \M(\beta_i,J)$ must be transverse to the diagonal.
 We say that a nodal map 
 $u \in \M(\beta,J)$ is {\em unobstructed} if
 \begin{equation}
   \label{eq:Du-nodal}
 (D_u, (d\ev_{w_i}^+, d\ev_{w_i}^-)_i) : T_u\B^{p,\lam} \to \E^{p,\lam}_u \oplus (\oplus_i(T_{(u(w_i^+),
   u(w_i^-))}(W^\infty)^2/T\Delta))  
 \end{equation}
 is surjective, and {\em leafwise unobstructed} if 
\[(D_u^\F, (d\ev_{w_i}^+, d\ev_{w_i}^-)_i) : T_u\B^{p,\lam}_\F \to \E^{p,\lam}_u \oplus (\oplus_i(T_{(u(w_i^+),
    u(w_i^-))}(L^\infty)^2/T\Delta_L))\]
is surjective, where
$L^\infty \subset W^\infty$ is the leaf containing $u$ and 
$\Delta \subset (W^\infty)^2$, $\Delta_L \subset (L^\infty)^2$ are diagonals.

The next result says that the moduli space is a smooth foliated manifold if all its elements are leafwise unobstructed. 

\begin{proposition}\label{prop:transv}
  Let $J$ be a compatible almost complex structure on $(W^\infty,\F)$
  that is cylindrical on the end, and let $\beta$ be the
  (componentwise) relative homology class of a $J$-holomorphic map
  with trivial holonomy.  The set $\M(\beta,J)$ is a smooth manifold
  of expected dimension with a codimension one foliation if all the
  elements $u \in \M(\beta,J)$ are leafwise unobstructed.
\end{proposition}
\begin{proof}
  In the case of smooth maps, the conclusion follows from Proposition \ref{prop:iftfol}, which is the foliated version of the
  implicit function theorem for Banach spaces. In the case of nodal curves modelled on a fixed tree, leafwise unobstructedness implies obstructedness. Indeed, $T_u \B$ consists of a transverse direction $\bran{\partial_{\bt_i}}$ for each component of $u$, a one-dimensional diagonal subspace $\Delta_{\bt}:=\bran{(\partial_{\bt_1},\dots,\partial_{\bt_k})}$ is in the kernel of \eqref{eq:Du-nodal}, and the quotient
  $\oplus_i \bran{\partial_{\bt_i}}/\Delta_{\bt}$ surjects onto the second summand of \eqref{eq:Du-nodal} since the nodal curve is modelled on a tree. The arguments in Proposition \ref{prop:iftfol} carry over to show that
  $\M(\beta,J)$ is a manifold with a codimension one foliation in the nodal case. 
\end{proof}

\begin{proposition}\label{prop:iftfol}
  {\rm(Foliated implicit function theorem)} Let $X$, $Y$ be Banach
  spaces. Let $U \subset X$ be an open set that has a foliation $\F$
  induced by a smooth transverse coordinate $\bt: U \to \R$. For any
  $k \geq 1$, let $F:U \to Y$ be a $C^k$-map of Banach
  manifolds. Suppose the restricted differential $DF_x|T\F$ is
  surjective for all $x \in F^{-1}(0)$. Then,
  \[M:=F^{-1}(0) \subset X\]
  is a $C^k$-Banach manifold with a foliation induced by $\bt|M$.  For
  any $x \in M$, $T_xM=\ker(DF_x)$, and the leaf tangent space is
  $\ker(DF_x|T\F)$.
\end{proposition}
\begin{proof}
  Since the restricted derivative $DF_x|T\F$ is surjective for all
  $x \in F^{-1}(0)$, the derivative $DF_x$ is also surjective. By the
  implicit function theorem for Banach manifolds \cite[Theorem
  A.3.3]{ms:jh}, we can conclude that $M$ is a $C^k$-Banach manifold,
  and $T_xM=\ker(DF_x)$ for all $x \in M$. The surjectivity of
  $DF_x|T\F$ also implies that $d\bt_x : \ker(DF_x) \to \R$ is
  surjective. Therefore, the restriction $\bt|M$ induces a foliation.
\end{proof}

\subsection{Gromov convergence}
The limit object of a converging sequence of holomorphic curves is a
{\em holomorphic building} lying in a leaf of the foliation, and consisting of curves in
different levels. The definition is the same as the unfoliated version, which we 
 recall 
 from \cite{BEHWZ}.

 Let $\Sig$ be a nodal curve modelled on a tree
$\Gamma$, and let $\Sig^\circ:=\Sig \bs \{z_1,\dots,z_k\}$ be a nodal
curve with punctures, where punctures lie in the complement of
nodes. A vertex $v$ of $\Gamma$ corresponds to an irreducible
component $\Sig_v$.
An edge $e$ of $\Gamma$ corresponds to a node of $\Sig$. Each edge is designated either as a {\em Reeb edge} or as an {\em inner edge}, whose meanings are described below.
 We denote by $\Sig_v^\circ$ the complement of
nodal points corresponding to Reeb edges in $\Sig^\circ \cap \Sig_v$.
A {\em holomorphic building}, denoted by
$u: \Sig \to W^\infty$, consists of a partial ordering $\preccurlyeq$ on the set of
vertices of $\Gamma$, and for each vertex $v$ of $\Gamma$, a holomorphic map $u_v : \Sig_v^\circ \to W^\infty$ or 
$u_v : \Sig_v^\circ \to \R \times M$ as in Section
\ref{subsec:holc}. The maps asymptote to Reeb orbits, both at
punctures $z_i$ and at punctures corresponding to Reeb edges, with
the Reeb orbit at a puncture point $z \in \Sig_v$ denoted by $\Ree_z$.

A {\em matching condition} is satisfied at each node of $\Sig$, which we now describe.
Consider a node $w$ with nodal lifts $w_\pm \in \Sig_{v_\pm}$
corresponding to an inner edge. Then, we require that $\Sig_{v_+} \approx \Sig_{v_-}$ with respect to the partial ordering $\preccurlyeq$,
and $u_+(w_+)=u_-(w_-)$.
The other case is if $w$ corresponds to a Reeb edge, in which case, we require 
$\Sig_{v_-} \prec \Sig_{v_+}$ and the maps $u_{v_+}$, $u_{v_-}$ asymptote to 
the Reeb orbits $\Ree_{w_+}$, $\Ree_{w_-}$
at the punctures $w_+$, $w_-$. The matching condition is that $\Ree_{w_+}$, $\Ree_{w_-}$ 
are
both equal to $\Ree_w : \R/T\Z \to M$ with opposite orientations, which 
means the following: For any choice of coordinates
\[(s,t) : \Op(w_\pm) \bs \{w_\pm\} \to \R_{\geq 0} \times \R/\Z, \]
there is a point $a_\pm \in \R$, and constants $c$, $\delta>0$ such that
\begin{equation}
  \label{eq:opp-orient}
  d_\F(u_{v_\pm}(s,t), (a_\pm \mp Ts, \Ree_w(\mp Tt))) \leq ce^{-\delta s}.  
\end{equation}
Thus the map $u_{v_\pm}$ is asymptotic to the Reeb orbit $\Ree_w$
in the $(\mp \infty)$-end of $\R \times M$. 
For both types of nodes, the matching condition
implies that various components of $u$ lie on the same leaf, and thus
$u$ is a holomorphic building (in an unfoliated sense) in a leaf
$\R \times L$ of $\R \times M$.

Convergence of each component of the limit map is modulo translations.
{\em Translation} by $\tau \in \R$ on $\R \times M$ is denoted by
\[e^\tau : \R \times M \to \R \times M, \quad (a,m) \mapsto
  (a+\tau,m).\]
A translation by  $\tau \geq 0$ is also defined on the cylindrical end $\R_{\geq 0} \times M$ in $W^\infty$. 


Gromov convergence of maps is as in the unfoliated case in \cite{BEHWZ} with
the difference that the convergence of map components is in the
manifold $W^\infty$ and not in the leaves. Therefore, it is
possible for the limit of a sequence of maps in a leaf $L \subset W^\infty$
to be in a leaf $L'$ that lies in the closure of
$L$.

\begin{definition}
  {\rm(Gromov convergence)}
Let $(\Sig_\nu, (z_{1,\nu}, \dots,z_{k,\nu}))$ be a
sequence of Riemann surfaces and let
$\Sig_\nu^\circ:=\Sig \bs \{z_{1,\nu}, \dots,z_{k,\nu}\}$ be punctured
curves.  A sequence of maps $u_\nu : \Sig_\nu^\circ \to W^\infty$
{\em Gromov converges} to a building $u:\Sig^\circ \to W^\infty$ defined on a
nodal punctured Riemann surface $\Sig^\circ$, if for each irreducible
component $\Sig^\circ_v \subset \Sig^\circ$, there is a sequence of
translations $(t_\nu(v))_\nu$ such that the following holds:
\begin{enumerate}
\item {\rm(Convergence of domains)} The surface $\Sig^\circ$ is the
  limit of $\Sig_\nu^\circ$ in the following sense: There are
  sequences of markings $z_{i,\nu}$, $k+1 \leq i \leq n$, and a stable
  limit $(\Sig,(z_1,\dots,z_n))$ of the sequence
  $(\Sig_\nu, (z_{1,\nu}, \dots,z_{n,\nu}))$, such that
  $\Sig^\circ:=\Sig \bs \{z_1,\dots,z_k\}$.
\item {\rm(Translations)} For any vertex $v$, either
  $t_\nu(v) \to \infty$ or $t_\nu(v)=0$ for all $\nu$.
\item {\rm(Convergence of maps)} For any sequence of embeddings of a
  compact set $i_\nu : K \to \Sig_\nu$ that has a limit
  $i: K \to \Sig^\circ_v$ mapping to the complement of nodes in a
  component $\Sig_v \subset \Sig$, the translated sequence of maps
  $e^{\tau_\nu(v)}(u_\nu \circ i_\nu):K \to \R \times M$ converges
  uniformly to $u_v \circ i : K \to W^\infty$ if $t_\nu(v)=0$, and
  $u_v \circ i : K \to \R \times M$ if $t_\nu(v) \to \infty$.
\end{enumerate}
\end{definition}

Gromov compactness applies to a sequence of maps in $W^\infty$ whose Hofer energy is uniformly bounded.  This different notion of energy
is needed because the integral of the leafwise symplectic form
$d(e^t \alpha)$ is infinite for asymptotically cylindrical holomorphic
curves. We define Hofer energy following \cite{BEHWZ}.

\begin{definition}
  {\rm(Hofer Energy)} Let $(M,\F,\xi)$ be a foliated contact manifold,
  and let $\alpha \in \Om^1(M)$ be a contact form.  Let $\seps>0$ be a
  small constant fixed in the rest of the paper, and define a compact
  strong symplectic foliation
  \begin{equation}
    \label{eq:wsymp}
    (W_\cpt,\om_\cpt):=([-\seps,\seps] \times M, d\alpha + d(\pi_\R\alpha)), 
  \end{equation}
  with the foliation on $W_{cpt}$ being pulled back from $M$.  For an
  increasing diffeomorphism $\varphi : \R \to (-\seps,\seps)$, and the
  corresponding map
  $\varphi \times \Id_M : \R \times M \to (-\seps,\seps) \times M$,
  define a pullback symplectic form
  \[\om_\varphi:=(\varphi \times \Id_M)^* \om_\cpt \in \Om^2(\R \times
    M).\]
  Let $\Sig^\circ$ be a punctured Riemann surface.

  \begin{enumerate}
  \item {\rm(For maps on symplectizations)} For a map
    $u: \Sig^\circ \to \R \times M$ whose image lies on a leaf, the
    {\em Hofer energy} is
    \begin{equation}
      \label{eq:ehof-prod}
      E_\Hof(u):= \sup_\varphi \int_{\Sig^\circ} u^*\om_\varphi, 
    \end{equation}
    where the supremum is over all increasing diffeomorphisms
    $\varphi : \R \to (-\seps,\seps)$.  The {\em horizontal energy}
    \footnote{The quantity $\int u^* d\alpha$ is called {\em
        $\om$-energy} and denoted $E_\om(u)$ in \cite{BEHWZ}.}  of $u$
    is
    \begin{equation}
      \label{eq:Ehor}
      E_\hor(u):=\int_{\Sig^\circ} u^*d\alpha.  
    \end{equation}
  \item {\rm(For maps on manifolds with cylindrical ends)} Let
    $W^\infty$ be a strong symplectic foliation with cylindrical ends,
    and let $u: \Sig^\circ \to W^\infty$ be a map to a leaf of
    $W^\infty$. The {\em Hofer energy} of $u$ is
    \[E_\Hof(u):=\int_{u^{-1}(W)} u^*\om + E_\Hof(u|(u^{-1}(W^\infty
      \bs W)), \]
    where the second term is to be interpreted as in
    \eqref{eq:ehof-prod}.
  \end{enumerate}
\end{definition}

\begin{remark}
  The expression for Hofer energy \eqref{eq:ehof-prod} of a map
  $u=(u_\R, u_M) : \Sig^\circ \to \R \times M$ can be rewritten as
  \begin{equation}
    \label{eq:ehof2}
    E_\Hof(u)= \int_{\Sig^\circ} u_M^*d\alpha + \sup_\varphi \int_{\Sig^\circ} d((\varphi \circ u_\R) (u_M^*\alpha)).  
  \end{equation}
\end{remark}

The Gromov compactness theorem will be applied to a sequence of
holomorphic maps in a fixed relative homology class
$\beta \in H_2(W, \cup_i N_{T_i})$, where
$N_{T_i} \subset M=\partial W$ is a Morse-Bott submanifold consisting
of Reeb orbits of period $T_i >0$. Maps in a
fixed relative homology class have a uniform bound on Hofer energy
by the following proposition and remark.

\begin{proposition}\label{prop:energyhomology} 
  For a (strong or weak) symplectic foliation $(W^\infty,J)$ with cylindrical ends, there is a constant $c>0$ such that 
  for any $J$-holomorphic map $u:\oSig \to W^\infty$  whose punctures are asymptotic to the periodic
  orbits $\gamma_1,\dots,\gamma_k$ in $M$,
  \begin{equation} \label{eq:enid} E_\Hof(u) \leq c\left( \int_{u^{-1}(W)} u^*\om
    + \int_{u^{-1}(W^\infty \bs W)}u^*d\alpha + \sum_i\int_{\gamma_i}
    \alpha\right).
  \end{equation}
\end{proposition}
\begin{proof}
  The proof is identical to the non-foliated case \cite[Proposition
  6.3]{BEHWZ}.
\end{proof}

\begin{remark} {\rm(A topological bound on Hofer energy)}
  \label{rem:HEbd}
  The right hand side of \eqref{eq:enid} is a topological invariant as follows. For a $J$-holomorphic curve $u : \oSig \to W^\infty$, define
  \[\hat u: \oSig \to W, \quad z \mapsto
    \begin{cases}
      u(z), &u(z) \in W,\\
      \pi_M(u(z)) \in \partial W, &u(z) \in \R_{\geq 0} \times M.
    \end{cases}
  \]
  The right hand side of \eqref{eq:enid} is equal to $c \int_{\oSig} \hat u^*\om$. If
  $\om$ is a strong symplectic form on $(W,\F)$, then
  $\int_{\oSig} \hat u^*\om$ is an invariant of the relative homology
  class $\beta \in H_2(W, \cup_i N_{T_i})$ represented by $u$. If $\om$ is a weak symplectic form on $(W,\F)$, the quantity $\int_{\oSig} \hat u^*\om$ is an invariant of the homology class $\beta$ if 
   on every leaf of $W^\infty$, $\beta$ is represented by a
  $J$-holomorphic curve lying in the cylindrical end
  $\R_{\geq 0} \times M$. Indeed, on the cylindrical end,
  $\om = d(e^t \alpha)$ is a closed form.
\end{remark}

\begin{theorem}\label{thm:folcpt}
  {\rm(Gromov compactness)}
  Let $(M,\F)$ be a foliated manifold with a Morse-Bott leafwise
  contact form $\alpha$, and let $W^\infty$ be an extended filling that is a weak or strong 
 symplectic foliation. Let $J$ be a tamed cylindrical almost complex structure on $W^\infty$, and let
  $u_\nu: \Sig_\nu \to W^\infty$ be a sequence of $J$-holomorphic
  maps whose Hofer energy is uniformly bounded.  Then a subsequence of
  $(u_\nu)_\nu$ Gromov converges to a holomorphic building.
\end{theorem}
\begin{proof}
  [Outline of proof of Theorem \ref{thm:folcpt}] We outline the proof
  of Theorem \ref{thm:folcpt} following \cite{BEHWZ}.  A punctured
  Riemann surface $(\Sig,j)$ has a unique hyperbolic metric of
  constant curvature $-1$ of finite volume in the conformal class of
  $j$.  The argument of \cite[Section 10]{BEHWZ} shows that after
  adding finitely many sequences $\{z_{k+1,\nu},\dots,z_{n,\nu}\}$ of
  marked points to the domains $\Sig_\nu$, we may assume that the
  domain $\Sig_\nu$ converges to a stable limit $\Sig$ such that
  \begin{itemize}
  \item the first derivative $\sup_\nu |du_\nu(z)|^{hyp}/r_\nu(z)$ is
    bounded with respect to the hyperbolic metric on
    $\Sig^\nu \bs \{z_{1,\nu},\dots,z_{n,\nu}\}$ and $r_\nu(z)$ is the
    injectivity radius at the point $z$ (see \cite[Section
    10.2.1]{BEHWZ});
  \item and for a sequence of long cylinders
    \begin{equation}
      \label{eq:Anu}
      A_\nu:=[-\tfrac{l_\nu} 2, \tfrac {l_\nu} 2] \times \bS^1 \subset \Sig_\nu, \quad l_\nu \to \infty
    \end{equation}
    that converge to a node in $\Sig$, the maps $u_\nu|A_\nu$ satisfy
    a {\em bubble-free} condition (see \cite[Section 10.2.3]{BEHWZ}),
    namely, that if for a sequence $s_\nu$ satisfying
    $|s_\nu \pm \frac {l_\nu} 2| \to \infty$, a subsequence of maps
    $u_\nu(\cdot - s_\nu)$ converges in $C^\infty_\loc$ to a limit
    $u: \R \times \bS^1 \to \R \times M$, then $u$ has zero horizontal
    energy, and so, $u$ is a Reeb cylinder.
  \end{itemize}
  The procedure of adding extra marked points can be carried out in
  the foliated setting in an analogous way. A finite number of marked
  points suffice to achieve the required conditions, because there is
  a positive lower bound on the horizontal energy of non-constant maps
  that are not Reeb cylinders, see Proposition \ref{prop:quant} and
  Remark \ref{rem:qsph}.  Then there exists a limiting map
  $u:\Sig^\circ\to \R \times M$ such that $u_\nu$ converges to $u$ in
  compact subsets in the complement of nodes and puncture points.  It
  remains to prove that at the puncture points, the map $u$ either has
  a removable singularity or asymptotes to Reeb cylinders, and the
  matching condition is satisfied at nodes. The former is the content
  of Theorem \ref{thm:rem-fol}, and the matching condition is a
  consequence of Theorem \ref{thm:foliated-thin}.
\end{proof}

\begin{proposition}{\rm(Energy quantization)}
  \label{prop:quant}
  Let $M$ be a compact manifold with a contact foliation $(\F,\alpha)$.  
  \begin{enumerate}
  \item {\rm(Cylinders)} \label{part:quant1} For any $K>0$ there is a constant $\hbar>0$ such that a 
  holomorphic map $u: \R \times \bS^1 \to \R \times M$ that is asymptotic
  to $\gamma_\pm : \bS^1 \to M$
  at the end $\pm \infty$ with
  $\int_{\gamma_+}\alpha \leq K$ is either a Reeb cylinder (that is, $E_\hor(u)=0$)  or
  $E_\hor(u) \geq \hbar$. 
\item {\rm(Planes)}\label{part:quant2}
  There is a constant $\hbar>0$ such that for any non-constant holomorphic map $u: \C \to \R \times M$,  $E_\hor(u) \geq \hbar$. 
  \end{enumerate}
\end{proposition}
\begin{proof}
To prove \eqref{part:quant1}, we observe that the horizontal energy of a cylinder $u$ that
asymptotes to the Reeb orbit $\gamma_\pm$ at $\pm \infty$ is
\begin{equation}
  \label{eq:Ehor-rel}
  E_\hor(u)=\int_{\R \times \bS^1} u^*d\alpha =   \int_{\gamma_+}\alpha - \int_{\gamma_-}\alpha.
\end{equation}
The lower bound on $E_\hor(u)$ is equal to the minimum difference
between two elements in the action spectrum of $\alpha$ both of which
are $\leq K$, which is a positive quantity by the Morse-Bott
condition. (Note that the proof of the corresponding fact \cite[Lemma
10.9]{BEHWZ} in a stable Hamiltonian manifold is more technical.)  The
proof of \eqref{part:quant2} carries over from the corresponding
result \cite[Lemma 5.11]{BEHWZ} in the unfoliated case.
\end{proof}

\begin{remark}\label{rem:qsph}
  {\rm(Energy quantization for spheres)} In a compact (weak or strong)
  symplectic foliated manifold $(W,\F,\om)$ with a tamed almost
  complex structure, there is a constant $\hbar>0$ that is a lower
  bound for the $\om$-area of $J$-holomorphic spheres. The proof of
  the corresponding result \cite[Proposition 4.1.4]{ms:jh} in the
  unfoliated case carries over, since it only relies on a bound on the
  curvature $R_\om$ of the Levi-Civita connection of the metric
  $\om(\cdot,J,\cdot)$ as part of the proof of the mean value
  inequality \cite[Proposition 4.3.1]{ms:jh}. There is a uniform bound
  on $R_\om$ corresponding to the leafwise metric $\om(\cdot, J\cdot)$ on any weak or strong symplectic foliation.
\end{remark}

\subsection{Asymptotic behaviour at punctures}
In this section, we show that holomorphic maps with finite Hofer
energy are asymptotically cylindrical, assuming the contact form on
the ends is Morse-Bott.  To simplify notation, we consider punctured
holomorphic curves in symplectizations.  We denote any map as
$u=(u_\R, u_M) : \oSig \to \R \times M$. Here $(M,\F,\alpha)$ is a
foliated manifold with contact form $\alpha \in \Om^1(M)$.

\begin{theorem}\label{thm:rem-fol}
  {\rm(Asymptotic convergence, foliated case)} Let $(M,\F,\alpha)$ be a
  Morse-Bott contact foliation of codimension one, and let
  \[u: \R_{\geq 0} \times \R/\Z \to \R \times M\]
  be a $J$-holomorphic map with finite Hofer energy, whose image is
  contained in the leaf $\R \times L \subset \R \times M$.  Then
  either $u$ has a removable singularity at $\infty$ or there exist
  constants $c$, $\delta>0$, $a \in \R$, a period $T \in \R$ and a
  periodic Reeb orbit $\gamma : \R/T\Z \to L$ such that
  %
  \begin{equation}
    \label{eq:asymp-close}
    d_{\cyl,\F}(u(s,t),\ol \gamma(s,t)) \leq c e^{- \delta s},  
  \end{equation}
  where
  \[\ol \gamma : \R \times \R/\Z \to \R \times L, \quad (s,t) \mapsto (a+
    Ts, \gamma(Tt)) \]
  is a Reeb cylinder that lifts $\gamma$, and $d_{\cyl,\F}$ is a product
  metric on $\R \times L$. Furthermore, if the period $T$ is known,
  the constant $\delta$ can be chosen independently of the map $u$. 
\end{theorem}

Theorem \ref{thm:rem-fol} generalizes the following result (Proposition 
\ref{prop:rem-unfol}) in the unfoliated setting. Proposition 
\ref{prop:rem-unfol} is proved in the Morse case in \cite{hwz:nondeg}
and generalized to the Morse-Bott case by Hofer-Wysocki-Zehnder
\cite{hwz:deg} and Bourgeois \cite{bourg:thesis}. There is a simpler
proof in \cite{gv:rem} for the Morse-Bott case, which is easier to
generalize to the foliated setting.

\begin{proposition}\label{prop:rem-unfol}
  {\rm(Asymptotic convergence, unfoliated case)} Suppose
  $(M,\alpha \in \Om^1(M))$ is a compact contact manifold whose Reeb
  orbits form Morse-Bott families, and let $J$ be a cylindrical almost
  complex structure on $\R \times M$. Let
  \[u:\R_{\geq 0} \times \R/\Z \to (\R \times M, J)\]
  be a $J$-holomorphic map with finite Hofer energy.  Then the
  conclusions of Theorem \ref{thm:rem-fol} apply.
\end{proposition}
\begin{proof} We outline the proof following \cite{gv:rem}. We first
  describe a twist of the map $u$.  By \cite[p165]{wendl:sft}, there
  is a sequence $s_k \to \infty$ for which the maps
  $u_M(s_k,\cdot) : \R/\Z \to M$ converge uniformly to a Reeb orbit
  $\gamma_\pre: \R/\Z \to M$ of period $T$.
Define the $(-T)$-twist of $u$ as 
  \begin{equation}
    \label{eq:utw}
    u_\tw: \R_{\geq 0} \times [0,1] \to \R \times M, \quad (s,t) \mapsto
    (u_\R(s,t)-Ts,\psi_{-Tt}(u_M(s,t))).
  \end{equation} 
  The sequence $u_{\tw,M}(s_k,\cdot):[0,1] \to M$ then converges to a constant
  map $m:=\gamma_\pre(0) \in M$.
(We recall that for any $\tau$, $\psi_\tau$ is the time $\tau$ flow of the Reeb vector field on $M$.) 
  For later use, we point out that the discussion in
  \cite[p165]{wendl:sft} implies that there is a sequence of
  translations $\tau_k \in \R$ such that there is a uniform convergence 
  \begin{equation}
    \label{eq:uconv}
    e^{\tau_k}u(s_k, \cdot): \R/\Z \to (\R \times M) \xrightarrow{\text{uniformly converge}}(0,\gamma_\pre) : \R/\Z \to (\R \times M).
  \end{equation}

  Eventually, using a version of the monotonicity lemma, we will show
  that $u_\tw(s,t)$ converges exponentially to a point
  $(a,m) \in (\R \times M)$ for some $a \in \R$, as $s \to \infty$. However, to apply
  the monotonicity lemma we need Lagrangian boundary conditions on the map, which requires us
  to define a ``doubling'' of $u_\tw$.

The {\em double} of $u_\tw$ is defined as 
  \begin{equation}
    \label{eq:vdef}
    v:\R \times [0,\tfrac 1 2] \to (\R \times M)^2 \quad (s,t) \mapsto (u_\tw(s,t),u_\tw(s,1-t)).  
  \end{equation}
  The map $v$ is holomorphic with respect to the $t$-dependent almost
  complex structure
  \[\tilde J := (\tilde J_t)_{t \in [0,\hh]}, \quad \tilde
    J_t:=\psi_t^*J \oplus (-\psi_{1-t}^*J),\]
  and the boundary $\R \times \{0\}$ resp. $\R \times \{\hh\}$ of the
  strip maps to
  \begin{equation}
    \label{eq:DelT}
    \Delta:=\{(x,x): x \in \R \times M\} \quad \text{resp.} \quad \Delta_T:=\{(x,\psi_T(x)): x \in \R \times M\}, 
  \end{equation}
  which is the diagonal resp. the twisted diagonal in
  $(\R \times M)^2$. Both of these are Lagrangian submanifolds with
  respect to the symplectic form $\om \oplus (-\om)$. The Morse-Bott
  condition on $M$ implies that the Lagrangians $\Delta$, $\Delta_T$
  intersect {\em cleanly} in $(\R \times M)^2$, that is,
  $\Delta \cap \Delta_T$ is a submanifold and
  \[T_x(\Delta \cap \Delta_T)=T_x\Delta \cap T_x\Delta_T \quad \forall
    x \in \Delta \cap \Delta_T.\]

  A version of the monotonicity theorem for holomorphic strips
  (Proposition \ref{prop:monot}) allows us to conclude that the image of
  $v$ is contained in a compact set as follows:
  The Hofer energy of $u_\tw$ decays to zero, that is,
  $E_\Hof(u_\tw|[s,\infty) \times [0,1]) \to 0$ as $s \to
  \infty$. Indeed, in the expression \eqref{eq:ehof2}, the
  contribution of $\int u_\tw^* d \alpha$, which is equal to
  $\int u^* d\alpha$ goes to zero in the ends since $E_\Hof(u)$ is
  finite; and the second term in \eqref{eq:ehof2} goes to zero because
  $u_\tw(s_k,\cdot) : [0,1] \to (\R \times M)$ converges to a constant
  map (see \cite[Lemma 11]{gv:rem}).  We choose a large $k$ such that
  the image $u_\tw(\{s_k\} \times [0,1])$ lies in a small enough
  neighborhood of $(a,m)$ for some $a \in \R$ (which is possible by \eqref{eq:uconv}), and
  $E_\Hof(u_\tw,\R_{\geq s_k} \times [0,1])$ is small enough. Applying
  the strip monotonicity theorem to $v$ we conclude that there is not
  enough energy for $v|(\R_{\geq s_k} \times [0,\hh])$ to leave a
  neighborhood of $(a,m)$. (See Step 2 in the proof of \cite[Theorem 1]{gv:rem} for details.)

  Having proved that the image of $v$ is bounded, we may apply results
  for the removal of singularity for pseudoholomorphic strips in
  compact symplectic manifolds, which is a straightforward
  generalization of the removal of singularities in a
  pseudoholomorphic map with a puncture in the interior or boundary of
  the domain.  The exponential decay for the map $u$ follows as a
  consequence. See \cite{gv:rem} for details.
\end{proof}

 \begin{proposition}\label{prop:monot}
   {\rm(Strip monotonicity)} Let $L_0, L_1$ be cleanly intersecting
   Lagrangian submanifolds in a symplectic manifold $(M,\om)$ equipped
   with a Riemannian metric $g$, and let $U_\J$ be a
   $C^0$-neighborhood on the space of tamed almost complex structures.
   Let $p \in L_0 \cap L_1$.  There are constants $R>0$, $c>0$ such
   that
   \begin{enumerate}
   \item for any point $x$ and $r>0$ such that
     $B_r(x) \subset B_R(p)$,
   \item a compact surface with corners $C \subset \R \times [0,1]$,
   \item and a map $u: C \to M$ that is holomorphic with respect to a
     domain-dependent almost complex structure $J:C \to U_\J$, and
     which satisfies the Lagrangian boundary conditions
     \[u(C \cap (\{i\} \times \R)) \subset L_i \quad \text{for $i=0,1$
         and} \quad u(\partial C \bs (\{0,1\} \times \R)) \subset
       \partial B_r(x), \]
     and $x \in u(C)$,
   \end{enumerate}
   the symplectic area of $u$ is bounded as
   \[\int_C u^*\om \geq cr^2. \]
 \end{proposition}

 The proof of Proposition \ref{prop:rem-unfol} generalizes to the
 corresponding result in the foliated setting.
 \begin{proof}[Proof of Theorem \ref{thm:rem-fol}]
   As in the unfoliated case, a preliminary limit Reeb orbit
   $\gamma_\pre$ exists. Furthermore, there is a sequence
   $s_k \to \infty$ and translations $\tau_k$ such that the sequence
   $e^{\tau_k}u(s_k,\cdot)$ converges in $C^\infty_\loc$ to
   $(0,\gamma_\pre) : \R/\Z \to (\R \times M)$.  In the foliated case,
   the convergence is in $(\R \times M)$ and not on a leaf of the
   foliation, and we allow the possibility that the limit orbit
   $\gamma_\pre$ is in a different leaf than the one containing the
   map $u$.  The twisted map $u_\tw$ is defined as in \eqref{eq:utw},
   and the sequence $e^{\tau_k}u_{\tw}(s_k,\cdot)$ of loops converges
   to a point $(0,m) \in \R \times M$, where $m:=\gamma_\pre(0)$. 

   Next we define a doubling $v$ of the twisted map $u_\tw$. 
   The doubled map is defined in a similar way as the unfoliated case
   in \eqref{eq:vdef}.  Let $V \subset M$ be a neighbourhood of $m$ 
   equipped with a transverse coordinate $\bt : V \to \R$ of the
   foliation.  The doubled map will be defined on
   \begin{equation}
     \label{eq:dsub}
     C:=\{(s,t) \in \R \times [0,\tfrac 1 2] : u_\tw(s,t), u_\tw(s,1-t)
     \in \R \times V\}.
   \end{equation}
   Note that $C$ may have multiple components, that are mapped by
   $u_\tw$ to different leaves of $\F|(\R \times V)$.  However for
   large $k$, $u_\tw(s_k, \cdot)$ is contained in $\R \times V$.  Fix
   any increasing diffeomorphism $\varphi : \R \to (-\eps,\eps)$.  Let
   \[\tilde V:= \{(x,y) \in (\R \times V) \times (\R \times V) :
     \bt(x)=\bt(y)\}, \quad \tilde \om:=(\om_\varphi \oplus
     -\om_\varphi)|\tilde V.\]
   Then $(\tilde V, \tilde \om)$ is a foliated symplectic manifold
   with a foliation $\tilde \F$ of codimension one.  We define
   $v: C \to \tilde V$ as in \eqref{eq:vdef}.  Furthermore,
   $\Delta$ and $\Delta_T$ (as in \eqref{eq:DelT}) are foliated
   submanifolds of $\tilde V$, that are a pair of cleanly intersecting
   Lagrangians in each leaf of $\tilde \F$. The strip boundaries
   $\R_{\geq 0} \times \{i\}$, $i=0$ and $i=1$ are mapped by $v$ to
   $\Delta_T$ and $\Delta_0$.

   We now finish the proof. 
   Denote $p:=((0,m),(0,m)) \in \tilde V$. 
   Choose foliated Darboux coordinates on a neighborhood
   $\tilde U \subset \tilde V$ of $p$, which is a diffeomorphism
   \[(\tilde U, \Delta, \Delta_T) \xrightarrow{(\pi,\bt)} (U,L_0,L_1)
     \times (0,\eps),\]
   such that $\bt$ is the transverse coordinate to the foliation, $U$
   is an open subset of $(\R^{2n},\om_{std})$, $L_0$, $L_1$ are
   Lagrangians in $(U, \om_{std})$, and $\pi^*\om_{std}$ is equal to
   $\tilde \om$ on $T\tilde \F$.
   As in the unfoliated case, for a large enough $k$,
   $\pi(e^{\tau_k}v(\{s_k\} \times [0,\hh]))$ is in a small enough neighborhood of $\pi(p)$, and 
   \[\tilde \om(v|(\R_{\geq s_k} \times [0,\texthh])) = E_\Hof(u_\tw|(\R_{\geq s_k} \times [0,1]))\]
   is small enough. An application of the strip monotonicity theorem
   shows that the image of $v$ is bounded in a leaf of $\tilde V$. We
   may then apply results for the removal of singularity for strips in
   compact symplectic manifolds, and show that $v(s,\cdot)$ converges
   to a limit at a rate that is exponential in $s$, from where the
   corresponding estimate \eqref{eq:asymp-close} for $u$ follows. We remark that 
   we also obtain the conclusion that the Reeb cylinder $\ol \gamma$ lies in the same leaf as $u$. 

   Finally, the exponential decay constant $\delta$ can be chosen uniformly
   because the Morse-Bott submanifold $N_T$ consisting of $T$-periodic
   orbits is compact in $X$. The details are the same as the
   unfoliated case in \cite{gv:rem}.
 \end{proof}

\subsection{Behaviour at thin cylinders}
We consider a symplectization of a contact foliation with Morse-Bott
orbits.  The main result of the section is that a sequence of
holomorphic annuli tangent to the foliation with small enough area
converge to a pair of semi-infinite cylinders, both whose infinite
ends are asymptotic to the same Reeb orbit with opposite
orientations (as in \eqref{eq:opp-orient}).

We first state and prove the result 
for an unfoliated contact manifold $(M,\alpha)$.

\begin{proposition}\label{prop:bubconn}
  {\rm(Limit of thin cylinders)} Let $(M,\alpha)$ be a compact contact manifold, and let 
  \begin{equation*}
    u_\nu : [-\tfrac{l_\nu} 2, \tfrac {l_\nu} 2] \times \R/\Z \to \R\times M, \quad l_\nu \to \infty
  \end{equation*}
  be a sequence of pseudoholomorphic annuli that has uniformly bounded
  Hofer energy, uniformly bounded derivatives $|du_\nu|$, and is
  bubble-free (as defined following \eqref{eq:Anu}). Suppose that
  $t_\nu^\pm$ is a sequence of translations such that the maps
  $e^{t_\nu^\pm}u_\nu(\cdot \mp \frac {l_\nu} 2)$ converges on compact
  subsets to $u_\pm : \R_\pm \times \R/\Z \to \R \times M$ that is
  asymptotic to the Reeb cylinder
  \[\ol \Ree_\pm : \R \times \R/\Z \to \R \times M, \quad (s,t)
    \mapsto (T_\pm s, \Ree_\pm(T_\pm t)),\]
  where $\Ree_\pm$ is a Reeb orbit of period $T_\pm$. Then, $\Ree_+$
  and $\Ree_-$ are the same Reeb orbit with opposite orientations (as
  in \eqref{eq:opp-orient}).
\end{proposition}

\begin{proof}[Proof of Proposition \ref{prop:bubconn}] 
  The proof is by describing the Gromov limit of the twisted doubles of the
  pseudoholomorphic annuli.

  We describe the first irreducible component $v_-$ of the limit of the twisted doubles, namely the one containing the $(-\frac {l_\nu} 2)$-end.    Assuming that the
  period of $\Ree_-$ is $T_-$, let 
  \[u_{\nu,\tw}: [-\tfrac {l_\nu} 2, \tfrac {l_\nu} 2] \times [0,1]
    \to \R \times M, \quad (s,t) \mapsto (u_{\nu,\R}(s,t) +
    t_\nu^--T_-(s+ \tfrac {l_\nu}
    2),\psi_{-T_-t}(u_{\nu,M}(s,t))).  \]
  be the $(-T_-)$-twist of $e^{t_\nu^-}u_\nu$, and let
  \[v_\nu :  [-\tfrac {l_\nu} 2, \tfrac {l_\nu} 2] \times [0,\texthh]
    \to (\R \times M)^2, \quad (s,t) \mapsto (u_{\nu,\tw}(s,t), u_{\nu,\tw}(s,1-t))\]
  be the double of $u_{\nu,\tw}$. 
  Similarly, let
  \[u_{-,\tw}:\R_{\geq 0} \times [0,1], \quad (s,t) \mapsto (u_{-,\R}(s,t) -
    T_-s,\psi_{-T_-t}(u_{\nu,M}(s,t)))\]
  be the $(-T_-)$-twist
  of $u_-$, and let 
    \[v_- :  \R_{\geq 0} \times [0,\texthh]
    \to (\R \times M)^2, \quad (s,t) \mapsto (u_{-,\tw}(s,t), u_{-,\tw}(s,1-t))\]
  be the double of $u_{-,\tw}$. 
  Thus, the sequence $v_\nu(\cdot - \frac {l_\nu} 2)$
  converges on compact subsets to
  $v_- : \R_{\geq 0} \times [0,\hh] \to (\R \times M)^2$.  Since $u_-$ is
  asymptotic to the Reeb cylinder $\ol \Ree_-$,
  the twisted strip $u_{\tw,-}$ converges to the point
  \begin{equation}
    \label{eq:Rminus0}
    u_{-,\tw}(\infty)=(0,\Ree_-(0)) \in \R \times M, 
  \end{equation}
  at an exponential rate, 
 and the doubled strip $v_-$ converges to the point
  \[p:=v_-(\infty)=((0,\Ree_-(0)), (0,\Ree_-(0))) \in (\R \times M)^2\]
at an exponential rate.

  Next, we describe the component $v$ of the Gromov limit map attached to $v_-$. (Initially we allow for the possibility that the domain of $v$ is an infinite strip $\R \times [0,\hh]$, but later in the proof, this possibility is ruled out using the bubble-free condition, and the domain of $v$ will be shown to be a semi-infinite strip $\R_{\leq 0} \times [0,\hh]$.) 
  Let $\kappa>0$ be such that the image of $v_-$ lies in
  $B_\kappa(p)$, and consequently, for some $s_\nu'$,
  \[v_\nu([-\tfrac {l_\nu} 2, s_\nu') \times [0,\texthh]) \subset B_\kappa(p),\]
  where we assume that $s_\nu' \leq \frac {l_\nu} 2$ is the maximal
  such value.  After truncating the annuli by a constant amount, we
  may assume that $E_-:=\tilde \om(v_-) < \frac {\hbar} 2$, where
  $\hbar$ is from Lemma \ref{lem:stripconv}.  Let
  $s_\nu'' \in [-\frac {l_\nu} 2, \frac {l_\nu} 2]$ be the largest
  value for which
  \[ \int_{[-\tfrac {l_\nu} 2, s_\nu''] \times
      [0,\tfrac 1 2]} v_\nu^* \tilde \om \leq E_- + \tfrac \hbar 2. \]
  Let $s_\nu:=\min\{s_\nu',s_\nu''\}$.  By
  Lemma \ref{lem:stripconv}, 
  after passing to a subsequence, the Gromov limit of the sequence of
  strips $v_\nu|[-\tfrac {l_\nu} 2, s_\nu] \times [0,\tfrac 1 2]$ is a
  pair of strips $(v_-,v)$, where $v_-$ is the same as above, and $v$
  is the limit $C^\infty_\loc$ limit of the maps
  $v_\nu(\cdot + s_\nu)$ with $v(-\infty)=p$. The map $v$ is
  non-constant because either it is not contained in $B_\kappa(p)$ or
  its $\tilde \om$-area is at least $\hbar/2$ and is therefore positive.

  We introduce the corresponding limit component for the sequence of
  annuli.  There is a sequence of translations $t_\nu$ such that
  $e^{t_\nu}u_\nu(\cdot + s_\nu)$ converges in $C^\infty_\loc$ to a
  limit $u$ in $\R \times M$.  Indeed, this is a consequence of the
  convergence of $v_\nu(\cdot + s_\nu)$, $u$ may be obtained from $v$
  by reversing and the doubling, and
  $t_\nu=t_\nu^--T_-(s_\nu+\frac {l_\nu} 2)$.  Since $v(-\infty)=p$,
  we also conclude that $u$ asymptotes to the orbit $\Ree_-$ at
  $-\infty$.

  Next, we show that the limit of the twisted doubles just has two
  components $v_-$ and $v$.  It is enough to show that
  $\frac {l_\nu} 2 - s_\nu$ is bounded. For the sake of contradiction,
  suppose $\frac {l_\nu} 2 - s_\nu \to \infty$. Then the domain of $v$
  is $\R \times [0,\hh]$, and the domain of $u$ is $\R \times \R/\Z$.
  Since $v$ is non-constant, $u$ is not a Reeb cylinder, which
  contradicts the bubble-free condition.

  We finish the proof. Since $\frac {l_\nu} 2 - s_\nu$ is bounded, the
  map $u$ is just a domain reparametrization of $u_+$. Since
  $v_-(\infty)=v(\infty)$, we conclude $\Ree_+=\Ree_-$.
\end{proof}

The following Lemma was used in the proof of Proposition
\ref{prop:bubconn}. The Lemma can be proved in an analogous manner to
the result on the closed case -- the Gromov convergence of a sequence
of annuli to a pair of disks connected an interior nodal point --
proved in \cite[Section 4.7]{ms:jh}

\begin{lemma}\label{lem:stripconv}
  {\rm(Gromov convergence for low energy strips)} Let $(W,\om)$ be a
  compact symplectic manifold, and let $L_0, L_1 \subset W$ be cleanly
  intersecting Lagrangian submanifolds.  Let
  $\ul J:=\{J_t\}_{t \in [0,1]}$ be a family of $\om$-tame almost
  complex structures.  There is a constant $\hbar>0$ such that the
  following holds. Let
  \[u_\nu : S_\nu:= [-\tfrac {l_\nu} 2, \tfrac {l_\nu} 2] \times [0,1]
    \to W, \quad u(\cdot,i) \subset L_i, \enspace i=0,1\]
  be a family of $\ul J$-holomorphic strips with Lagrangian boundary
  conditions with $\int_{S_\nu}u_\nu^*\om < \hbar$. Then, after
  passing to a subsequence, $u_\nu(\cdot \pm \frac {l_\nu} 2)$
  converges uniformly in compact subsets to
  $u_\pm: \R_\pm \times [0,1] \to W$, the limits $u_+(\infty)$,
  $u_-(-\infty)$ exist and are equal.
\end{lemma}

Proposition \ref{prop:bubconn} on the limit of thin cylinders generalizes to the foliated
setting.
\begin{theorem}\label{thm:foliated-thin}
  Proposition \ref{prop:bubconn} holds if $(M,\alpha)$ is replaced by
  a contact foliation $(M,\F,\alpha)$ whose Reeb orbits are
  Morse-Bott, and $(u_\nu)_\nu$ is a sequence of holomorphic maps
  tangent to the foliation $\F$.
\end{theorem}
\begin{remark}
  In the statement of Theorem \ref{thm:foliated-thin}, the maps
  $e^{t_\nu^\pm}u_\nu(\cdot \pm \frac {l_\nu} 2)$ converge to $u_\pm$
  in the total space of the foliated manifold $W$, and not on the leaf
  space.
\end{remark}
\begin{proof}
  [Proof of Theorem \ref{thm:foliated-thin}] The proof of Proposition
  \ref{prop:bubconn} entirely carries over.
  The new feature in the proof in the foliated setting is that, in order to define the doubling, 
  the images of the relevant maps  need to lie in a neighborhood where a transverse coordinate $\bt$ is defined as in \eqref{eq:dsub}.
  This is indeed the case, because the sequence $u_{\nu,\tw}(\cdot + \frac {l_\nu} 2)$ converges to a map $u_{-,\tw}$ that exponentially converges to the point $(0,\Ree_-(0))$ as $s \to \infty$, 
  and 
 for the proof, it is enough to consider subsets of the domain that
 are mapped by $u_{\nu,\tw}$ to a small neighborhood of $(0,\Ree_-(0))$.


\end{proof}
\bibliographystyle{amsplain} \bibliography{fill}
\end{document}